\documentclass[12pt]{amsart}
\usepackage{amssymb}
\usepackage{amsmath}
\usepackage{amsthm}
\usepackage{mathtools}
\usepackage{calrsfs}
\usepackage{appendix}
\usepackage[all]{xy}
\usepackage{url}
\usepackage[mathscr]{euscript}

\usepackage{tikz-cd}
\usepackage{hyperref}

\newtheorem{theorem}{Theorem}[section]
\newtheorem{lemma}[theorem]{Lemma}
\newtheorem{proposition}[theorem]{Proposition}
\newtheorem{corollary}[theorem]{Corollary}

\theoremstyle{definition}
\newtheorem{definition}[theorem]{Definition}
\newtheorem{example}[theorem]{Example}
\newtheorem{remark}[theorem]{Remark}

\theoremstyle{remark}

\newcommand\bovermat[2]{
	\makebox[-2pt][l]{$\smash{\overbrace{\phantom{
					\begin{bmatrix}#2\end{bmatrix}}}^{\text{#1}}}$}#2}

\newcommand{\FF}{\mathbb{F}}
\newcommand{\ZZ}{\mathbb{Z}}

\newcommand{\TT}{\mathbb{T}}
\newcommand{\GG}{\mathbb{G}}

\newcommand{\LL}{\mathbb{L}}
\newcommand{\CC}{\mathbb{C}}

\newcommand{\bA}{\mathbf{A}}

\newcommand{\bff}{\mathbf{f}}
\newcommand{\bg}{\mathbf{g}}
\newcommand{\bh}{\mathbf{h}}

\newcommand{\bn}{\mathbf{n}}
\newcommand{\bp}{\mathbf{p}}

\newcommand{\bQ}{\mathbf{Q}}
\newcommand{\bs}{\mathbf{s}}

\newcommand{\bG}{\mathbf{G}}
\newcommand{\bu}{\mathbf{u}}
\newcommand{\bv}{\mathbf{v}}

\newcommand{\bT}{\mathbf{T}}

\newcommand{\cE}{\mathcal{E}}
\newcommand{\cG}{\mathcal{G}}

\newcommand{\cN}{\mathcal{N}}

\newcommand{\rA}{\mathrm{A}}
\newcommand{\rd}{\mathrm{d}}
\newcommand{\rB}{\mathrm{B}}
\newcommand{\rF}{\mathrm{F}}
\newcommand{\rG}{\mathrm{G}}
\newcommand{\rH}{\mathrm{H}}
\newcommand{\rL}{\mathrm{L}}
\newcommand{\rP}{\mathrm{P}}

\newcommand{\rR}{\mathrm{R}}
\newcommand{\rS}{\mathrm{S}}

\newcommand{\sC}{\mathscr{C}}
\newcommand{\sH}{\mathscr{H}}

\newcommand{\sP}{\mathscr{P}}
\newcommand{\sT}{\mathscr{T}}

\newcommand{\pd}{\partial}

\newcommand{\bsalpha}{\boldsymbol{\alpha}}

\newcommand{\bsdelta}{\boldsymbol{\delta}}
\newcommand{\bseta}{\boldsymbol{\eta}}
\newcommand{\bsgamma}{\boldsymbol{\gamma}}
\newcommand{\bsomega}{\boldsymbol{\omega}}
\newcommand{\bslambda}{\boldsymbol{\lambda}}
\newcommand{\bse}{\boldsymbol{e}}
\newcommand{\bsh}{\boldsymbol{h}}
\newcommand{\bsell}{\boldsymbol{\ell}}

\newcommand{\bsn}{\boldsymbol{n}}

\newcommand{\bsx}{\boldsymbol{x}}
\newcommand{\bsy}{\boldsymbol{y}}
\newcommand{\bsz}{\boldsymbol{z}}

\DeclareMathOperator{\Der}{Der}

\DeclareMathOperator{\Ext}{Ext}

\DeclareMathOperator{\Ker}{Ker}
\DeclareMathOperator{\GL}{GL}
\DeclareMathOperator{\Mat}{Mat}
\DeclareMathOperator{\Cent}{Cent}

\DeclareMathOperator{\End}{End}

\DeclareMathOperator{\tens}{tens}

\DeclareMathOperator{\Hom}{Hom}

\DeclareMathOperator{\Rep}{Rep}
\DeclareMathOperator{\Res}{Res}

\DeclareMathOperator{\Exp}{Exp}
\DeclareMathOperator{\Id}{Id}

\DeclareMathOperator{\im}{Im}
\DeclareMathOperator{\trdeg}{tr.deg}

\DeclareMathOperator{\Lie}{Lie}

\newcommand{\ok}{\overline{k}}
\newcommand{\oK}{\mkern2.5mu\overline{\mkern-2.5mu K}}

\newcommand{\tr}{\mathrm{tr}}

\newcommand{\power}[2]{{#1 [[ #2 ]]}}

\newcommand{\norm}[1]{\lvert #1 \rvert}
\newcommand{\dnorm}[1]{\lVert #1 \rVert}
\newcommand{\inorm}[1]{{\lvert #1 \rvert}}

\begin{document}
	
	\title[Logarithms of Anderson $t$-modules]{On the algebraic independence of 
logarithms of Anderson $t$-modules}
	
\author{O\u{g}uz Gezm\.{i}\c{s}}
\address{Department of Mathematics, National Tsing Hua University, Hsinchu City 30042, Taiwan R.O.C.}
\email{gezmis@math.nthu.edu.tw}

	\author{Changningphaabi Namoijam}
	\address{Department of Mathematics, Colby College, 5830 Mayflower Hill, Waterville, Maine 04901, U.S.A}
	\email{cnamoijam@gmail.com}
	

	\subjclass[2020]{Primary 11G09, 11J93}
	
	\date{\today}

	\keywords{Drinfeld modules, $t$-modules, transcendence, algebraic independence}
	\begin{abstract} 
 In the present paper, we determine the algebraic relations among the tractable coordinates of logarithms of Anderson $t$-modules constructed by taking the tensor product of Drinfeld modules of rank $r$ defined over the algebraic closure of the rational function field and their $(r-1)$-st exterior powers with the Carlitz tensor powers. Our results, in the case of the tensor powers of the Carlitz module, generalize the work of Chang and Yu on the algebraic independence of polylogarithms.
	\end{abstract}

	\maketitle


\section{Introduction}

Let $\mathbb{F}_q$ be the finite field with $q$ elements where $q$ is a positive power of a prime $p$. We let $\theta$ be a variable over $\mathbb{F}_q$ and set  $A:=\mathbb{F}_q[\theta]$. Let $K:=\mathbb{F}_q(\theta)$ be the fraction field of $A$ and $K_{\infty}:=\mathbb{F}_q((1/\theta))$ be the completion of $K$ corresponding to the norm $\norm{\cdot}$ at the infinite place of $K$ normalized so that $\inorm{\theta}=q$. We further set $\mathbb{C}_{\infty}$ to be the completion of a fixed algebraic closure of $K_{\infty}$ and $\overline{K}$ to be the  algebraic closure of $K$ in $\mathbb{C}_{\infty}$. 

Let $L$ be a field so that $\overline{K}\subseteq L \subseteq \mathbb{C}_{\infty}$. For any matrix  $B=(B_{ik})\in \Mat_{m\times \ell}(L)$ and $j\in \mathbb{Z}_{\geq 0}$, set $B^{(j)}:=(B_{ik}^{q^{j}})\in  \Mat_{m\times \ell}(L)$. For $m,\ell\in \ZZ_{\geq 1}$, we  let  $\Mat_{m\times \ell}(L)[\tau]$ be the set of polynomials of $\tau$ with coefficients in $\Mat_{m\times \ell}(L)$. When $m=\ell$, we define the ring $\Mat_{m}(L)[\tau]$ of polynomials in $\tau$ with coefficients in $\Mat_{m}(L)$ subject to the condition $\tau B=B^{(1)}\tau$ for each $B\in \Mat_{m}(L)$.

Let $t$ be a variable over $\mathbb{C}_{\infty}$ and set $\bA:=\mathbb{F}_q[t]$. In his influential work \cite{Dri74}, Drinfeld introduced  \textit{elliptic modules}, now called Drinfeld modules, to provide a proof for some special cases of the Langlands correspondence for $\GL_2$ over global function fields. In our context, \textit{a Drinfeld $\bA$-module (over $\oK$)} is a tuple $G_0:=(\GG_{a/\oK}, \phi)$ consisting of the additive group scheme $\GG_{a/\oK}$ over $\oK$ and an $\mathbb{F}_q$-algebra homomorphism $\phi:\bA\to \overline{K}[\tau]$ defined by 
\begin{equation}\label{E:DrinfeldDef}
\phi(t):=\theta + a_1\tau +\dots +a_r\tau^r\in \overline{K}[\tau], \ \ a_r\neq 0.
\end{equation}

Later, Anderson \cite{And86} extended the notion of Drinfeld modules to higher dimensional objects, called \textit{$t$-modules (over $L$) of dimension $s$}, which are  tuples $G=(\mathbb{G}_{a/L}^s, \varphi)$ of the $s$-copies of the additive group scheme $\mathbb{G}_{a/L}$ over $L$ and an $\mathbb{F}_q$-algebra homomorphism $\varphi:\bA\to \Mat_s(L)[\tau]$ given by 
\[
\varphi(t)=A_0+A_1\tau+\dots+A_m\tau^m
\]
for some $m\in \ZZ_{\geq 0}$ so that $\rd_{\varphi}(t):=\rd \varphi(t)=A_0=\theta \Id_s+N$ where $\Id_s$ is the $s\times s$ identity matrix and $N$ is a nilpotent matrix. We set $\Lie(G)(L):=\Mat_{s\times 1}(L)$ and equip it with the $\bA$-module structure given by 
\[
t \cdot x:=\rd_{\varphi}(t)x=A_0x, \ \ x\in \Lie(G)(L).
\]
Furthermore, we define  $G(L):=\Mat_{s\times 1}(L)$ whose $\bA$-module structure is given by 
\[
t\cdot x:=\varphi(t)x=A_0x+A_1x^{(1)}+\dots+A_mx^{(m)}, \, \, x\in G(L).
\]	
For any $t$-module $G$, there exists a unique everywhere convergent $\mathbb{F}_q$-linear homomorphism $\Exp_G:\Lie(G)(\CC_{\infty})\to G(\CC_{\infty})$, \textit{the exponential function of $G$}, given by 
$\Exp_G(x)=\sum_{i\geq 0}\alpha_ix^{(i)}$ for $x\in \Lie(G)(\CC_{\infty})$ so that it has the property that $\alpha_0=\Id_s$ and
\begin{equation}\label{E:FuncEqtn}
\Exp_G(\rd_{\varphi}(t)x)=\varphi(t)\Exp_G(x).
\end{equation}
We set $\Lambda_{G}:=\Ker(\Exp_{G})$ which is an $\bA$-module, and call a non-zero element of $\Lambda_{G}$ \textit{a period of $G$}.

In what follows, we  briefly mention the formalism of the de Rham cohomology for $t$-modules. The theory of biderivations, quasi-periodic functions, and the de Rham module for Drinfeld modules were developed by Anderson, Deligne, Gekeler, and Yu (see \cite{Gekeler89, Yu90}). Later, the theory were extended to general $t$-modules by Brownawell and Papanikolas (see \cite{BP02}). More precisely, Brownawell and Papanikolas defined \textit{a $\varphi$-biderivation of $G$ (over $L$)} as a map $\delta:\bA \rightarrow \Mat_{1\times s}(L[\tau]\tau)$ satisfying
\[
\delta(ab) = a(\theta)\delta(b) + \delta(a)\varphi(b)
\]
for each $a,b\in \bA$ and then introduced the de Rham module $\mathrm{H}_{\mathrm{DR}}(G)$ as a certain quotient space of biderivations (see \S\ref{SS:bids} for its explicit definition). For the setting of $t$-modules, \textit{the quasi-periodic function of $G$ associated to $\delta$} is defined as  the unique everywhere convergent function $F_{G,\delta}:\Lie(G)(\mathbb{C}_{\infty})\to \mathbb{C}_{\infty}$ satisfying
\[
	F_{G,\delta}(\rd_{\varphi}(a)\bsz) = a(\theta)F_{G,\delta}(\bsz) + \delta(a)(\Exp_{G}(\bsz))
\]
for each $\bsz \in \Lie(G)(\mathbb{C}_{\infty})$. Furthermore, for each $\bsy \in \Lie(G)(\CC_\infty)$, we call $F_{G, \delta}(\bsy)$ \textit{a quasi-logarithm}.  We refer the reader to \S\ref{SS:bids} for further details on these objects.

Among all interesting examples of $t$-modules (see for example \cite{AndThak90} and \cite{NPapanikolas21}), in this paper, we focus on two families of $t$-modules $\mathcal{E}_n=(\mathbb{G}^{rn+r-1}_{a/\overline{K}}, \varphi_n)$ and $G_n=(\mathbb{G}_{a/\oK}^{rn+1},\phi_n)$ of dimension $rn+r-1$ and $rn+1$ respectively (see \S\ref{S:tmoduleEn} and \S\ref{S:tmoduleGn} for their explicit definitions). The significance of these families rely on the fact that, under the equivalence between the categories of $t$-modules and $\bA$-finite dual $t$-motives (see \S\ref{SS:t-motives} for more details), $\mathcal{E}_n$ corresponds to the tensor product of $(r-1)$-st exterior power of the dual $t$-motive of the Drinfeld $\bA$-module $G_0$ with the $n$-th Carlitz tensor power (see \eqref{E:exterisom}), and $G_n$ corresponds to the tensor product of the dual $t$-motive of $G_0$ with the $n$-th Carlitz tensor power (see \eqref{E:tensisom}). We refer the reader to \S\ref{S:tmoduleEn} and \S\ref{S:tmoduleGn} for explicit descriptions of these relations. We inform the reader that in an upcoming work \cite{GN24}, the $t$-modules $G_n$ and $\mathcal{E}_n$ as well as the main results of the present paper are used to prove the transcendence of the special values of Goss $L$-functions of Drinfeld $\bA$-modules defined over $K$ at positive integers. Their link to such values are indeed the main motivation for the authors to focus on the arithmetic aspects of the $t$-modules we introduce here.

Our main goal in this paper is to study the algebraic independence of the \emph{tractable} coordinates of logarithms, in the sense of \cite[Def.~3.3.1]{CM21}, of the $t$-modules $\mathcal{E}_n$ and $G_n$. Here, all $r-1$ coordinates of a logarithm of $\mathcal{E}_0$ are tractable and for $n \geq 1$, the tractable coordinates of a logarithm of $\mathcal{E}_n$ are defined as the bottom $r$-coordinates. Similarly, for $n \geq 1$, the bottom $r$-coordinates of a logarithm of $G_n$ are the tractable coordinates.

Let $K_n$ be the fraction field of the ring of endomorphisms of $\mathcal{E}_n$ (see Remark~\ref{R:isomorphismsofvss} and \S\ref{S:GGofEn}). The first result of the present paper can be stated as follows.

\begin{theorem}\label{T:1}
    	For $n \geq 0$ and each $1 \leq \ell \leq m$, let $\bsy_{\ell} =[y_{\ell,1}, \dots, y_{\ell,rn+r-1}]^{\tr} \in \Lie(\mathcal{E}_n)(\CC_{\infty})$ be such that $\Exp_{\mathcal{E}_n}(\bsy_{\ell})\in \mathcal{E}_n(\oK)$. Suppose that the set  $\{\bsy_{1}, \dots, \bsy_{m}\}$ is linearly independent over $K_{n}$. Then the following assertions hold:
	\begin{itemize}
		\item[(i)] Let  $F_{\mathcal{E}_0,\delta}$ be the quasi-periodic function associated to the $\varphi_0$-biderivation $\delta$ which maps $t\mapsto (\tau, 0,\dots,0)$. If $n=0$, then 
		\begin{equation}\label{E:T1i}
  \bigcup\limits_{\ell=1}^{m} \{F_{\mathcal{E}_0,\delta}(\bsy_{\ell}), y_{\ell, 2}, \dots, y_{\ell, r-1}\}
  \end{equation}
  is algebraically independent over $\oK$. In particular,  for some non-constant polynomial $\bff \in \oK[X_1,\dots,X_{(r-1)m}]$, if 
		$
		\bff(F_{\mathcal{E}_0,\delta}(\bsy_1), y_{1,2},\dots, y_{1, r-1}, \dots,  F_{\mathcal{E}_0,\delta}(\bsy_m), y_{m,2},\dots, y_{m, r-1})
		$ is non-zero, then it is transcendental over $\oK$.
  
		\item[(ii)] If $n\in \ZZ_{\geq 1}$, then 
		\begin{equation}\label{E:T1ii}
  \bigcup\limits_{\ell=1}^{m} \{y_{\ell, rn}, y_{\ell, rn+1}, \dots, y_{\ell, rn+r-1}\}  
  	\end{equation}
 is algebraically independent over $\oK$. In particular, for some non-constant polynomial $\bff \in \oK[X_1,\dots,X_{rm}]$, if $\bff(y_{1, rn}, \dots, y_{1, rn+(r-1)}, \dots, y_{m, rn}, \dots, y_{m, rn+(r-1)})$ is non-zero, then it is transcendental over $\oK$.
	\end{itemize} 
\end{theorem}
We remark that the de Rham module $\mathrm{H}_{\mathrm{DR}}(\mathcal{E}_n)$ of $\mathcal{E}_n$ is an $r$-dimensional $\oK$-vector space parametrizing the extensions of $\mathcal{E}_n$ by $\GG_a$. Let $\{\delta_1, \dots, \delta_r\}$ be a $\oK$-basis of $\mathrm{H}_{\mathrm{DR}}(\mathcal{E}_n)$ and $\bsy=[y_1,\dots,y_{rn+r-1}]^{\tr}\in \Lie(\mathcal{E}_n)(\mathbb{C}_{\infty})$ be such that $\Exp_{\mathcal{E}_n}(\bsy)\in \mathcal{E}_n(\oK)$. Then, the field $\oK\left(\cup_{i=1}^r \{F_{\mathcal{E}_n, \delta_i}(\bsy)\}\right)$ is independent of the choices of the basis of $\mathrm{H}_{\mathrm{DR}}(\mathcal{E}_n)$. Moreover, we indeed have (see the proof of Corollary \ref{C:Int0} in \S\ref{S:Proof1} for more for details) 
\[
\oK\left(\cup_{i=1}^r \{F_{\mathcal{E}_n, \delta_i}(\bsy)\}\right) = \begin{cases} \oK\left(F_{\mathcal{E}_0,\delta}(\bsy), y_2, \dots, y_{r-1}\right) & \text{ if }n=0\\
 \oK\left(y_{rn}, y_{rn+1}, \dots, y_{rn+r-1}\right) & \text{ otherwise.}
 \end{cases}
\]
We then have the following corollary to Theorem \ref{T:1}.

\begin{corollary}\label{C:Int0}
	For $n \geq 0$ and each $1 \leq \ell \leq m$, let $\bsy_{\ell} =[y_{\ell,1}, \dots, y_{\ell,rn+r-1}]^{\tr} \in \Lie(\mathcal{E}_n)(\CC_{\infty})$ be such that $\Exp_{\mathcal{E}_n}(\bsy_{\ell})\in \mathcal{E}_n(\oK)$. Suppose that the set  $\{\bsy_{1}, \dots, \bsy_{m}\}$ is linearly independent over $K_{n}$. Then, the set 
  \[
  \bigcup\limits_{\ell=1}^{m} \{F_{\mathcal{E}_n, \delta_1}(\bsy_\ell), \dots, F_{\mathcal{E}_n, \delta_r}(\bsy_\ell)\}
  \]
  is algebraically independent over $\oK$. 
\end{corollary}

In what follows, we state our results on the algebraic independence of logarithms for the family $G_n$ of $t$-modules. Let $K_{n}^{\tens}$ be the fraction field of the ring of endomorphisms of $G_n$ (see Remark~\ref{R:isomorphismsofvss} and \S\ref{S:GGofGn}).

\begin{theorem}\label{T:2}
    	For $n \geq 1$ and each $1 \leq \ell \leq m$, let $\bsy_{\ell} =[y_{\ell,1}, \dots, y_{\ell,rn+1}]^{\tr} \in \Lie(G_n)(\CC_{\infty})$ be such that $\Exp_{G_{n}}(\bsy_{\ell})\in G_n(\oK)$. Suppose that  the set $\{\bsy_{1}, \dots, \bsy_{m}\}$ is linearly independent over $K_{n}^{\tens}$. Then, 
	\[
 \bigcup\limits_{\ell=1}^{m} \{y_{\ell, r(n-1)+2}, y_{\ell, r(n-1)+3}, \dots, y_{\ell, rn+1}\}
 \] is algebraically independent over $\oK$. In particular, for some non-constant polynomial $\widetilde{\bff} \in \oK[X_1,\dots,X_{rm}]$, if
	$\widetilde{\bff}(y_{1, r(n-1)+2}, \dots, y_{1, rn+1}, \dots, y_{m, r(n-1)+2}, \dots, y_{m, rn+1})$ is non-zero, then it is transcendental over $\oK$.
\end{theorem}

As in the case of $\mathcal{E}_n$, we note that the de Rham module $\mathrm{H}_{\mathrm{DR}}(G_n)$ of $G_n$ is also an $r$-dimensional $\oK$-vector space parametrizing the extensions of $G_n$ by $\GG_a$.  Let $\{\delta_1^{\tens}, \dots, \delta_r^{\tens}\}$ be a $\oK$-basis of $\mathrm{H}_{\mathrm{DR}}(G_n)$. Setting $\bsy =[y_1, \dots, y_{rn+1}]^\tr\in \Lie(G_n)(\CC_\infty)$, for each $n\geq 1$, we have (see also the proof of Corollary \ref{C:Int1} in \S\ref{SS:ProofThmiii})
\[
\oK\left(\cup_{i=1}^r \{F_{G_n, \delta^{\tens}_i}(\bsy)\}\right) = \oK\left(y_{r(n-1)+2}, y_{r(n-1)+3}, \dots, y_{rn+1}\right).
\]
Thus, our next corollary can be stated as follows.

\begin{corollary}\label{C:Int1}
	For $n \geq 1$ and each $1 \leq \ell \leq m$, let $\bsy_{\ell} =[y_{\ell,1}, \dots, y_{\ell,rn+1}]^{\tr} \in \Lie(G_n)(\CC_{\infty})$ be such that $\Exp_{G_{n}}(\bsy_{\ell})\in G_n(\oK)$. Suppose that  $\bsy_{1}, \dots, \bsy_{m}$ are linearly independent over $K_{n}^{\tens}$. Then, 
	\[
 \bigcup\limits_{\ell=1}^{m} \left\{ F_{G_n, \delta^{\tens}_1}(\bsy_\ell), \dots, F_{G_n, \delta^{\tens}_r}(\bsy_\ell)\right\}
 \] 
 is algebraically independent over $\oK$. 
\end{corollary}

\begin{remark} Let $G_0:=\mathbf{C}=(\mathbb{G}_{a/\oK},C)$ be the Carlitz module given by $C(t):=\theta+\tau$. Consider the $n$-th Carlitz tensor power $\mathbf{C}^{\otimes n}=(\mathbb{G}_{a/\overline{K}},C^{\otimes n})$ (see Example \ref{Ex:0}(ii) for its explicit definition). Then, for $n \geq  2$, we have $G_{n-1} =\mathbf{C} \otimes \mathbf{C}^{\otimes (n-1)}= \mathbf{C}^{\otimes n}$. For $m=0$, set $L_0:=1$ and for $m\geq 1$, consider $L_m:=(\theta-\theta^{q^m})L_{m-1}\in A$. Let $\alpha_1,\dots,\alpha_u\in \oK^{\times}$ be such that $\inorm{\alpha_i}<q^{nq/(q-1)}$ for each $1\leq i \leq u$. Define \textit{the polylogarithm} $\mathcal{L}_i:=\sum_{m\geq 0}\frac{\alpha_i^{q^m}}{L_m^n}\in \mathbb{C}_{\infty}$. Then, by \cite[\S2.2, 2.4]{AndThak90}, we have 
\[
\Exp_{\mathbf{C}^{\otimes n}}(\mathcal{V}_i)=\begin{pmatrix} 0\\
\vdots \\
0\\
\alpha_i
\end{pmatrix}\in \mathbf{C}^{\otimes n}(\oK), \ \ \mathcal{V}_i:=\begin{pmatrix} *\\
\vdots \\
*\\
\mathcal{L}_i
\end{pmatrix}.
\]
Suppose that $\mathcal{L}_1,\dots, \mathcal{L}_u$ are linearly independent over $K$. Note that $K_{n-1}^{\text{tens}} = K_C = K$ (see
Remark \ref{R:isomorphismsofvss}(ii) and Proposition \ref{P:rings}). This means that $\mathcal{V}_1,\dots, \mathcal{V}_u$ are linearly independent over $K_{n-1}^{\text{tens}}$. Thus, Theorem \ref{T:2} implies that $\mathcal{L}_1,\dots, \mathcal{L}_u$ are algebraically
independent over  over $\oK$. This recovers the result of Chang and Yu \cite[Cor. 3.2]{CY07} for the Carlitz tensor power case.
\end{remark}

In what follows, we explain the strategy of the proof of our main results. We inform the reader that Theorem \ref{T:1} and Theorem \ref{T:2} use the same ideas up to certain technical differences which will be discussed explicitly throughout the paper and so, we explain our methods for Theorem~\ref{T:1}. Our main tool is Papanikolas’s theorem \cite[Thm. 1.1.7]{P08} which  proves that the transcendence degree of the period matrix of a $t$-motive is equal to the dimension of its associated motivic Galois group (see Theorem~\ref{T:TannakianMain}). 
 First, by using its relation to the $t$-motive $H_\phi$ associated to the Drinfeld $\bA$-module $G_0$, we show that the $t$-motive $H_n$ associated to $\mathcal{E}_n$ is simple and that the corresponding motivic Galois groups $\Gamma_{H_\phi}$ and $\Gamma_{H_n}$ of $G_0$ and $\mathcal{E}_n$ respectively are equal (see \S\ref{S:GGofEn}). Then, since an explicit description of $\Gamma_{H_\phi}$ is known \cite[Thm.~3.5.4]{CP12}, we obtain an explicit description of $\Gamma_{H_n}$. 
 
 Next, we define particular elements in the Tate algebra (see \S\ref{SS:gahaextn}) and construct $t$-motives $Y_1, \dots, Y_m$ (see \S\ref{SS:tmotiveconsextn}) so that the entries of their period matrices are related to $\bsy_1, \dots, \bsy_m$ satisfying the hypothesis of Theorem~\ref{T:1}.  These $t$-motives represent classes in $\Ext_{\sT}^1(\mathbf{1}_{\sT}, H_n)$, where $\sT$ denotes the category of $t$-motives. To describe precisely the elements of $K_n$, we construct an explicit map that describes an isomorphism between $\End(G_0)$ and $\End(\mathcal{E}_n)$, and further identify $\End(\mathcal{E}_n)$ with an integral domain described using this isomorphism (see \S\ref{S:EndEn}). 
 Inspired by the methods of Chang and Papanikolas \cite{CP12}, using this identification of $\End(\mathcal{E}_n)$, we prove an essential result (see Theorem~\ref{T:Trivial2}), which shows that if $\bsy_1, \dots, \bsy_m$ are linearly independent over $K_n$ from an $\bA$-basis of $\Lambda_{\mathcal{E}_n}$, then $Y_1, \dots, Y_m$ are linearly independent over $\End_{\sT}(H_n)$. Then, we explicitly calculate the motivic Galois group of the $t$-motive $\mathbf{X}$ isomorphic to $\oplus_{\ell=1}^m Y_\ell$ (see Theorem~\ref{T:MaintTtr2}).
This result yields the algebraic independence of desired coordinates of logarithms and periods of $\mathcal{E}_n$.

The outline of the present paper can be given as follows. In \S\ref{S:tmodules}, we introduce $t$-modules as well as certain objects attached to them such as Anderson generating functions, biderivations, and quasi-periodic functions. Later, we revisit the theory of Papanikolas established in \cite{P08} and provide necessary details used in the paper. In \S\ref{S:tmoduleEn} and \S\ref{S:tmoduleGn}, we take a closer look at the $t$-modules $\mathcal{E}_n$ and $G_n$. More precisely, we determine their motivic Galois groups and describe their endomorphisms. Finally, we prove Theorem \ref{T:1} in \S\ref{S:MainResultsEn} and Theorem \ref{T:2} in \S\ref{S:MainResultsGn}.

\subsection*{Acknowledgments} The authors thank the referee for careful reading and useful suggestions that improved the presentation of the results. The first author was supported by NSTC Grant 113-2115-M-007-001-MY3. The first author was also partially supported by Deutsche Forschungsgemeinschaft (DFG) through CRC-TR 326 `Geometry and Arithmetic of Uniformized Structures', project number 444845124.  The second author was partially supported by MOST Grant 110-2811-M-007-517.

\section{\texorpdfstring{$t$}{t}-modules}\label{S:tmodules}
In this section, we first introduce $t$-modules which can be seen as a higher dimensional generalization of Drinfeld $\bA$-modules (see \cite{And86}, \cite{BP20} for more details) and objects corresponding to them such as Anderson generating functions, biderivations, and quasi-periodic functions. We also provide necessary background on $t$-motives and state transcendence results of Papanikolas \cite{P08}.

\subsection{Tate algebras and twisting} Let $t$ be an indeterminate over $\CC_{\infty}$. We define \textit{the Tate algebra} $\mathbb{T}$ to be the ring of power series of $t$ with coefficients in $\CC_{\infty}$ satisfying the condition: 
\[
\TT:=\Big\{ \sum_{i=0}^{\infty} c_it^i\in \CC_{\infty}[[t]]   \ \ | \ \  \inorm{c_i}\to 0 \text{ as } i
\to \infty\Big\}.
\]
One can equip $\TT$ with the non-archimedean norm $\dnorm{\cdot}$ given by
$
\dnorm{\sum_{i=0}^{\infty} c_it^i}:=\max\{\inorm{c_i}\ \ |\ \  i\in \ZZ_{\geq 0}  \}.
$

For positive integers $m$ and $\ell$, we canonically extend the norm $\dnorm{\cdot}$ to the set of $(m\times \ell)$-matrices  $\Mat_{m\times \ell}(\TT)$ with coefficients in $\TT$  so that for any $B=(B_{ik})\in \Mat_{m\times \ell}(\TT)$, we have $\dnorm{B}:=\max\{\dnorm{B_{ik}}\}$. 

For any $f=\sum_{i\geq 0}c_it^i\in \TT$ and $j\in \ZZ$, we define $f^{(j)}:=\sum_{i\geq 0}c_i^{q^j}t^i\in \TT$ and extend it to any matrix  $B=(B_{ik})\in \Mat_{m\times \ell}(\TT)$ so that $B^{(j)}:=(B_{ik}^{(j)})\in  \Mat_{m\times \ell}(\TT)$. Furthermore, we set $\LL$ to be the fraction field of  $\TT$ and for any element $\mathfrak{p}=f/g\in \LL$ with $f\in \TT$ and $g\in \TT\setminus\{0\}$, we set $\mathfrak{p}^{(j)}:=f^{(j)}/g^{(j)}$. The extension of the twisting operation for the elements of $\Mat_{m\times \ell}(\LL)$ can be defined similarly. Considering the entries of $B\in \Mat_{m\times \ell}(\LL)$ as rational functions of $t$, for any $\xi\in \CC_{\infty}$, we also fix the notation $B(\xi):=B_{|t=\xi}$ whenever the right hand side converges in $\Mat_{m\times \ell}(\CC_{\infty})$. 

We now introduce an element in $\TT$ which will be fundamental for our computations. Let $(-\theta)^{1/(q-1)}$ be a fixed $(q-1)$-st root of $-\theta$. Consider the infinite product called \emph{the Anderson-Thakur series}
\[
\Omega:=\Omega(t):=(-\theta)^{-q/(q-1)}\prod_{i=1}^{\infty}\Big(1-\frac{t}{\theta^{q^i}}\Big)\in \CC_{\infty}[[t]].
\]
One can see that $\Omega$ is indeed a unit in $\TT$ and an entire function of $t$ \cite[Sec. 3.3.4]{P08}. Moreover, it satisfies the functional equation  $\Omega^{(-1)}=(t-\theta)\Omega$. We furthermore define \textit{the Carlitz period} $\widetilde{\pi}$ by
\[
\widetilde{\pi}:=-\Omega^{-1}(\theta)=\theta(-\theta)^{1/(q-1)}\prod_{i=1}^{\infty}(1-\theta^{1-q^i})^{-1}\in \CC_{\infty}^{\times}.
\]

	\subsection{\texorpdfstring{$t$}{t}-modules and Anderson generating functions}\label{TmodulesNAGF} 
	In this subsection, we restate the definition of $t$-modules for the convenience of the reader, and then introduce endomorphisms of $t$-modules and Anderson generating functions attached to them.
	
	\begin{definition} 
		\begin{itemize}
			\item[(i)] We let $L$ be a field with $K \subseteq L \subseteq \CC_\infty$. A \emph{$t$-module  defined over $L$ of dimension $s\in \ZZ_{\geq 1}$} is a tuple $\smash{G=(\mathbb{G}_{a/L}^{s},\varphi)}$ consisting of the $s$-dimensional additive algebraic group $\mathbb{G}_{a/L}^s$ over $L$ and an $\mathbb{F}_q$-linear ring  homomorphism $\varphi:\bA\to \Mat_s(L)[\tau]$ given by 
			\begin{equation}\label{E:varphi}
			\varphi(t)=A_0+A_1\tau+\dots+A_m\tau^m
			\end{equation}
			for some $m\in \ZZ_{\geq 0}$ so that $\rd_{\varphi}(t):=\rd \varphi(t)=A_0=\theta \Id_s+N$ where $\Id_s$ is the $s\times s$ identity matrix and $N$ is a nilpotent matrix. 
			
			\item[(ii)] Let $G_1=(\mathbb{G}_{a/L}^{s_1},\varphi_1)$ and $G_2=(\mathbb{G}_{a/L}^{s_2},\varphi_2)$ be two $t$-modules.  \textit{A morphism $\rP: G_1 \rightarrow G_2$ defined over $L$} is an element $\rP \in \Mat_{s_2 \times s_1}(L[\tau])$ satisfying
			\begin{equation*}\label{E:isogn}
			\rP\varphi_1(t) = \varphi_2(t) \rP.
			\end{equation*}
			We set $\Hom_{L}(G_1,G_2)$ to be the set of morphisms $G_1\to G_2$ defined over $L$ and fix the notation $\End_{L}(G):=\Hom_{L}(G,G)$ for the endomorphism ring of $G=(\mathbb{G}_{a/{L}}^{s},\varphi)$. Note that $\End_{L}(G)$ has an $\bA$-module structure  given by 
			\[
			a\cdot \rP:=\varphi(a)\rP, \ \ a\in \bA,  \ \ \rP\in \End_{L}(G).
			\] 
			Furthermore, when $s_1=s_2$, we say that $G_1$ and $G_2$ are \textit{isomorphic over ${L}$} if there exists an element $u\in \GL_{s_1}({L})$ so that $u\varphi_1(t) = \varphi_2(t) u$.
		\end{itemize}
	\end{definition}
	
	\begin{example}\label{Ex:0} 
		\begin{itemize}
			\item[(i)] A Drinfeld $\bA$-module is a $t$-module of dimension one.  		
	\item[(ii)]
 The Drinfeld $\bA$-module $\mathbf{C}:=(\GG_{a/K}, C)$ of rank one given by 
	\[
	C(t):=\theta+\tau
	\]	
 is called the \emph{Carlitz module}.
	For any $n\in \smash{\ZZ_{\geq 1}}$, we define the $n$-dimensional $t$-module $\mathbf{C}^{\otimes n}:=(\mathbb{G}_{a/K}^n,C^{\otimes n})$, where the $\mathbb{F}_q$-linear homomorphism  $C^{\otimes n}:\bA\to \Mat_n(K)[\tau]$ is defined by 
		
		\[
		C^{\otimes n}(t):=\begin{bmatrix}
		\theta&1& & \\
		& \ddots&\ddots & \\
		& & \theta & 1\\
		& & & \theta  
		\end{bmatrix}+\begin{bmatrix}
		0&\dots&\dots &0 \\
		\vdots& & &\vdots \\
		0& &  & \vdots\\
		1&0&\dots& 0 
		\end{bmatrix}\tau.
		\] 
\noindent We call $\mathbf{C}^{\otimes n}$ \textit{the $n$-th tensor power of the Carlitz module} (see \cite{AndThak90} for more details). 
		\end{itemize}
	\end{example}

 In what follows, we briefly explain the setting of vector valued Anderson generating functions. One can consult \cite{Gre19A,Gre19B,GreDac20,NPapanikolas21} for further details.

For a $t$-module $G=(\mathbb{G}_{a/L}^s,\varphi)$ and $\boldsymbol{w}\in \Lie(G)(\CC_{\infty})$, we define the \emph{Anderson generating function $\cG_{\boldsymbol{w}}(t)$ of $G$ at $\boldsymbol{w}$} to be the $s$-dimensional vector of power series in $t$ given by 
\begin{equation}\label{E:AGF}
\cG_{\boldsymbol{w}}(t):=\sum_{i=0}^{\infty} \Exp_G(\rd_{\varphi}(t)^{-i-1} \boldsymbol{w}) t^i\in \TT^s.
\end{equation}
To distinguish our objects, when $G$ is a Drinfeld $\bA$-module, we fix the notation $f_{\boldsymbol{w}}(t)$ for the Anderson generating function of $G$ at $\boldsymbol{w}$.

For any element $B=B_0+B_1\tau +\dots +B_k\tau^k\in \Mat_{m\times s}(L[\tau])$ and $\textbf{x}\in \TT^s$, we set, by a slight abuse of notation,
\[
B\cdot \textbf{x}:=B_0\textbf{x}+\dots+B_k\textbf{x}^{(k)}.
\]  
We collect some properties of Anderson generating functions in the following lemma which can be seen as a generalization of Pellarin's results \cite{Pellarin08} on Anderson generating functions of Drinfeld $\bA$-modules. We refer the reader to \cite[Sec. 4.2]{NPapanikolas21} for its proof.
\begin{lemma}\cite[Lem.~4.2.2, Rem.~4.2.5, Prop.~4.2.12(b)]{NPapanikolas21}\label{L:AGFproperties}\label{L:AGFNP21}
Let $G=(\mathbb{G}_{a/L}^s,\varphi)$ be a $t$-module. Let $\boldsymbol{w}\in \Lie(G)(\CC_{\infty})$ and  $\Exp_G=\smash{\sum_{i\geq 0}\alpha_i\tau^i}$ be the exponential function of $G$. 
	\begin{itemize}
	\item[(i)] We have 
	\[
	\cG_{\boldsymbol{w}}(t)=\sum_{i=0}^{\infty}\alpha_i ((\rd_{\varphi}(t) -t\Id_s)^{-1})^{(i)}\boldsymbol{w}^{(i)},
	\]
	which converges in $\TT^s$. Furthermore each entry of $\mathcal{G}_{\boldsymbol{w}}(t)$ can be extended to a meromorphic function of t with possible poles only at $t=\theta,\theta^q,\theta^{q^2},\dots.$
	\item[(ii)] We have 
	\[
	\varphi(t)\cdot\cG_{\boldsymbol{w}}(t)=t\cG_{\boldsymbol{w}}(t)+\Exp_G(\boldsymbol{w}).
	\]
	\end{itemize}
\end{lemma}

\subsection{Biderivations and quasi-periodic functions of \texorpdfstring{$t$}{t}-modules}\label{SS:bids} 
Our purpose now is to introduce the theory of biderivations for $t$-modules and quasi-periodic functions attached to them. We note that such theory for Drinfeld $\bA$-modules were developed by Anderson, Deligne, Gekeler and Yu for Drinfeld $\bA$-modules (\cite{Gekeler89}, \cite{Yu90}) and extended by Brownawell and Papanikolas \cite{BP02} to a more general setting. We closely follow \cite{BP02} for our exposition. Throughout this subsection, we assume that $L$ is an algebraically closed field in $\CC_{\infty}$ containing $K$.

\begin{definition}Let $G=(\mathbb{G}_{a/L}^s,\varphi)$ be a $t$-module. \textit{A $\varphi$-biderivation defined over $L$} is an $\mathbb{F}_q$-linear map $\delta:\bA \rightarrow \Mat_{1\times s}(L[\tau]\tau )$ satisfying
	\[
	\delta(ab) = a(\theta)\delta(b) + \delta(a)\varphi(b), \quad \textup{ } \,  a,b \in \bA.
	\]
\end{definition}
The set of $\varphi$-biderivations forms an $L$-vector space, and we will denote it by $\Der(\varphi)$.

Let $N\in \Mat_s(L)$ be the nilpotent part of $\rd_{\varphi}(t)=\theta \Id_s+N$. We define
\[
N^{\perp}:= \{v \in \Mat_{1\times s}(L)\ \  | vN=0\}.
\]
Consider a vector $U=[U_1,\dots,U_s]\in \Mat_{1\times s}(L[\tau])$ such that $\rd U=[\rd U_1,\dots,\rd U_s]\in N^{\perp}$. This allows us to define the map $\delta^{(U)}:\bA \rightarrow \Mat_{1\times s}(L[\tau]\tau )$ given by 
\[
\delta^{(U)}(a) := U \varphi(a) - a(\theta)U, \ \ a\in \bA.
\]
Observe that
\[
\delta^{(U)}(ab)= a(\theta)(U \varphi(b) - b(\theta)U)+(U\varphi(a)-a(\theta)U)\varphi(b)= a(\theta)\delta^{(U)}(b) + \delta^{(U)}(a)\varphi(b)
\]
for all $a,b\in \bA$ and therefore, $\delta^{(U)}$ indeed forms a $\varphi$-biderivation.  Set $\Der_{\mathrm{si}}(\varphi)=\{\delta^{(U)} | \rd U =0\}$, which is an $L$-subspace of $\Der(\varphi)$, is called the space of \emph{strictly inner $\varphi$-biderivations}. Then, the \emph{de Rham module} for $\varphi$ is $\rH_{\mathrm{DR}}(\varphi)= \Der(\varphi)/\Der_{\mathrm{si}}(\varphi)$. 

Let $\bsz:=[z_1,\dots,z_s]^{\tr}$ be a column vector of $s$-many variables. For a given biderivation $\delta$ defined over $L$, by \cite[Prop.~3.2.1]{BP02}, there exists an $\mathbb{F}_q$-linear entire function $F_{G,\delta}:\CC_{\infty}^s\to \CC_{\infty}$ given by the unique power series 
$F_{G,\delta}(\bsz) = \sum_{h\geq 1} c_{h,1}z_1^{q^h}+ \dots + c_{h,s}z_s^{q^h}  \in \power{L}{z_1, \dots, z_s}$ satisfying
\begin{equation}\label{E:quasper}
	F_{G,\delta}(\rd_{\varphi}(a)\bsz) = a(\theta)F_{G,\delta}(\bsz) + \delta(a)(\Exp_{G}(\bsz)).
\end{equation}

We call this unique function $F_{G,\delta}$ \textit{the quasi-periodic function associated to $\delta$}. For any $\lambda\in \Lambda_{G}\setminus \{0\}$ and a $\varphi$-biderivation $\delta$, we call $F_{G,\delta}(\lambda)$ \textit{the quasi-period of $\delta$ associated to $\lambda$}, and for a general $\bsy \in \Lie(G)(\CC_\infty)$, we call  $F_{G,\delta}(\bsy)$ \textit{the quasi-logarithm of $\delta$ associated to $\bsy$}. As an example, by \cite[Prop. 3.2.2]{BP02}, the quasi-periodic function $F_{G,\delta^{(U)}}$ associated to $\delta^{(U)}$ is given by $F_{G,\delta^{(U)}}(\bsz)=U\Exp_{G}(\bsz)-\rd U\bsz$. 

We finish this subsection with a brief discussion, which will be used later, on quasi-periodic functions and quasi-periods for the Drinfeld $\bA$-module $G_0:=(\GG_{a/\oK}, \phi)$ defined in \eqref{E:DrinfeldDef}. One can see that the map $\delta_{0}:\bA\to L[\tau]\tau $ given by $\delta_{0}(t)=\phi(t)-\theta$ forms a $\phi$-biderivation. Similarly, for each $1\leq i \leq r-1$, the map $\delta_{i}:\bA\to L[\tau]\tau $ given by $\delta_{i}(t)=\tau^i$ also forms a $\phi$-biderivation. By Drinfeld \cite{Dri74}, the $\bA$-module $\Lambda_{G_0}$ is free of rank $r$ and thus, we can fix an $\bA$-basis $\{\lambda_1,\dots,\lambda_r\}$ for $\Lambda_{G_0}$. Note that the above discussion in fact implies that $F_{G_0,\delta_{0}}(z)=\Exp_{\phi}(z)-z$ for any $z\in \CC_{\infty}$ and hence, $F_{G_0,\delta_{0}}(\lambda_j)=-\lambda_j$ for any $1\leq j \leq r$. Furthermore, by \cite[pg. 194]{Gekeler89} (see also \cite[Sec. 4.2]{Pellarin08}), we have
\begin{equation}\label{E:quasiperiod}
	F_{G_0,\delta_{i}}(\lambda_j)=f_{\lambda_j}^{(i)}(\theta).
	\end{equation}
We finally set the period matrix $\mathcal{P}$ of $G_0$ to be 
\begin{equation}\label{E:periodmat}
\mathcal{P} := \begin{bmatrix}
-\lambda_1 & F_{G_0,\delta_{1}}(\lambda_1) & \dots & F_{G_0,\delta_{r-1}}(\lambda_1)\\
\vdots& \vdots &  & \vdots\\
-\lambda_r & F_{G_0,\delta_{1}}(\lambda_r) & \dots & F_{G_0,\delta_{r-1}}(\lambda_r)
\end{bmatrix}\in \Mat_{r}(\CC_{\infty}).
\end{equation}
By the Legendre relation of Anderson (see \cite[Sec. 8.1]{GP19}), there exists $\xi \in \oK^{\times}$ such that
\begin{equation}\label{E:periodLeg}
\det(\mathcal{P}) = \frac{\widetilde{\pi}}{\xi}.
\end{equation}

\subsection{\texorpdfstring{$t$}{t}-motives and system of difference equations}\label{SS:t-motives}
This subsection aims to give a rapid introduction to the main tools used in Papanikolas's theory. We refer the reader to \cite{BP20,CPY19,HartlJuschka16,P08} for further details.

Recall that $\oK$ is the algebraic closure of $K$ in $\CC_{\infty}$ and  $\oK(t)$ is the fraction field of the polynomial ring $\oK[t]$. 
We define the non-commutative ring 
\[
\oK(t)[\sigma,\sigma^{-1}]:=\Big\{\sum_{i=-n}^m g_i\sigma^i \  |  \ g_i\in \oK(t), \text{ } n,m\in \ZZ_{\geq 0} \Big\}
\]
 subject to the condition
$
\sigma g=g^{(-1)}\sigma$ for all $ g\in \oK(t)$.
We define $\oK[t,\sigma]$ to be the subset of  $\oK(t)[\sigma,\sigma^{-1}]$  containing elements only of the form $\sum_{i=0}^m f_i\sigma^i$ where $f_i\in \oK[t]$. For $k,\ell\geq \ZZ_{\geq 1}$, we also set $\Mat_{k\times \ell}(\oK[\sigma])$ to be set of elements of the form $\sum_{i=0}^m g_i\sigma^i$ with $g_i\in \Mat_{k\times \ell}(\oK)$.

\begin{definition}
	\begin{itemize}
\item[(i)] A \textit{pre-$t$-motive} $H$ is a left $\oK(t)[\sigma, \sigma^{-1}]$-module that is finite dimensional over $\oK(t)$.
\item[(ii)] We define the pre-$t$-motive $\textbf{1}_{\sP}:=\oK(t)$ which is  equipped with the left $\oK[\sigma,\sigma^{-1}]$-module structure given by 
$
\sigma \cdot f=f^{(-1)}$ for any $f\in \textbf{1}_{\sP}$.
\item[(iii)]For any pre-$t$-motive $H_1$ and $H_2$, the pre-$t$-motive $H_1\otimes H_2:=H_1\otimes_{\oK(t)}H_2$ on which $\sigma$ acts diagonally is called \textit{the tensor product of  $H_1$ and $H_2$}.
\item[(iv)] For any pre-$t$-motive $H$, we define \textit{the dual of $H$} to be the pre-$t$-motive $H^{\vee}$ given by $\smash{H^{\vee}:=\Hom_{\oK(t)}(H,\textbf{1}_{\sP})}$.
\end{itemize}
\end{definition}

The category $\sP$ of pre-$t$-motives is defined so that the objects are given by pre-$t$-motives and the morphisms are given by left $\oK(t)[\sigma, \sigma^{-1}]$-module homomorphisms. We denote the trivial object by $\textbf{1}_{\sP}$. The category $\sP$ is closed under taking direct sums and for any pre-$t$-motive $H$ and a positive integer $m$, we denote the direct sum of $m$-many $H$ by $H^m$. We set $H^\dagger:= \LL \otimes_{\oK(t)}H$ and equip it with the left $\oK(t)[\sigma,\sigma^{-1}]$-module structure given by
\[
\sigma\cdot (g\otimes h)=g^{(-1)}\otimes \sigma h, \ \ g\in \LL, \ \ h\in H.
\]
We also let $H^{\text{Betti}}$ be the set of elements of $H^{\dagger}$ fixed by $\sigma$. Note that it is a finite dimensional $\mathbb{F}_q(t)$-vector space \cite[Prop. 3.3.8]{P08}. 

Let $\Phi \in \GL_r(\oK(t))$ be the matrix representing the $\sigma$-action on a fixed $\oK(t)$-basis of a pre-$t$-motive $H$ which is free of rank $r$ over $\oK(t)$. We say that $H$ is \emph{rigid analytically trivial} if there exists a matrix $\Psi \in \GL_r(\LL)$ such that
\[ \Psi^{(-1)} = \Phi \Psi.\]
Furthermore we call $\Psi$  \emph{a rigid analytic trivialization of $H$}.
For a rigid analytically trivial pre-$t$-motive $H$  with a $\oK(t)$-basis $\bn\in \Mat_{r\times 1}(H)$ and a rigid analytic trivialization $\Psi\in \GL_r(\LL)$, the entries of $\Psi^{-1}\bn$ form an $\FF_q(t)$-basis for $H^{\text{Betti}}$ (\cite[Thm. 3.3.9(b)]{P08}) and moreover, by \cite[Thm. 3.3.15]{P08}, the category of rigid analytically trivial pre-$t$-motives forms a neutral Tannakian category over $\FF_q(t)$ with the fiber functor $H \mapsto H^{\text{Betti}}$. 

\begin{definition}
    An \emph{$\bA$-finite dual $t$-motive $\sH$ over $\oK$} is a left $\oK[t, \sigma]$-module which is free and finitely generated over both $\oK[t]$ and $\oK[\sigma]$ with the property that the determinant of the matrix representing the $\sigma$-action on $\sH$ is equal to $c(t-\theta)^n$ for some $c\in \oK^{\times}$ and $n\in \ZZ_{>0}$.
\end{definition}

The category of $\bA$-finite dual $t$-motives is defined so that the objects are given by $\bA$-finite dual $t$-motives and the morphisms are given by left $\oK[t,\sigma]$-module homomorphisms. Let $\sH$ and $\sH'$ be two $\bA$-finite dual $t$-motives. 		
The set of morphisms $\sH\to \sH'$ of $\bA$-finite dual $t$-motives is denoted by $\Hom_{\oK[t,\sigma]}(\sH,\sH')$ and furthermore, we fix the notation $\End_{\oK[t,\sigma]}(\sH):=\Hom_{\oK[t,\sigma]}(\sH,\sH)$ for the endomorphism ring of $\sH$.

Let $\Phi \in \Mat_r(\oK[t]) \cap \GL_r(\oK(t))$ be the matrix representing the $\sigma$-action on $\sH$ for a given $\oK[t]$-basis. If there exists a matrix $\Psi \in \GL_r(\TT)$ so that $\Psi^{(-1)} = \Phi \Psi$, then $\sH$ is said to be {\emph{rigid analytically trivial}}. Note that by \cite[Prop. 3.1.3]{ABP04}, the entries of $\Psi$ are regular at $t=\theta$ as a function of $t$.

We set $H :=\smash{\oK(t) \otimes_{\oK[t]} \sH}$ and define $\sigma(f \otimes h) := f^{(-1)} \otimes \sigma h$ for any $f\in \oK(t)$ and $h\in \sH$. It is a left  $\oK(t)[\sigma,\sigma^{-1}]$-module which is finite dimensional over $\oK(t)$. Hence, the assignment $\sH \mapsto H$ is a functor from the category of $\bA$-finite dual $t$-motives to the category $\sP$ of pre-$t$-motives. 

Consider the $\bA$-finite dual $t$-motives $\sH_1$ and $\sH_2$ and the pre-$t$-motives $\sH_1 =\smash{\oK(t) \otimes_{\oK[t]} \sH_1}$ and $H_2=\smash{\oK(t) \otimes_{\oK[t]} \sH_2}$. Let $\Hom_{\oK(t)[\sigma, \sigma^{-1}]}(H_1, H_2)$  be the set of left $\oK(t)[\sigma, \sigma^{-1}]$-module homomorphisms from $H_1$ to $H_2$. By  \cite[Prop.~3.4.5]{P08}, the natural map
\begin{equation}\label{E:homsdualpre}
\Hom_{\oK[t, \sigma]}(\sH_1, \sH_2) \otimes_{\bA}\FF_q(t) \rightarrow \Hom_{\oK(t)[\sigma, \sigma^{-1}]}(H_1, H_2) 
\end{equation}
is an isomorphism of $\FF_q(t)$-vector spaces.

We define \textit{the category  $\sT$ of $t$-motives} to be the strictly full Tannakian subcategory generated by the images of rigid analytically trivial $\bA$-finite dual $t$-motives under the assignment $\sH \mapsto H$ which lie in the category of rigid analytically trivial pre-$t$-motives \cite[\S3.4.10]{P08}. Morphisms between $t$-motives are left $\oK(t)[\sigma, \sigma^{-1}]$-module homomorphisms. Let $H$ and $H'$ be two $t$-motives. We set $\Hom_\sT(H,H')$ to be the set of morphisms $H\to H'$ of $t$-motives and furthermore, we fix the notation $\End_\sT(H):=\Hom_\sT(H,H)$ for the endomorphism ring of $H$. 

Let $\sT_H$ denote the strictly full Tannakian subcategory of $\sT$ generated by a $t$-motive $H$. As $\sT_H$ is a neutral Tannakian category over $\FF_q(t)$,  there is an affine group scheme $\Gamma_H$ over $\mathbb{F}_q(t)$ which is a subgroup of the $\mathbb{F}_q(t)$-group scheme $\GL_{s/\FF_q(t)}$ of $s\times s$ invertible matrices so that $\sT_H$ is equivalent to the category $\Rep(\Gamma_H, \FF_q(t))$ of finite dimensional representations of $\Gamma_H$ over $\FF_q(t)$ \cite[\S3.5]{P08}. Furthermore, we call $\Gamma_H$ \textit{the Galois group of $H$}.\par

\begin{theorem}\cite[Thm. 1.1.7]{P08} \label{T:TannakianMain}
Let $H$ be a $t$-motive and let $\Gamma_H$ be its Galois group. Suppose that $\Phi \in \GL_r(\oK(t)) \cap \Mat_r(\oK[t])$ represents the multiplication by $\sigma$ on $H$ and that $\det(\Phi) = c(t-\theta)^n$ where $c \in \oK^\times$ and $n \in \ZZ_{>0}$. Let $\Psi \in \GL_r(\TT)$ be a rigid analytic trivialization of $H$ and $\oK(\Psi(\theta))$ be the field generated by the entries of $\Psi(\theta)$. Then,
\[\trdeg_{\oK} \oK(\Psi(\theta)) = \dim \Gamma_H.
\]
\end{theorem}

We briefly explain the theory of difference equations and the Galois groups attached to them for the triple $(\FF_q(t),\oK(t),\LL)$, which is \textit{$\sigma$-admissible} in the sense of \cite[Sec. 4.1.1]{P08}. For more details on the subject, we refer the reader to \cite[Sec. 4]{P08}. 

Fix $\Phi \in \GL_r(\oK(t))$ and suppose that $\Psi= (\Psi_{ij}) \in \GL_r(\LL)$ satisfies
\[
\Psi^{(-1)} = \Phi \Psi.
\]

We define a $\oK(t)$-algebra map $ \nu : \oK(t)[X, 1/\det X] \rightarrow \LL$ by setting $\nu(X_{ij}) := \Psi_{ij}$, where $X = (X_{ij})$ is an $r \times r$-matrix of independent variables. We define  $\Sigma := \im \nu = \oK(t)[\Psi, 1/\det \Psi] \subseteq \LL$ and consider $Z_\Psi  = \text{Spec} \ \Sigma$. Note that, by the construction, $Z_\Psi$ is the smallest closed subscheme of $\GL_{r/\oK(t)}$ such that $\Psi \in Z_\Psi(\LL)$. \par

Let $\smash{\Psi_1=(\Psi_1)_{ij}, \Psi_2=(\Psi_2)_{ij} \in \GL_r(\LL \otimes_{\oK(t)}\LL)}$ be so that $\smash{(\Psi_1)_{ij} = \Psi_{ij} \otimes 1}$ and $(\Psi_2)_{ij} = 1 \otimes \Psi_{ij}$. We set $\widetilde{\Psi} := (\widetilde{\Psi}_{ij}):=\Psi_1^{-1}\Psi_2 \in \GL_r(\LL \otimes_{\oK(t)} \LL)$. We define an $\FF_q(t)$-algebra map $ \mu : \smash{\FF_q(t)[X, 1/\det X] \rightarrow \LL \otimes_{\oK(t)} \LL}$ by setting $\smash{\mu(X_{ij}) := \widetilde{\Psi}_{ij}}$. We further  consider $\Gamma_\Psi:= \text{Spec} (\im \mu)$. Then $\Gamma_\Psi$ is the smallest closed subscheme of $\GL_{r/\mathbb{F}_q(t)}$ such that $\widetilde{\Psi} \in \Gamma_\Psi(\LL \otimes_{\oK(t)} \LL)$. We call $\Gamma_\Psi$ \textit{the Galois group of the system $\smash{\Psi^{(-1)}} = \Phi \Psi$}.

\begin{theorem}\cite[Sec. 4.3, Sec. 4.3.2, Thm. 4.5.10]{P08} \label{T:Pap}
Let $H$ be a $t$-motive, $\Gamma_H$ be its Galois group and $\Phi \in \GL_r(\oK(t))$ represent the multiplication by $\sigma$ on a $\oK(t)$-basis of $H$. Let $\Psi \in \GL_r(\LL)$ be such that $\Psi^{(-1)} = \Phi \Psi$. Then, the $\FF_q(t)$-scheme $\Gamma_\Psi$ is absolutely irreducible and smooth over the algebraic closure $\smash{\overline{\FF_q(t)}}$ of $\mathbb{F}_q(t)$. Moreover, $\Gamma_\Psi \cong \Gamma_H$ over $\FF_q(t)$. 
\end{theorem}

\subsection{\texorpdfstring{$\bA$}{A}-finite dual \texorpdfstring{$t$}{t}-motives and \texorpdfstring{$t$}{t}-modules}\label{S:AFdual}
We consider the map $*:\oK[\tau]\to \oK[\sigma]$ given by 
$
\big(\sum a_i\tau^i\big)^{*}=\sum a_i^{(-i)}\sigma^i.
$
For any matrix $B=(B_{ij})\in \Mat_{\ell\times m}(\oK[\tau])$, we also define $B^{*}=(B^{*}_{ij})\in \Mat_{m\times \ell}(\oK[\sigma])$ such that $B^{*}_{ij}:=(B_{ji})^{*}$. Observe that for any $B\in \Mat_{\ell_1\times \ell_2}(\oK[\tau])$ and $ D\in \Mat_{\ell_2\times \ell_3}(\oK[\tau])$, we have $(BD)^{*}=D^{*}B^{*}$.

Given a $t$-module $G=(\mathbb{G}_{a/\oK}^{d},\varphi)$, we set $\sH_G:= \Mat_{1\times d}(\oK[\sigma])$ and equip it with a $\oK[t]$-module structure given by
\begin{equation*}\label{E:identification}
ct \cdot h:= ch\varphi(t)^{*}, \ \ c\in \oK, \ \ h\in \sH_G.
\end{equation*}
If $\sH_G$ is an $\bA$-finite dual $t$-motive, then we call $G$ \textit{an $\bA$-finite $t$-module of dimension $d$}.

Let $G'=\smash{(\mathbb{G}_{a/\oK}^{d'},\varphi')}$ be another $t$-module and set $\sH:=\sH_G$ and $\sH':=\sH_{G'}$. For any $m\in \ZZ_{\geq 1}$ and $1\leq j \leq m$, let $s_{j,m}\in \Mat_{1\times m}(\mathbb{F}_q)$ be the vector whose $j$-th entry is 1 and the other entries are zero. Then, clearly the set  $\{s_{1,d}, \dots, s_{d,d}\} \subset \Mat_{1\times d}(\oK[\sigma])$ ( $\{s_{1,d'}, \dots, s_{d',d'}\} \subset \Mat_{1\times d'}(\oK[\sigma])$ resp.) forms a $\oK[\sigma]$-basis for $\sH$ ( $\sH'$ resp.). 

Let $\rP: G \rightarrow G'$ be a morphism of $t$-modules defined over $\oK$, that is, we have an element $\rP \in \Mat_{d' \times d}(\oK[\tau])$ satisfying
$
\rP\varphi(t) = \varphi'(t) \rP.
$		 
We define a map
$
\epsilon_{\rP}: \sH \rightarrow \sH'
$
given by 
\[
\epsilon_{\rP}(h)=h\rP^{*}, \ \ h\in \sH.
\]
Clearly, it is $\oK[\sigma]$-linear and furthermore, note that 
\begin{equation*}
    t \cdot \epsilon_{\rP}(h) = t \cdot h\rP^{*} = h \rP^{*} \varphi'(t)^{*} = h \varphi(t)^{*} \rP^{*} = (t \cdot h) \rP^{*} = \epsilon_{\rP}(t \cdot h).
\end{equation*}
Thus, $\epsilon_{\rP}$ commutes with $t$ and we have proved the following lemma.

\begin{lemma} Let $G$ and $G'$ be $\bA$-finite $t$-modules. Then for every morphism $\rP: G \rightarrow G'$, the map $\epsilon_{\rP}: \sH \rightarrow \sH'$ is a morphism of $\bA$-finite dual $t$-motives, that is, a left $\oK[t, \sigma]$-module homomorphism. 
\end{lemma}

Now, let $d'=d$. It is easy to check by using the construction of $\epsilon_{\rP}$ that the assignment $G \mapsto \sH_{G}$ forms a functor from the category of $\bA$-finite $t$-modules of dimension $d$ over $\oK$ to the category of $\bA$-finite dual $t$-motives which are of rank $d$ over $\oK[\sigma]$. 

\begin{proposition}[{cf. \cite[Thm. 2.5.11]{HartlJuschka16}}]\label{P:fullyfaithful}
 The functor $G\mapsto \sH_{G}$ from the category of $\bA$-finite $t$-modules of dimension $d$ defined over $\oK$ to the category of $\bA$-finite dual $t$-motives of rank $d$ over $\oK[\sigma]$ is fully faithful.
\end{proposition}

\begin{proof}
Consider the $\bA$-finite dual $t$-motives $\sH$ and $\sH'$ defined as above. Let $\mathfrak{h}$ be an element in $\Hom_{\oK[t,\sigma]}(\sH,\sH')$ and consider the $\oK[\sigma]$-basis $\{s_{1,d}, \dots, s_{d,d}\}$ of $\sH$. For any $1\leq j \leq d$, there exist $\mathcal{B}_{j1},\dots,\mathcal{B}_{jd}\in \oK[\sigma]$ such that $\mathfrak{h}(s_{j,d})=\smash{\sum_{i=1}^d}\mathcal{B}_{ji}s_{i,d}$. Set $\mathcal{B}':=(\mathcal{B}'_{ij}) \in \Mat_d(\oK[\tau])$ so that $(\mathcal{B}')^{*}=(\mathcal{B}_{ji})$. Thus, we see that
\[
s_{j,d} \varphi(t)^{*}(\mathcal{B}')^{*} = (t \cdot s_{j,d}) (\mathcal{B}')^{*} = \mathfrak{h}(t \cdot s_{j,d}) = t \cdot \mathfrak{h}(s_{j,d}) = t \cdot (s_{j,d} (\mathcal{B}')^{*}) = s_{j,d} (\mathcal{B}')^{*} \varphi'(t)^{*}.
\]
The above calculation implies that $\mathcal{B}' \varphi(t) = \varphi'(t)\mathcal{B}'$ and hence, $\mathcal{B}': G \rightarrow G'$ is a morphism of $t$-modules, that is, $\mathfrak{h}$ is the morphism $\epsilon_{\mathcal{B}'}: \sH \rightarrow \sH'$ associated to $\mathcal{B}'$. Thus, the functor is full. On the other hand, since $\epsilon_{\mathcal{B}'} =0$ if and only if $\mathcal{B}' =0$, the faithfulness is obvious. 
\end{proof}

In what follows, we briefly describe how to obtain the $\bA$-finite $t$-module $G=(\mathbb{G}^d_{a/\oK},\varphi)$ corresponding to a given $\bA$-finite dual $t$-motive $\sH$ (see also \cite[\S4.4]{BP20}). Note that one can realize the quotient $\sH/(\sigma-1)\sH$ to be $\oK^d$ as $\mathbb{F}_q$-vector spaces. Moreover, considering the $\oK[t]$-module action on $\oK^d$ after such an identification, we obtain the isomorphism
\[
\sH/(\sigma-1)\sH\cong G
\]
as $\bA$-modules. In other words, the $t$-action on $\sH/(\sigma-1)\sH$ coincides with the $t$-action on $G$. In our following example, we provide an explicit description for this phenomenon in the case of Drinfeld $\bA$-modules.

\begin{example}\label{Ex:tmotive} 
\begin{enumerate}
    \item[(i)] Let $G_0=(\GG_{a/\oK}, \phi)$ be the Drinfeld $\bA$-module of rank $r\geq 2$ given in \eqref{E:DrinfeldDef}. To ease the notation, we set $\sH_{\phi}:=\sH_{G_0}=\oK[\sigma]$. For each $1\leq i \leq r$, we let $n_i:=\sigma^{i-1}$. Using the definition of the $\oK[t]$-action on $\sH_{\phi}$ as well as the right division algorithm on $\oK[\sigma]$, one can see that the set $\{n_1, \dots, n_r\}$ forms a $\oK[t]$-basis for $\sH_{\phi}$. Let 
	$\bsn:= [n_1, \dots, n_r]^{\tr}$. The $\oK$-linear action of $\sigma$ on $\bsn$ is given by $\sigma \bsn = \Phi_{\phi} \bsn$, where
	\begin{equation}\label{E:sigmaactDM}
	\Phi_{\phi}:= \begin{bmatrix}
	0 & 1 & \dots & 0\\
	\vdots& \vdots & \ddots & \vdots\\
	0 & 0 & \dots & 1\\
	(t-\theta)/a_r^{(-r)} & -a_1^{(-1)}/a_r^{(-r)} & \dots & -a_{r-1}^{(-r+1)}/a_r^{(-r)}
	\end{bmatrix}\in \Mat_r(\oK[t]) \cap \GL_r(\oK(t)).
	\end{equation}
    
For any $x_1,\dots,x_r\in \oK[t]$, following \cite[\S4.6]{BP20}, we set
\[
[x_1,\dots,x_r]:=x_1\cdot 1+\cdots +x_r\cdot \sigma^{r-1}\in \sH_{\phi}.
\]
For any $1\leq i \leq r-1$ and $x\in \oK[t]$, observe that 
\begin{equation}\label{E:sigma1}
(\sigma^{i}-1)\cdot [x,0,\dots,0]=[-x,0,\dots,x^{(-i)},\dots,0]
\end{equation}
where $x^{(-i)}$ appears in the $(i+1)$-st coordinate. Furthermore, we have 
\begin{equation}\label{E:sigma2}
(\sigma^r-1)\cdot [a_rx,0,\dots,0]=[x^{(-r)}(t-\theta)-a_rx,-a_1^{(-1)}x^{(-r)},\dots, -a_{r-1}^{(-(r-1))}x^{(-r)}].
\end{equation}
Now using \eqref{E:sigma1} and \eqref{E:sigma2}, we obtain
\begin{align*}
t[x,0,\dots,0]&=[tx,0,\dots,0]\\
&=[tx+a_1x^q,-a_1^{(-1)}x,0,\dots,0]+[-a_1x^q,a_1^{(-1)}x,0,\dots,0]\\
&=[tx+a_1x^q,-a_1^{(-1)}x,0,\dots,0]+(\sigma-1)\cdot [a_1x^q,0,\dots,0]\\
&\qquad \vdots\\
&=[tx+a_1x^q+\cdots+a_{r-1}x^{q^{r-1}},-a_1^{(-1)}x,\dots, -a_{r-1}^{(-(r-1))}x]\\
&\ \ +(\sigma-1)\cdot [a_1x^q,0,\dots,0]+\cdots + (\sigma^{r-1}-1)\cdot [a_{r-1}x^{q^{r-1}},0,\dots,0]\\
&=[\theta x+ a_1x^q+\cdots+ a_{r-1}x^{q^{r-1}} + a_{r}x^{q^{r}}, 0,\dots, 0]\\
&\ \ +[x(t-\theta)-a_rx^{{q^r}}, -a_1^{(-1)}x^{(-r)},\dots, -a_{r-1}^{(-(r-1))}x^{(-r)}]\\
&\ \  +\sum_{i=1}^{r-1}(\sigma^i-1)\cdot [a_{i}x^{q^{i}},0,\dots,0]\\
&=[\theta x+ a_1x^q+\cdots+ a_{r}x^{q^{r}}, 0,\dots, 0]\\
&\ \ +(\sigma^r-1)[a_rx^{q^r},0,\dots,0] +\sum_{i=1}^{r-1}(\sigma^i-1)\cdot [a_{i}x^{q^{i}},0,\dots,0],
\end{align*}	
which shows that the $\oK[t]$-action on $\oK$, given by $\phi(t)$, is the same as the $\oK[t]$-action on the quotient $\sH_{\phi}/(\sigma-1)\sH_{\phi}$. Thus, $\sH_{\phi}$ is the $\bA$-finite dual $t$-motive associated to $G_0$.  

For $u \in \CC_{\infty}$, recall that $f_{u}(t)$ is the Anderson generating function of $G_0$ at $u$ given as in \eqref{E:AGF} and the fixed $\bA$-basis $\{\lambda_1, \dots, \lambda_r\}$  of $\Lambda_{G_0}$ in \S\ref{SS:bids}. We introduce the matrices $\Upsilon\in \Mat_r(\TT)$ and $V\in \GL_{r}(\oK)$ by
	\begin{equation}\label{E:UpsilonNV}
	\Upsilon := \begin{bmatrix}
	f_{\lambda_1}^{(1)}(t) &\dots & \dots &\dots & f_{\lambda_1}^{(r)} (t)\\
	\vdots &  &  & & \vdots\\
	\vdots & &  & & \vdots\\
	\vdots & &  & & \vdots\\
	f_{\lambda_r}^{(1)}(t) & \dots & \dots &\dots &  f_{\lambda_r}^{(r)}(t)
	\end{bmatrix} \ \   \mathrm{ and} \, \,
	V := \begin{bmatrix}
	a_1 & a_2^{(-1)}& \dots & a_{r-1}^{(-r+2)}&a_r^{(-r+1)}\\
	a_2 & a_3^{(-1)}&& a_r^{(-r+2)}&\\
	\vdots & \vdots&\reflectbox{$\ddots$}&\\
	\vdots &a_r^{(-1)}&&\\
	a_r&&&
	\end{bmatrix}.
	\end{equation}
	By \cite[Prop. 6.2.4]{GP19} (see also \cite[Prop.~3.3.9]{P08}), we have $\Upsilon\in \GL_r(\TT)$. Define 
	\[	\Psi_\phi := (\Upsilon V)^{-1}\in \GL_r(\TT).
	\]
 By using Lemma~\ref{L:AGFproperties}(ii) one checks directly that $\smash{\Psi_{\phi}^{(-1)}} = \Phi_\phi \Psi_\phi$. Thus, the $\bA$-finite dual $t$-motive $\sH_{\phi}$ is rigid analytically trivial and hence, $H_{\phi} :=\oK(t)\otimes_{\oK[t]} \sH_{\phi} $  is a $t$-motive. 
\item[(ii)] For a positive integer $n$, let $\mathbf{C}^{\otimes n}=(\GG_{a/K}, C^{\otimes n})$ be the $n$-th tensor power of the Carlitz module defined as in Example~\ref{Ex:0}. Similar to (i), set $\sH_{C^{\otimes n}}:=\oK[t]\widetilde{n}$ to be the free $\oK[t]$-module with a $\oK[t]$-basis $\{\widetilde{n}\}$. Defining $\sigma \widetilde{n} = (t-\theta)^n \, \widetilde{n}$, one can see that $\sH_{C^{\otimes n}}$ is an $\bA$-finite dual $t$-motive associated to $\mathbf{C}^{\otimes n}$ with $\oK[\sigma]$-basis given by $\{\widetilde{n}, (t-\theta)\widetilde{n}, \dots, (t-\theta)^{n-1}\widetilde{n}\}$. Moreover, one checks directly that $(\Omega^n)^{(-1)} = (t-\theta)^n \Omega^n$. Thus, $\sH_{C^{\otimes n}}$ is rigid analytically trivial and hence, $H_{C^{\otimes n}}:=\oK(t)\otimes_{\oK[t]} \sH_{C^{\otimes n}}$ is a $t$-motive. 
 \end{enumerate}
\end{example}

\begin{remark}\label{R:isomorphismsofvss}
\begin{itemize}
\item[(i)] Let $G$ be an $\bA$-finite $t$-module, $\sH_{G}$ the corresponding $\bA$-finite dual $t$-motive, and $H_G=\oK(t) \otimes_{\oK[t]} \sH_G\in \sT$ the $t$-motive associated to $\sH_G$.  Recall the definition of $\End_{\oK}(G)$ from \S\ref{TmodulesNAGF} and set $\End(G):=\End_{\oK}(G)$. Proposition~\ref{P:fullyfaithful} implies the isomorphism of rings $\End(G) \cong \End_{\oK[t, \sigma]}(\sH_{G})$. Thus, by \eqref{E:homsdualpre}, we also obtain the isomorphism $\End(G)\otimes_{\bA}\FF_q(t) \cong \End_{\sT}(H_{G})$ of $\FF_q(t)$-vector spaces.
\item[(ii)] Let $G_0=(\GG_{a/\oK},\phi)$ be the Drinfeld $\bA$-module given in \eqref{E:DrinfeldDef}. By Drinfeld \cite{Dri74}, we know that $\End(G_0)$ can be identified by the image of the ring isomorphism 
\[
\End(G_0)\to \{c\in \CC_{\infty}|\ \ c\Lambda_{G_0}\subseteq \Lambda_{G_0}\}\subset \CC_{\infty}
\]
	sending $c=c_0+c_1\tau+\dots+c_{\ell}\tau^{\ell}$ to $\rd c=c_0$. We set $K_{\phi}\subset \CC_{\infty}$ to be the fraction field of $\End(G_0)$. Thus, by part (i), we have the ring isomorphism
$K_{\phi}\cong \End_{\sT}(H_{\phi})$ which was first established by Anderson.
\end{itemize}
\end{remark}

\section{The \texorpdfstring{$t$}{t}-module \texorpdfstring{$\mathcal{E}_n$}{En}}\label{S:tmoduleEn}
Throughout this section, we fix a Drinfeld $\bA$-module $G_0=(\GG_{a/\oK}, \phi)$ given by 
\[
\phi(t):=\theta + a_1\tau +\dots +a_r\tau^r\in \overline{K}[\tau], \ \ a_r\neq 0.
\]
Our goal in this section is to provide several properties of the $t$-module $\mathcal{E}_n$ defined below.  

\begin{definition}
	\begin{itemize}
		\item[(i)] Let $\mathcal{E}_0=(\mathbb{G}^{r-1}_{a/\overline{K}}, \varphi_0)$ be a $t$-module where  
	$
 \varphi_0(t):=\theta \Id_{r-1}+E_1\tau+E_2\tau^2
 $
 is so that the matrices $E_1\in \Mat_{r-1}(\oK)$ and $E_2\in \Mat_{r-1}(\oK)$ are given as 
	\begin{equation}\label{E:matrices20}
	E_1:=(-1)^{r-1}\begin{bmatrix}
	-a_{r-1}&a_r& &  \\
	\vdots& &\ddots&  \\
	-a_{2}&0&\cdots &a_r\\
	-a_{1}&0&\cdots &0
	\end{bmatrix}\text{ and }E_2:=\begin{bmatrix}
	0&\cdots&0\\
	\vdots	& &\vdots\\
	a_r&\cdots &0
	\end{bmatrix}.
	\end{equation}
		
		\item[(ii)] Let $n\in \mathbb{Z}_{\geq 1}$. Let $\mathcal{E}_n=(\mathbb{G}^{rn+r-1}_{a/\overline{K}}, \varphi_n)$ be a $t$-module where 
  \[
  \varphi_n(t):=\theta \Id_{rn+r-1}+N'+E'\tau
  \]
  and the matrices $N'\in \Mat_{rn+r-1}(\mathbb{F}_q)$ and $E'\in \Mat_{rn+r-1}(\oK)$ are defined as 
		
		\begin{equation}\label{E:matrices2}
		N':=\begin{bmatrix}
		0&\cdots &0&\bovermat{$rn-1$}{1& 0&\cdots & 0} \\
		& \ddots& & \ddots& \ddots & &\vdots \\
		& &\ddots& &\ddots & \ddots&0\\
		& &  & 0&\cdots&0 & 1\\
		& &  & & 0&\cdots& 0\\
		& &  & & &\ddots& \vdots\\
		& &  & & & & 0\\
		\end{bmatrix}\begin{aligned}
		&\left.\begin{matrix}
		\\
		\\
		\\
		\\
		\end{matrix} \right\} 
		rn-1\\ 
		&\left.\begin{matrix}
		\\
		\\
		\\
		\end{matrix}\right\}
		r
		\end{aligned}
		\end{equation}
		and
		\begin{equation}\label{E:matrices2A}
		E':=(-1)^{r-1}\begin{bmatrix}
		0&\cdots&\cdots & \cdots &\cdots & \cdots & 0\\
		\vdots& & & & & &\vdots\\
		0& & & & & & 0\\
		1& 0 & \cdots&\cdots&\cdots &\cdots &0\\
		-a_{r-1}&a_r&\ddots &  & & &\vdots\\
		\vdots& & \ddots& \ddots & & & \vdots\\
		-a_{1}&0&\cdots &a_r&0&\cdots& 0
		\end{bmatrix}\begin{aligned}
		&\left.\begin{matrix}
		\\
		\\
		\\
		\end{matrix} \right\} 
		rn-1\\ 
		&\left.\begin{matrix}
		\\
		\\
		\\
		\\
		\end{matrix}\right\}
		r\\
		\end{aligned}.
		\end{equation}
	\end{itemize}
\end{definition}

Let $\Phi_0:= \det(\Phi_\phi) (\Phi_\phi^{-1})^\tr$ denote the cofactor matrix of $\Phi_\phi$ defined in \eqref{E:sigmaactDM}, and for $n \in \ZZ_{>0}$, set $\Phi_n:= \Phi_0 (t-\theta)^n$. Since $\sH_{\phi}$ is a free $\overline{K}[t]$-module of rank $r$, $\smash{\wedge_{\oK[t]}^{r-1}\sH_{\phi}}$ is also free over $\oK[t]$ of rank $r$. Consider the set $\{f_1,\dots,f_r\} \subseteq \wedge_{\oK[t]}^{r-1}\sH_{\phi}$ given by 
	\[
	f_i:=(-1)^{r-i}n_1\wedge \cdots \wedge\sigma^{r-(i+1)} n_1\wedge\sigma^{r-(i-1)} n_1\wedge \cdots \wedge \sigma^{r-1}n_1
	\]
 for $2\leq i \leq r-1$, and $f_1:=(-1)^{r-1}n_1\wedge \sigma n_1\dots \wedge \sigma^{r-2}n_1$,  $f_r:=\sigma n_1\wedge \dots \wedge \sigma^{r-1}n_1$.
Since the left $\sigma$-action  on $\wedge_{\oK[t]}^{r-1}\sH_{\phi}$ is diagonal, we obtain $\sigma f_1=(-1)^{r-1}f_r$ and
	\[
	\sigma f_i=\frac{(-1)^{r-1}}{a_r^{(-r)}}(t-\theta)f_{i-1}+\frac{(-1)^{r-1}}{a_r^{(-r)}}a_{r-(i-1)}^{(-(r-(i-1)))}f_r, \ \ 2\leq i \leq r.
	\]
	Thus, we have
	\begin{equation}\label{E:extmap}
	\sigma \begin{bmatrix} f_r\\
	f_{r-1}\\\vdots\\f_2\\f_1\end{bmatrix}=\dfrac{(-1)^{r-1}}{a_r^{(-r)}}\begin{bmatrix}
	a_1^{(-1)} & (t-\theta) & 0 & \dots & 0\\
	a_2^{(-2)} & 0 & (t-\theta)&&\vdots\\
	\vdots&\vdots&\ddots&\ddots&0\\
	\vdots&\vdots&\ddots&\ddots&(t-\theta)\\
	a_r^{(-r)} & 0&\dots&\dots&0
	\end{bmatrix}\begin{bmatrix} f_r\\
	f_{r-1}\\\vdots\\f_2\\f_1\end{bmatrix} = \Phi_0\begin{bmatrix} f_r\\
	f_{r-1}\\\vdots\\f_2\\f_1\end{bmatrix}.
	\end{equation}
Moreover, for $n \geq 1$ we consider the tensor product $\wedge_{\oK[t]}^{r-1}\sH_{\phi}\otimes_{\oK[t]} \sH_{C^{\otimes n}}$ which clearly is also free over $\oK[t]$ of rank $r$ with basis $\{f_1\otimes \widetilde{n}, \dots, f_r\otimes \widetilde{n}\}$. Moreover, since $\sigma \widetilde{n} = (t-\theta)^n \widetilde{n}$, we have 
\begin{equation}\label{E:extmaptensC}
	\sigma \begin{bmatrix} f_r\otimes \widetilde{n}\\
	f_{r-1}\otimes \widetilde{n}\\\vdots\\f_2\otimes \widetilde{n}\\f_1\otimes \widetilde{n}\end{bmatrix}=\Phi_0(t-\theta)^n\begin{bmatrix} f_r\otimes \widetilde{n}\\
	f_{r-1}\otimes \widetilde{n}\\\vdots\\f_2\otimes \widetilde{n}\\f_1\otimes \widetilde{n}\end{bmatrix}
 = \Phi_n\begin{bmatrix} f_r\otimes \widetilde{n}\\
	f_{r-1}\otimes \widetilde{n}\\\vdots\\f_2 \otimes \widetilde{n}\\f_1\otimes \widetilde{n}\end{bmatrix}.
	\end{equation}

For the rest of the subsection, let $n$ be a non-negative integer.  Set $\rA \in \Mat_r(\ok[t])$ to be 
\[
\rA := \begin{bmatrix}
    a_1^{(-1)}(t-\theta)^{n} & a_r^{(-1)}(t-\theta)^{n+1}& a_{r-1}^{(-1)}(t-\theta)^{n+1}& \dots & \dots & a_2^{(-1)}(t-\theta)^{n+1}\\
    0 & 0& (t-\theta)^{n+1} & 0 & \dots & 0 \\
    \vdots & \vdots & 0 & (t-\theta)^{n+1} &\ddots & \vdots \\
    \vdots & \vdots &\vdots & \ddots & \ddots  &0\\
     0 & 0&0& \dots&0& (t-\theta)^{n+1}\\
     (t-\theta)^{n} &0&\dots&\dots&\dots&0
\end{bmatrix}.
\]
Recall from \S\ref{S:AFdual} that $\sH_{\mathcal{E}_n}=\Mat_{1 \times rn+r-1}(\oK[\sigma])$ is a $\oK[t,\sigma]$-module whose  $\oK[t]$-module structure is given by
\begin{equation}\label{E:action2}
ct \cdot h := ch\varphi_n(t)^{*}, \ \ c\in \oK, \ \ h\in \sH_{\mathcal{E}_n}.
\end{equation}
Note that $\{s_{1,rn+r-1},\dots,s_{rn+r-1,rn+r-1}\}$ forms a $\oK[\sigma]$-basis for $\sH_{\mathcal{E}_n}$. For $2\leq i \leq r$ and $n\in \ZZ_{\geq 0}$, we let $\ell_{i,n}:= s_{rn+i-1,rn+r-1}$. Moreover, let $\ell_{1,0}:=\sigma s_{r-1, r-1}$ when $n=0$ and $\ell_{1,n}:=s_{rn,rn+r-1}$ otherwise. Observe that the set $\{\ell_{1,n}, \dots, \ell_{r,n}\}$ forms a $\oK[t]$-basis of $\sH_{\mathcal{E}_n}$. 

Set $\gamma$ to be a $(q-1)$-st root of $\big((-1)^{r-1}/a_r^{(-r)}\big)^{-1}\in \oK$ and  $\widetilde{\gamma}$ to be a $(q-1)$-st root of $((-1)^{r-1}/a_r)^{-1}\in \oK$.

Consider the matrix 
\begin{equation}\label{E:CchangeofC}
\rB:= \gamma^{(1)}\cdot \begin{bmatrix}
1& 0&\dots &\dots & \dots & 0\\
& a_r^{(-1)}&a_{r-1}^{(-1)} & \dots&\dots & a_2^{(-1)}\\
&&\ddots&\ddots&&\vdots\\
&&&\ddots&\ddots&\vdots\\
&  &&&a_r^{(-r+2)}&a_{r-1}^{(-r+2)}\\
&  & &&&a_r^{(-r+1)}
\end{bmatrix}\in \GL_r(\oK).
\end{equation}

Let $\chi_0 \in \Mat_r(\oK)$ be the diagonal matrix 
with $\widetilde{\gamma}^{(-1)}$ in the first diagonal entry and $\widetilde{\gamma}$ in the remaining diagonal entries, and if $n\geq 1$, we let $\chi_n := \widetilde{\gamma} \Id_r \in \Mat_r(\oK)$. For a non-negative integer $n$, set 
\begin{equation}\label{E:basisEn}
\bsell :=[\ell_{1,n} \dots, \ell_{r,n}]^{\tr}
\end{equation}
and note that $\sigma \bsell =\chi_n^{(-1)} \rA\chi_n^{-1}\bsell$. Define  $\bsh_{\mathcal{E}_n}:=\rB\chi_n^{-1}\bsell=[h_1, \dots, h_r]^\tr$. Observe that we have $\Phi_n = \rB^{(-1)}\rA \rB^{-1}$ and $\sigma \bsh_{\mathcal{E}_n}=\Phi_n\bsh_{\mathcal{E}_n}$. Since the entries of $\bsh_{\mathcal{E}_n}$ form a $\oK[t]$-basis for $\sH_{\mathcal{E}_n}$, one can see by \eqref{E:extmap} and \eqref{E:extmaptensC} that for $n \in \ZZ_{\geq 0}$ as $\bA$-finite dual $t$-motives
\begin{equation}\label{E:exterisom}
\smash{\wedge_{\oK[t]}^{r-1}\sH_{\phi}\otimes_{\oK[t]} \sH_{C^{\otimes n}} \cong \sH_{\mathcal{E}_n}},
\end{equation}
and hence, from hereon, we identify $\smash{\wedge_{\oK[t]}^{r-1}\sH_{\phi}\otimes_{\oK[t]} \sH_{C^{\otimes n}}}$ with $\sH_{\mathcal{E}_n}$. 
By using the $\oK[t]$-module action on $\sH_{\cE_n}$ given in \eqref{E:action2}, we have 
\begin{equation}\label{basistnsigma2}
\begin{split}
&(t-\theta)^{n-k} \ell_{i,n}=s_{rk+i-1,rn+r-1} \text{ for } 1\leq k \leq n,\text{ and } 1\leq i \leq r,\\
&(t-\theta)^n\ell_{j,n}=s_{j-1,rn+r-1} \text{ for } 2\leq j \leq r. 
\end{split}
\end{equation}

 \begin{remark}\label{R:aspextC} 
Let $n$ be a non-negative integer. By direct computation, one can show that the top coefficient of $\varphi_n(t^{rn+r-1})$ is an invertible lower triangular matrix and therefore, the $t$-module $\cE_n$ is \emph{almost strictly pure} in the sense of \cite{NPapanikolas21}.
  \end{remark}

\subsection{The Galois group of \texorpdfstring{$\mathcal{E}_n$}{En}}\label{S:GGofEn}
Let $n$ be a non-negative integer and set $\sH_n:= \sH_{\cE_n}$. The $\bA$-finite dual $t$-motive $\sH_n$ has a rigid analytic trivialization given by $\Psi_n := \Psi_0\Omega^{n}\in \GL_r(\TT)$,
where $\Psi_0:= \det(\Psi_\phi) (\Psi_\phi^{-1})^\tr$ denotes the cofactor matrix of $\Psi_\phi$ defined in Example~\ref{Ex:tmotive} (i). Hence, the pre-$t$-motive $H_{n} := \oK(t)\otimes_{\oK[t]}\sH_{n} $ associated to $\sH_{n}$ forms a $t$-motive. We set $\rL_{n}:=\End_{\sT}(H_n)$ and $\rL_{\phi}:=\End_{\sT}(H_{\phi})$. 

\begin{proposition}\label{P:equalend2} 
We have $\End_{\oK[t, \sigma]}(\sH_{\phi}) = \End_{\oK[t, \sigma]}(\sH_{n})$ and $\rL_{\phi}= \rL_n$. In particular, $\rL_n$ forms a field.
\end{proposition}

\begin{proof}
	Let $g\in \End_{\oK[t, \sigma]}(\sH_{n})$ and $\rG\in \Mat_{r}(\oK[t])$ be such that $[g(h_1),\dots,g(h_r)]^{{\tr}}= \rG\bsh_{\mathcal{E}_n}$. Since $\sigma g = g\sigma$, we obtain $\rG^{(-1)} \Phi_n = \Phi_n \rG$ . This implies that  $(\rG^{\tr})^{(-1)} \Phi_{\phi} = \Phi_{\phi} \rG^{\tr}$.  By a slight abuse of notation, we see that $G^{\tr}$ defines an element in $\End_{\oK[t, \sigma]}(\sH_{\phi})$ such that $[g(n_1),\dots,g(n_r)]^{{\tr}}= \rG^{\tr}\bsn$. The other inclusion can be obtained similarly. The second equality can be proved by using the same argument and the last assertion is a consequence of the second equality and Remark \ref{R:isomorphismsofvss}.
\end{proof}

We call a $t$-motive $H$ \textit{simple} if it is simple as a left $\oK(t)[\sigma, \sigma^{-1}]$-module. 

\begin{proposition}\label{P:simpletensor2}
	For any $n\geq 0$, the $t$-motive $H_n$ is simple.
\end{proposition}

\begin{proof}
	Consider the dual $H_{\phi}^{\vee}= \Hom_{\oK(t)}(H_{\phi}, \textbf{1}_{\sP})$ of the $t$-motive $H_{\phi}$. We claim that $H_{\phi}^{\vee}$ is a simple left $\oK(t)[\sigma, \sigma^{-1}]$-module. Let $\mathcal{U}$ be a left $\oK(t)[\sigma, \sigma^{-1}]$-submodule of $H_{\phi}^{\vee}$ and set $\mathcal{U}^{\perp}:= \{h \in H_{\phi} \, |\, f(h)=0 \, \text{for all} \, f \in \mathcal{U}\}$. Then, it follows that $\mathcal{U}^{\perp}$ is a left $\oK(t)[\sigma, \sigma^{-1}]$-submodule of $H_{\phi}$. Since, by \cite[Cor.~3.3.3]{CP12}, $H_{\phi}$ is simple, we have that $\mathcal{U}^{\perp}$ is either $\{0\}$ or equal to $H_{\phi}$. By using the definition of $\mathcal{U}^{\perp}$, we see that $\mathcal{U}=H_{\phi}^{\vee}$  if $\mathcal{U}^{\perp}=\{0\}$ and $\mathcal{U}$ is trivial  if $\mathcal{U}^{\perp}=H_{\phi}$. This proves the claim. Let $\sC_n$ be the $\bA$-finite dual $t$-motive associated  to the $n$-th tensor power of the Carlitz module $\mathbf{C}^{\otimes n}=(\mathbb{G}^n_{a/K},C^{\otimes n})$ given in Example~\ref{Ex:0} and define $C_n := \oK(t)\otimes_{\oK[t]}\sC_{n} $ to be the $t$-motive associated to $\sC_{n}$. One notes that 
	\begin{equation}\label{E:relation}
	H_n \cong (H_{\phi}^{\vee}\otimes C_1) \otimes C_{n}
	\end{equation}
	as $\oK(t)[\sigma,\sigma^{-1}]$-modules. By \cite[$\S$3.2.11]{P08}, we have that $C_n\otimes C_n^{\vee} \cong \textbf{1}_{\sP}$ as pre-$t$-motives  and the functor $\mathfrak{F}: \sP \rightarrow \sP$ defined by 
	$
	P \mapsto P \otimes C_{n+1}^{\vee}
	$
	yields an equivalence of categories. On the other hand, by using the construction of $\sH_{n}$, one can see that $\mathfrak{F}(H_n) = H_{\phi}^{\vee}$. \par
	
	Let $J$ be a submodule of $H_n$ and let $\iota_J:J \rightarrow H_n$ be the inclusion monomorphism. Then, we have a left $\oK(t)[\sigma, \sigma^{-1}]$-module homomorphism $\mathfrak{F}(\iota_J):\mathfrak{F}(J) \rightarrow H_{\phi}^{\vee}$. Since the image $\im\mathfrak{F}(J)$ of $\mathfrak{F}(J)$ under the map $\mathfrak{F}(\iota_J)$ is a submodule of $H_{\phi}^{\vee}$ and $H_{\phi}^{\vee}$ is a simple left $\oK(t)[\sigma, \sigma^{-1}]$-module, $\im \mathfrak{F}(J)$ is either $0$ or $H_{\phi}^{\vee}$. If $\im \mathfrak{F}(J)=0$, then since $\mathfrak{F}$ is fully faithful, $\iota_J$ is the zero map which implies that $J=0$. If $\im \mathfrak{F}(J)= H_{\phi}^{\vee}$, then $\mathfrak{F}(\iota_J)$ is an isomorphism. Since $\mathfrak{F}$ is fully faithful, this implies that $\iota_J$ is an isomorphism. Hence, $J=H_n$. 
\end{proof}

Let $g\in \rL_n$ and $\rG\in \Mat_{r}(\oK(t))$ be such that $[g(h_1),\dots,g(h_r)]^{{\tr}}= \rG\bsh_{\mathcal{E}_n}$ and define 
\begin{equation}\label{E:fixedPsi}
\mathcal{F}_{g}:=(\Psi_n)^{-1}\rG\Psi_n\in \Mat_{r}(\LL).
\end{equation}
Note that $\sigma$ fixes $\mathcal{F}_{g}$, and thus, by \cite[Lem. 3.3.2]{P08}, we have $\mathcal{F}_{g} \in \Mat_{r}(\FF_q(t))$. Recall the $\mathbb{F}_q(t)$-vector space $(H_n)^{\text{Betti}}$ given in \S\ref{SS:t-motives}. By \cite[Prop. 3.3.8]{P08} we obtain $\End((H_n)^{\text{Betti}})=\Mat_{r}(\FF_q(t))$. Hence, we have the following injective map
$
\rL_n \to \Mat_r(\mathbb{F}_q(t)),
$ defined by 
\[
g \mapsto \mathcal{F}_{g}.
\]
Recall the Galois group $\Gamma_{H_n}$ of $H_n$. Setting $\rR$ to be an $\FF_q(t)$-algebra and using the functoriality of the representation $\mu: \Gamma_{ H_n} \rightarrow \GL((H_n)^{\text{Betti}})$ in $(H_n)^{\text{Betti}}$ (\cite[\S 3.5.2]{P08}),  for any $\gamma \in \Gamma_{H_n}(\rR)$ we have the commutative diagram (cf. \cite[$\S$3.2]{CP11}):
\begin{equation} \label{E:Com2}
\begin{tikzcd}
\rR \otimes_{\FF_q(t)} H_n^{\text{Betti}} \arrow{r}{\mu(\gamma)} \arrow{d}{1\otimes \mathcal{F}_{g}} & \rR \otimes_{\FF_q(t)} H_n^{\text{Betti}} \arrow{d}{1 \otimes \mathcal{F}_{g}} \\
\rR \otimes_{\FF_q(t)} H_n^{\text{Betti}} \arrow{r}{\mu(\gamma)} & \rR \otimes_{\FF_q(t)} H_n^{\text{Betti}}
\end{tikzcd}.
\end{equation}

Before we give the main result of this subsection, we note that by \cite[Thm. 4.5.10]{P08}, $\rL_{\phi}$ can be embedded into $\End(H_{\phi}^{\text{Betti}})=\Mat_{r}(\FF_q(t))$ and set $\bs:=[\rL_{\phi}:\mathbb{F}_q(t)]$. We further define the centralizer $\Cent_{\GL_r/\FF_q(t)}$ to be the algebraic group over $\mathbb{F}_q(t)$ such that 
\[
\Cent_{\GL_r/\FF_q(t)}(\rL_n)(\rR):= \{\gamma \in \GL_r(\rR) \mid \gamma g = g \gamma \, \, \, \forall \, g\in \rR\otimes_{\FF_q(t)} \rL_n \subseteq \Mat_r(\rR)\}.
\]
Note that by \eqref{E:Com2} we have $\Gamma_{H_n} \subseteq \Cent_{\GL_r/\FF_q(t)}(\rL_n)$. Indeed we can obtain more about the structure of $\Gamma_{H_n}$ in the following theorem, whose proof has been shortened compared to the previous version, thanks to the suggestion of the referee.

\begin{theorem}\label{T:equalitygamma2}
	Let $[K_{\phi}:K]=\bs$. We have $\Gamma_{H_n} = \Gamma_{H_{\phi}}=\Cent_{\GL_r/\FF_q(t)}(\rL_n)$. In particular, 
	\[
	\dim \Gamma_{H_n}=\trdeg_{\oK}\oK\left(\bigcup_{i =1}^r\bigcup_{j=0}^{r-1}\widetilde{\pi}^{n+1} F_{G_0,\delta_j}(\lambda_i)\right)=r^2/\bs.
	\]
\end{theorem}

\begin{proof}
	By Proposition~\ref{P:equalend2} and \cite[Thm.~3.5.4]{CP12} we see that 
	\[\Gamma_{H_n}\subseteq \Cent_{\GL_r/\FF_q(t)}(\rL_n) = \Gamma_{H_{\phi}}.\]
	Recall the period matrix $\mathcal{P}$ introduced in \eqref{E:periodmat} for the Drinfeld $\bA$-module $G_0=(\GG_{a/\oK},\phi)$ given in \eqref{E:DrinfeldDef}.
Using \eqref{E:quasiperiod}, \eqref{E:periodLeg}, and \cite[(3.4.3)]{CP12}, we see that 
\[
\det \Upsilon(\theta)=(-1)^{r-1}\det \mathcal{P}=(-1)^{r-1}\frac{\tilde{\pi}}{\xi}
\]
and
\[
\oK(\Psi_n(\theta))=\oK(\tilde{\pi}^{-n-1}\Upsilon(\theta))\subset \oK(\Upsilon(\theta),\tilde{\pi})=\oK(\Upsilon(\theta))=\oK(\mathcal{P})=\oK(\mathcal{P},\tilde{\pi}) \supset \oK(\tilde{\pi}^{n+1}\mathcal{P}).
\]
Since these inclusions are algebraic extensions, by \cite[Thm.~3.5.4]{CP12} and Theorem \ref{T:TannakianMain}, it follows that
\[
\dim \Gamma_{H_n}=\trdeg_{\oK}\oK(\Psi_n(\theta))=\trdeg_{\oK}\oK(\tilde{\pi}^{n+1}\mathcal{P})=\trdeg_{\oK}\oK(\mathcal{P})=r^2/\bs.
\]    
Since $\Cent_{\GL_r/\FF_q(t)}(\rL_n)$ is connected, we have the equality $\Gamma_{H_n} = \Gamma_{H_{\phi}}$.
\end{proof}

\subsection{Endomorphisms of \texorpdfstring{$\mathcal{E}_n$}{En}}\label{S:EndEn} 
To analyze endomorphisms of $\mathcal{E}_n$, we start by introducing hyperderivatives. Setting $\CC_{\infty}((t))$ to be the field of formal Laurent series of $t$ with coefficients in $\CC_{\infty}$, we define \textit{the $j$-th hyperdifferential operator} $\partial^{j}_{t}:\CC_{\infty}((t))\to \CC_{\infty}((t))$ by
\[
\partial^{j}_{t}\bigg(\sum_{i\geq i_0}c_it^i\bigg):=\sum_{i\geq i_0}c_i\binom{i}{j}t^{i-j}\text{,  }\ \ \ c_i\in \CC_{\infty},
\]
where $\binom{i}{j}$ denotes the usual binomial coefficient but modulo $p$. 
For any element $C=(C_{ik})\in \Mat_{m\times \ell}(\CC_{\infty}((t)))$, we also set $\partial^{j}_{t}(C):=(\partial^{j}_{t}(C_{ik}))$.

 Our aim in this subsection is to explicitly construct an isomorphism 
\begin{equation*}\label{E:isomorphism}
\End(G_0) \cong \End(\mathcal{E}_n).
\end{equation*} of rings. We start with stating a proposition which can be proven by using the ideas of the proof of \cite[Prop. 4.1.1]{CP12}.

\begin{proposition}\label{P:entriesmor}
	Given $g\in \rL_n $, let $\rG = (\rG_{ij}) \in \Mat_r(\oK(t))$ satisfy $[g(h_1),\dots,g(h_r)]^{{\tr}}= \rG\bsh_{\mathcal{E}_n}$. Then  $\rG_{21}(\theta) =\dots = \rG_{r1}(\theta)= \rG_{12}(\theta)=\dots = \rG_{1r}(\theta)  = 0$.
\end{proposition}

\begin{proof}
	Since, as in the proof of Proposition~\ref{P:equalend2}, each $\rG^{\tr}$ produces an element $g$ in $\rL_{\phi}$, the part  $\rG_{12}(\theta)=\dots = \rG_{1r}(\theta)  = 0$ follows from \cite[Prop.~4.1.1(b)]{CP12}. By \cite[Prop.~4.1.1(a)]{CP12}, $\rG_{j1}$ is regular at $t=\theta$ for each $1\leq j \leq r$. Note that by the definitions of  $\Psi_n$ and $\Psi_{\phi}$, we have 
	\begin{equation}\label{E:rG}
	V^{\tr}\Upsilon^{\tr}\mathcal{F}_{g} =\rG V^{\tr}\Upsilon^{\tr},
	\end{equation}
 where $\mathcal{F}_g$ is defined in \eqref{E:fixedPsi}. Recall the fixed $\bA$-basis $\{\lambda_1,\dots,\lambda_r\}$ for $\Lambda_{G_0}$ given in \S\ref{SS:bids}. Specializing  both sides of \eqref{E:rG} at $t=\theta$ and using \eqref{E:quasiperiod} we obtain for $2 \leq \ell \leq r$
	\begin{equation}\label{E:etaell}
	\begin{split}
	&-\sum_{i=1}^r (\mathcal{F}_{g})_{i1}(\theta)(a_\ell^{(-\ell+1)}F_{G_0,\delta_1}(\lambda_i)+a_{\ell+1}^{(-\ell+1)}F_{G_0,\delta_2}(\lambda_i)+\dots+a_r^{(-\ell+1)}F_{G_0,\delta_{r-\ell+1}}(\lambda_i))\\
	&  = -\rG_{\ell 1}(\theta)\lambda_1 + (\rG_{\ell 2}(\theta)a_2^{(-1)}+\dots + \rG_{\ell r}(\theta)a_r^{(-r+1)})F_{G_0,\delta_1}(\lambda_1)\\
	&\ \ \  +(\rG_{\ell 2}(\theta)a_3^{(-1)}+\dots + \rG_{\ell (r-1)}(\theta)a_r^{(-r+2)})F_{G_0,\delta_2}(\lambda_1) +\cdots \\
	&\ \ \  + (\rG_{\ell 2}(\theta)a_{r-1}^{(-1)}+ a_r^{(-1)}\rG_{\ell 3}(\theta)a_r^{(-2)})F_{G_0,\delta_{r-2}}(\lambda_1)+ \rG_{\ell 2}(\theta)a_r^{(-1)}F_{G_0,\delta_{r-1}}(\lambda_1).
	\end{split}
	\end{equation}
	By Theorem~\ref{T:equalitygamma2}, we see that $\{\lambda_1\}$ along with $\cup_{j=1}^{r-1}\{F_{G_0,\delta_j}(\pi_1), \dots, F_{G_0,\delta_j}(\pi_{r/\bs}))\}$ for any maximal $K_0$-linearly independent subset $\{\pi_1, \dots, \pi_{r/\bs}\}$ of $\Lambda_{G_0}$ are linearly independent over $\oK$. It follows that the coefficients of $\lambda_1$ for each $2 \leq \ell \leq r$ in \eqref{E:etaell} are zero. Hence, we have $\rG_{21}(\theta) = \dots = \rG_{r1}(\theta) = 0$. This completes the proof. 
\end{proof}

Recall the dual $t$-motive $\sH_{{\phi}}$ of the Drinfeld $\bA$-module $G_0=(\GG_{a/\oK},\phi)$ from Example \ref{Ex:tmotive} and note that $\{n_1\}$ is a $\oK[\sigma]$-basis for $\sH_{\phi}$ \cite[Sec. 1.5.6]{BP20}. By Proposition \ref{P:fullyfaithful}, any $\eta \in \End(G_0)$ defines a morphism $\epsilon_{\eta}\in \End_{\oK[t, \sigma]}(\sH_{{\phi}})$ given by $\epsilon_{\eta}(n) = n \eta^{*}$ for all $n \in \sH_{{\phi}}$. Let $\rF=(\rF_{ij}) \in \Mat_{r}(\oK[t])$ be such that $[\epsilon_{\eta}(n_1),\dots,\epsilon_{\eta}(n_r)]^{{\tr}}= \rF\bsn$. Thus, we have 
\begin{equation}\label{E:actiondrinfelfnor}
\epsilon_{\eta}(n_1)=\epsilon_{\eta}(1)=[\rF_{11},\dots,\rF_{1r}]\bsn=\eta^{*}.
\end{equation}

By Proposition \ref{P:fullyfaithful} and Proposition~\ref{P:equalend2}, there exist a unique element $\rP_{\eta}\in \End(\mathcal{E}_n)$ and a map  $\epsilon_{\rP_{\eta}}\in \End(\sH_n)$ such that
$[\epsilon_{\rP_{\eta}}(h_1),\dots,\epsilon_{\rP_{\eta}}(h_r)]^{{\tr}}= \rF^{\tr}\bsh_{\mathcal{E}_n}$ and 
\begin{equation}\label{E:acthom22}
\begin{bmatrix}
\epsilon_{\rP_{\eta}}(s_{1,rn+r-1})\\ \vdots\\\vdots\\ \epsilon_{\rP_{\eta}}(s_{rn+r-1,rn+r-1})\end{bmatrix} =\begin{bmatrix}s_{1,rn+r-1}\rP_{\eta}^{*}\\ \vdots\\\vdots\\s_{rn+r-1,rn+r-1}\rP_{\eta}^{*}
\end{bmatrix}.
\end{equation}

For the matrix $\rB$ defined in \eqref{E:CchangeofC}, let $(\rB^{-1}\rF^{\tr})'\in \Mat_{r-1\times r}(\oK[t])$ be the matrix formed after omitting the first row of $\rB^{-1}\rF^{\tr}$. By using \eqref{basistnsigma2} and noting that $\bsh_{\mathcal{E}_n}=\rB\chi_n^{-1} \bsell$, we obtain
\begin{equation}\label{E:actionhom2}
\begin{bmatrix}
\epsilon_{\rP_{\eta}}(s_{1,rn+r-1}) \\ \epsilon_{\rP_{\eta}}(s_{2,rn+r-1}) \\ \vdots \\ \vdots\\\epsilon_{\rP_{\eta}}(s_{rn+r-2,rn+r-1})\\ \epsilon_{\rP_{\eta}}(s_{rn+r-1,rn+(r-1)})
\end{bmatrix} = 
\begin{bmatrix}
0&\dots&\dots&\dots&0& ((\rB\chi_n^{-1})^{-1}\rF^{\tr})'\rB\chi_n^{-1}(t-\theta)^n\\
\vdots& &&&\vdots& (\rB\chi_n^{-1})^{-1}\rF^{\tr}\rB\chi_n^{-1}(t-\theta)^{n-1}\\
\vdots&& &&\vdots& \vdots\\
\vdots&&& &\vdots& (\rB\chi_n^{-1})^{-1}\rF^{\tr}\rB\chi_n^{-1}(t-\theta)\\
0&\dots&\dots&\dots&0& (\rB\chi_n^{-1})^{-1}\rF^{\tr}\rB\chi_n^{-1}\\
\end{bmatrix}\begin{bmatrix}
s_{1,rn+r-1} \\ s_{2,rn+r-1} \\ \vdots \\ \vdots\\s_{rn+r-2,rn+r-1}\\ s_{rn+r-1,rn+r-1}
\end{bmatrix}.
\end{equation}
By using the $\oK[t]$-module action on $\sH_{n}$ given in \eqref{E:action2},  for each $i\geq 0$ and $1\leq j\leq rn+r-1$ the $\ell$-th entry of $\rd(t^i\cdot s_{j,rn+r-1})\in \Mat_{1 \times rn+r-1}(\oK)$ is equal to 
\begin{equation}\label{E:actionmor2}
\begin{cases} 
\pd_t^k(t^i)|_{t=\theta}, &  \text{if}\, \, \ell =j-rk  \\
0, &\text{if} \, \,  \ell \neq j-rk.
\end{cases}
\end{equation}
Let us set $\bT:=\rB^{-1}\rF^{\tr} \rB\in \Mat_r(\oK[t])$. Let $\rS'$ denote the matrix formed by throwing out the first column of $\bT$ and let $\bT'$ be the matrix formed by the lower right $(r-1)\times (r-1)$ block of $\bT$. Using \eqref{E:actionmor2}, one can compare \eqref{E:acthom22} and \eqref{E:actionhom2} to calculate 
\begin{equation}\label{E:dEnd2}
\rd\rP_{\eta} = 
\left.\begin{bmatrix}
(\bT')^{\tr} & \pd_t^{1}(\rS')^{\tr} & \dots & \dots & \dots & \pd_{t}^{n-1}(\rS')^{\tr} & \pd_t^n(\rS')^{\tr}\\
& \bT^{\tr} & \pd_t^1(\bT)^{\tr} & \dots & \dots & \pd_t^{n-2}(\bT)^{\tr}& \pd_t^{n-1}(\bT)^{\tr}\\
&&\ddots&\ddots&\ddots&\vdots&\pd_t^{n-2}(\bT)^{\tr}\\
&&&\ddots&\ddots&\vdots&\vdots\\
&&&&\ddots&\vdots&\vdots\\
&&&&&\bT^{\tr}&\pd_t^1(\bT)^{\tr}\\
&&&&&&\bT^{\tr}
\end{bmatrix}\right|_{t=\theta} \in \Mat_{rn+r-1}(\oK).
\end{equation}
Moreover, if $\rF^{\tr} := (\rF'_{i,j})$ then by Proposition \ref{P:entriesmor} and the fact that $\rF_{11}'(\theta)=\rF_{11}(\theta)=\rd\eta$, we can write $\bT^{\tr}$ 
of the form
\begin{equation}\label{E:W2}
\bT^{\tr}_{|t=\theta}=\begin{bmatrix}
\rd\eta & 0 & \dots & 0\\
0 & * &\dots& *\\
\vdots &\vdots&\ddots&\vdots\\
0 & * &\dots& *
\end{bmatrix}.
\end{equation}

\begin{proposition}\label{P:rings2}
	The map $\End(G_0) \to \End(\mathcal{E}_n)$ given by 
	$
	\eta \rightarrow \rP_{\eta}
	$ for each $\eta \in \End(G_0)$ is an isomorphism of rings and $\rd \rP_{\eta}$ is of the form \eqref{E:dEnd2}. Moreover, $\End(\mathcal{E}_n)$ can be identified with $\{\rd\rP \ \ |  \rP \in \End(\mathcal{E}_n)\}$ which is an integral  domain.
\end{proposition}
\begin{proof} We start with proving the first part for $n\in \ZZ_{\geq 0}$. The injectivity of the map in the first assertion follows from the above discussion and the surjectivity can be also seen by using Proposition \ref{P:fullyfaithful} and Proposition \ref{P:equalend2}. We now prove the second assertion for $n\in \ZZ_{\geq 1}$. By the isomorphism given in Remark \ref{R:isomorphismsofvss}, the set $\{\rd\eta \ \ | \eta \in \End(G_0)\}$ forms an integral domain. Furthermore one can see, by using Proposition \ref{P:fullyfaithful}, \eqref{E:W2}, and the first part of the proposition, that the map $ \{\rd\eta \ \ | \eta \in \End(G_0)\}
	\to \{\rd \rP\ \  | \rP \in \End(\mathcal{E}_n)\}$ given by $\rd\eta\mapsto \rd \rP_{\eta}$ is a ring isomorphism and hence, $\{\rd\rP \ \ |  \rP \in \End(\mathcal{E}_n)\}$ is an integral domain. Finally, we prove the second assertion for $\mathcal{E}_0$. Consider the map $ \{\rd\eta \ \ | \eta \in \End(G_0)\}
	\to \{\rd \rP\ \  | \rP \in \End(\mathcal{E}_0)\}$ given by $\rd\eta\mapsto \rd \rP_{\eta}$. Proposition \ref{P:fullyfaithful} implies that it is a ring homomorphism. For injectivity, suppose that $\rd \rP_{\eta}=0$. Uniqueness of the exponential map implies the following equality
	\[
	\rP_{\eta}\Exp_{\mathcal{E}_n}=\Exp_{\mathcal{E}_n}\rd\rP_{\eta}
	\]
	which holds in $\Mat_{r-1}(\oK)[[\tau]]$. This implies that $\rP_{\eta}=0$. Thus, by the first part, we have $\eta=0$ and hence, $\rd\eta=0$. For the surjectivity, note from the first part of the proposition that for any $\rP\in \End(G_0)$, there exists a unique $\eta\in \End(G_0)$ so that $\rP=\rP_{\eta}$ and therefore, $\rd \eta$ becomes the preimage of $\rd\rP$.
\end{proof}

\section{The \texorpdfstring{$t$}{t}-module \texorpdfstring{$G_n$}{Gn}}\label{S:tmoduleGn}
Let $G_0=(\GG_{a/\oK}, \phi)$ be the Drinfeld $\bA$-module of rank $r\geq 2$ defined as in \eqref{E:DrinfeldDef}. Our goal  is to make a similar analysis as in \S\ref{S:tmoduleEn} for $t$-modules constructed by the tensor product of $G_0$ with Carlitz tensor powers (see \ref{E:tensisom}). The main object of this section can be defined as follows.

\begin{definition}For any positive integer $n$, we define the $t$-module $G_n:=(\mathbb{G}_{a/K}^{rn+1},\phi_n)$, where $\phi_n:\bA\to \Mat_{rn+1}(K)[\tau]$ is given by $\phi_n(t):=\theta \Id_{rn+1}+N+E\tau$ so that $N\in \Mat_{rn+1}(\mathbb{F}_q)$ and $E\in \Mat_{rn+1}(K) $ are  defined as

\begin{equation}\label{E:matrices}
N:=\begin{bmatrix}
0&\cdots &0&\bovermat{$rn+1-r$}{1&0& \cdots & 0}  \\
& \ddots& & \ddots& \ddots & &\vdots \\
& &\ddots& &\ddots & \ddots&0\\
& &  & 0&\cdots&0 & 1\\
& &  & & 0&\cdots& 0\\
& &  & & &\ddots& \vdots\\
& &  & & & & 0\\
\end{bmatrix}\begin{aligned}
&\left.\begin{matrix}
\\
\\
\\
\\
\end{matrix} \right\} 
rn+1-r\\ 
&\left.\begin{matrix}
\\
\\
\\
\end{matrix}\right\}
r\\
\end{aligned}
\end{equation}
and 
\begin{equation}\label{E:matricesA}
E:=\begin{bmatrix}
0&\cdots&\cdots & \cdots &\cdots & \cdots & 0\\
\vdots& & & & & &\vdots\\
0& & & & & & 0\\
1&0 &\cdots &\cdots &\cdots &\cdots &0\\
0&\ddots& \ddots&  & & &\vdots\\
\vdots& &1& \ddots & & & \vdots\\
a_1&\cdots&\cdots &a_r&0&\cdots& 0
\end{bmatrix}\begin{aligned}
&\left.\begin{matrix}
\\
\\
\\
\end{matrix} \right\} 
rn+1-r\\ 
&\left.\begin{matrix}
\\
\\
\\
\\
\end{matrix}\right\}
r\\
\end{aligned}.
\end{equation}
\end{definition}

Set $\Phi_n^{\tens}:= \Phi_\phi (t-\theta)^n$. For $n \geq 1$, we consider the tensor product $\smash{\sH_{\phi}\otimes_{\oK[t]} \sH_{C^{\otimes n}}}$ which is free over $\oK[t]$ of rank $r$ with basis $\{n_1\otimes \widetilde{n}, \dots, n_r\otimes \widetilde{n}\}$. Moreover, since $\sigma \widetilde{n} = (t-\theta)^n \widetilde{n}$, by \eqref{E:sigmaactDM} we have 
\begin{equation}\label{E:DMtensC}
	\sigma \begin{bmatrix} n_1\otimes \widetilde{n}\\
	n_2\otimes \widetilde{n}\\\vdots\\n_{r-1}\otimes \widetilde{n}\\n_r\otimes \widetilde{n}\end{bmatrix}=\Phi_\phi(t-\theta)^n\begin{bmatrix} n_1\otimes \widetilde{n}\\
	n_2\otimes \widetilde{n}\\\vdots\\n_{r-1}\otimes \widetilde{n}\\n_r\otimes \widetilde{n}\end{bmatrix}
 = \Phi_n^{\tens}\begin{bmatrix} n_1\otimes \widetilde{n}\\
	n_2\otimes \widetilde{n}\\\vdots\\n_{r-1}\otimes \widetilde{n}\\n_r\otimes \widetilde{n}\end{bmatrix}.
	\end{equation}

We consider the $\oK[t,\sigma]$-module $\sH_{G_n}=\Mat_{1\times rn+1}(\oK[\sigma])$ whose $\oK[t]$-module structure, by \S\ref{S:AFdual}, is given by
\begin{equation}\label{E:action}
ct \cdot h :=c h\phi_n(t)^{*}, \ \ c\in \oK, \ \ h\in \sH_{G_n}.
\end{equation}
Note that the vectors $\{s_{1,rn+1},\dots,s_{rn+1,rn+1}\}$ forms a $\oK[\sigma]$-basis for $\sH_{G_n}$. For $1\leq i \leq r$, we let $\ell_i^{\tens}:= s_{r(n-1)+i+1,rn+1}$. Observe that the set $\{\ell_1^{\tens}, \dots, \ell_r^{\tens}\}$ forms a $\oK[t]$-basis of $\sH_{n}^{\tens}$. Define the matrix $B^{\tens}\in \GL_{r}(\oK)$ so that its inverse is given by
\begin{equation}\label{E:BchangeofB}
(\rB^{\tens})^{-1} = \begin{bmatrix}
0 & a_2^{(-1)}&a_3^{(-2)} &\dots & a_{r-1}^{(-r+2)} & a_r^{(-r+1)}\\
0 & a_3^{(-1)}&a_4^{(-2)} & \dots& a_r^{(-r+2)}&\\
\vdots&\vdots&&\reflectbox{$\ddots$}&&\\
\vdots&\vdots&\reflectbox{$\ddots$}&&&\\
0 & a_r^{(-1)} &&&\\
1 &  & &&&
\end{bmatrix}.
\end{equation}
Set $\bsell^{\tens} :=[\ell_1^{\tens}, \dots, \ell_r^{\tens}]^{\tr}$ and define $\bsh_{G_n} = \rB^{\tens} \bsell^{\tens}=[h_1^{\tens},\dots,h_r^{\tens}]$. If $\rA^{\tens}\in\Mat_r(\oK[t])$ represents the $\sigma$-action on $\sH_{G_n}$, that is if $\sigma\bsell^{\tens} = \rA^{\tens} \bsell^{\tens}$, then 
\begin{equation*}\label{E:Ataction}
\rA^{\tens} = \begin{bmatrix}
0&\dots&\dots&0& \frac{-a_1^{(-1)}}{a_r^{(-1)}}(t-\theta)^n & (t-\theta)^{n+1}\\
(t-\theta)^n &\ddots& &\vdots& \frac{-a_2^{(-1)}}{a_r^{(-1)}}(t-\theta)^n &0 \\
&(t-\theta)^n &\ddots&\vdots&\frac{-a_3^{(-1)}}{a_r^{(-1)}}(t-\theta)^n&\vdots\\
&&\ddots&0&\vdots&\vdots\\
&&&(t-\theta)^n&\frac{-a_{r-1}^{(-1)}}{a_r^{(-1)}}(t-\theta)^n&\vdots\\
&&&&\frac{(t-\theta)^n}{a_r^{(-1)}}&0
\end{bmatrix}.
\end{equation*}
Observe that $\Phi_n^{\tens} =(\rB^{\tens})^{(-1)}\rA^{\tens} (\rB^{\tens})^{-1}$ and $\sigma \bsh_{G_n} =\Phi^{\tens}_n\bsh_{G_n} $. Since the entries $h_1^{\tens},\dots,h_r^{\tens}$ of $\bsh_{G_n} $ form an $\oK[t]$-basis for $\sH_{G_n}$, one can see by \eqref{E:DMtensC} that 
\begin{equation}\label{E:tensisom}
\sH_{G_n}\cong \sH_{\phi}\otimes_{\oK[t]} \sH_{C^{\otimes n}}
\end{equation}
and hence, we identify $\sH_{\phi}\otimes_{\oK[t]} \sH_{C^{\otimes n}}$ with $\sH_{G_n}$. By using the $\oK[t]$-module action on $\sH_{n}^{\tens}$ given in \eqref{E:action}, we have
\begin{equation}\label{basistnsigma}
\begin{split}
&(t-\theta)^{n-k} \ell_{i}^{\tens}=s_{r(k-1)+i+1,rn+1} \text{ for } 1\leq k \leq n,\text{ and } 1\leq i \leq r,\\
&(t-\theta)^n\ell_r^{\tens}=s_{1,rn+1}. 
\end{split}
\end{equation}
\begin{remark}\label{R:aspC} 
Let $n\geq 1$. By direct computation, one can show that the top cofficient of $\phi_n(t^{rn+1})$ is an invertible lower triangular matrix and therefore, the $t$-module $G_n$ is \emph{almost strictly pure} in the sense of \cite{NPapanikolas21}.
  \end{remark}

\subsection{The Galois group of \texorpdfstring{$G_n$}{Gn}}\label{S:GGofGn}
Let $n$ be a positive integer and set $\sH_n^{\tens}:= \sH_{G_n}$.  For any $n\in \ZZ_{\geq 0}$, define the matrix 
	\begin{equation*}
	\Psi_{n}^{\tens}:=\Psi_{\phi}\Omega^{n}\in \GL_{r}(\TT).
	\end{equation*}
    The $\bA$-finite dual $t$-motive $\sH_n^{\tens}$ has a rigid analytic trivialization given by $\Psi_n^{\tens}$ and hence, the pre-$t$-motive $H_{n}^{\tens} := \oK(t)\otimes_{\oK[t]}\sH_{n}^{\tens} $ associated to $\sH_{n}^{\tens}$ forms a $t$-motive. Recall the dual $t$-motive $\sH_{\phi}$ and the field $\rL_{\phi}$ defined in \S\ref{S:GGofEn}. We set $\rL^{\tens}_{n}:=\End_{\sT}(H^{\tens}_n)$ and state the following proposition. 

\begin{proposition}\label{P:equalend} We have $\End_{\oK[t, \sigma]}(\sH_{\phi}) = \End_{\oK[t, \sigma]}(\sH^{\tens}_{n})$ and $\rL_{\phi}= \rL^{\tens}_n$. In particular, $\rL^{\tens}_n$ forms a field.
\end{proposition}
\begin{proof}
	Let $f\in \End_{\oK[t, \sigma]}(H_{\phi})$ and $\rF \in \Mat_{r}(\oK[t])$ be such that $[f(n_1),\dots,f(n_r)]^{{\tr}}= \rF\bsn$. Since $\sigma f = f \sigma$, we obtain $\rF^{(-1)} \Phi_\phi = \Phi_\phi \rF$ from which we observe that $\rF^{(-1)} \Phi^{\tens}_{n} = \Phi^{\tens}_{n} \rF$. Thus, by a slight abuse of notation, $f$ defines an element in $\End_{\oK[t, \sigma]}(\sH^{\tens}_{n})$ such that $[f(h_1^{\tens}),\dots,f(h_r^{\tens})]^{{\tr}}= \rF\bsh_{G_n}$. The other inclusion is obtained by retracing the steps and hence, we obtain the first equality. The second equality can be achieved by using the same argument. Finally the last assertion follows from using the second equality and noting that $\rL_{\phi}$ is a field by Remark \ref{R:isomorphismsofvss}.
\end{proof}
Recall, from the proof of Proposition \ref{P:simpletensor2}, the $t$-motive $C_n=\oK(t)\otimes_{\oK[t]}\sC_n$ where $\sC_n$ is the $\bA$-finite dual $t$-motive corresponding to the $t$-module $\mathbf{C}^{\otimes n}$. Since $H_n^{\tens}\cong H_{\phi}\otimes C_n$, the next proposition follows from the same idea of the proof of Proposition \ref{P:simpletensor2}.

\begin{proposition}\label{P:simpletensor}
	For any $n\in \ZZ_{\geq 1}$, the $t$-motive $H_n^{\tens}$ is simple.
\end{proposition}

	Let $f\in \rL_n^{\tens}$ and $\rF\in \Mat_{r}(\oK(t))$ be such that $[f(h_1^{\tens}),\dots,f(h_r^{\tens})]^{{\tr}}= \rF\bsh_{G_n}$. Note that $\sigma$ fixes 
 \[
 \mathcal{F}_f:=(\Psi^{\tens}_n)^{-1}\rF\Psi^{\tens}_n,
 \]
 and thus, by \cite[Lem. 3.3.2]{P08}, we have $\mathcal{F}_f \in \Mat_{r}(\FF_q(t))$. Moreover, for any $n\in \ZZ_{\geq 1}$, by \cite[Prop. 3.3.8]{P08} we obtain $\End((H_n^{\tens})^{\text{Betti}})=\Mat_{r}(\FF_q(t))$ where we recall the definition of the $\mathbb{F}_q(t)$-vector space $(H_n^{\tens})^{\text{Betti}}$ from \S\ref{SS:t-motives}. Hence, we have the following injective map
$
\rL_n^{\tens} \to \End((H_n^{\tens})^{\text{Betti}}),
$ defined by 
\[
f \mapsto \mathcal{F}_f.
\]
Let $\rR$ be any $\FF_q(t)$-algebra. Since the representation $\mu: \Gamma_{ H_n^{\tens}} \rightarrow \GL((H_n^{\tens})^{\text{Betti}})$ is functorial in $H_n^{\tens}$ (\cite[\S3.5.2]{P08}),  for any $\gamma \in \Gamma_{H_n^{\tens}}(\rR)$ we have the commutative diagram (cf. \cite[$\S$3.2]{CP11}):
\begin{equation} \label{E:Com}
\begin{tikzcd}
\rR \otimes_{\FF_q(t)} (H_n^{\tens})^{\text{Betti}} \arrow{r}{\mu(\gamma)} \arrow{d}{1\otimes \mathcal{F}_f} & \rR \otimes_{\FF_q(t)} (H_n^{\tens})^{\text{Betti}} \arrow{d}{1 \otimes \mathcal{F}_f} \\
\rR \otimes_{\FF_q(t)}(H_n^{\tens})^{\text{Betti}} \arrow{r}{\mu(\gamma)} & \rR \otimes_{\FF_q(t)} (H_n^{\tens})^{\text{Betti}}
\end{tikzcd}.
\end{equation}
Recall the centralizer $\Cent_{\GL_r/\FF_q(t)}$ defined in \S\ref{S:GGofEn}. We have $\Gamma_{H_n^{\tens}} \subseteq \Cent_{\GL_r/\FF_q(t)}(\rL_n^{\tens})$ which follows from \eqref{E:Com}. Thus, the following theorem can be easily obtained by the same idea in the proof of Theorem \ref{T:equalitygamma2}.

\begin{theorem}\label{T:equalitygamma}
	We have $\Gamma_{H_n^{\tens}} = \Gamma_{H_{\phi}}=\Cent_{\GL_r/\FF_q(t)}(\rL_n^{\tens})$. In particular, $\dim \Gamma_{H_n^{\tens}}=r^2/\bs$, where $[K_{\phi}:K]=\bs=[\rL_{\phi}:\mathbb{F}_q(t)]$.
\end{theorem}

\subsection{Endomorphisms of \texorpdfstring{$G_n$}{Gn}}\label{S:EndGn}
Let $\eta\in \End(G_0)$, $\epsilon_{\eta}\in \End_{\oK[t, \sigma]}(\sH_{\phi})$ and $\rF=(\rF_{ij}) \in \Mat_{r}(\oK[t])$ be defined as in \S\ref{S:EndEn} satisfying \eqref{E:actiondrinfelfnor}. By Proposition \ref{P:fullyfaithful} and Proposition~\ref{P:equalend}, there exist a unique element $\rP_{\eta}^{\tens}\in \End(G_n)$ and a map  $\epsilon_{\rP_\eta^{\tens}}\in \End(\sH_n^{\tens})$ such that
$[\epsilon_{\rP_{\eta}}(h_1^{\tens}),\dots,\epsilon_{\rP_{\eta}}(h_r^{\tens})]^{{\tr}}= \rF\bsh_{G_n}$ and 
\begin{equation}\label{E:acthom2}
\begin{bmatrix}
\epsilon_{\rP^{\tens}_{\eta}}(s_{1,rn+1})\\ \vdots\\\vdots\\ \epsilon_{\rP^{\tens}_{\eta}}(s_{rn+1,rn+1})\end{bmatrix} =\begin{bmatrix}s_{1,rn+1}(\rP^{\tens}_{\eta})^{*}\\ \vdots\\\vdots\\s_{rn+1,rn+1}(\rP^{\tens}_{\eta})^{*}
\end{bmatrix}.
\end{equation}
For the matrix $\rB^{\tens}$ defined in \eqref{E:BchangeofB}, let $((\rB^{\tens})^{-1}\rF)_r$ be the $r$-th row of $(\rB^{\tens})^{-1}\rF$. By using \eqref{basistnsigma} and noting that $\bsh_{G_n} = \rB^{\tens} \bsell^{\tens}$, we obtain
\begin{equation}\label{E:actionhom}
\begin{bmatrix}
\epsilon_{\rP^{\tens}_{\eta}}(s_{1,rn+1}) \\ \epsilon_{\rP^{\tens}_{\eta}}(s_{2,rn+1}) \\ \vdots \\ \vdots\\\epsilon_{\rP^{\tens}_{\eta}}(s_{rn,rn+1})\\ \epsilon_{\rP^{\tens}_{\eta}}(s_{rn+1,rn+1})
\end{bmatrix} = 
\begin{bmatrix}
0&\dots&\dots&\dots&0& ((\rB^{\tens})^{-1}\rF)_{r}(\rB^{\tens})(t-\theta)^n\\
\vdots& &&&\vdots& (\rB^{\tens})^{-1}\rF (\rB^{\tens})(t-\theta)^{n-1}\\
\vdots&& &&\vdots& \vdots\\
\vdots&&& &\vdots& (\rB^{\tens})^{-1}\rF (\rB^{\tens})(t-\theta)\\
0&\dots&\dots&\dots&0& (\rB^{\tens})^{-1}\rF (\rB^{\tens})\\
\end{bmatrix}\begin{bmatrix}
s_{1,rn+1} \\ s_{2,rn+1} \\ \vdots \\ \vdots\\s_{rn,rn+1}\\ s_{rn+1,rn+1}
\end{bmatrix}.
\end{equation}
On the other hand, observe, by using the $\oK[t]$-module action on $\sH^{\tens}_{n}$ given in \eqref{E:action}, that for each $i\geq 0$ and $1\leq j\leq rn+1$ the $\ell$-th entry of $\rd(t^i\cdot s_{j,rn+1})\in \sH^{\tens}_n$ is equal to 
\begin{equation}\label{E:actionmor}
\begin{cases} 
\pd_t^k(t^i)|_{t=\theta}, &  \text{if}\, \, \ell =j-rk  \\
0, &\text{if} \, \,  \ell \neq j-rk.
\end{cases}
\end{equation}
Let us set $\bQ:=(\rB^{\tens})^{-1}\rF \rB^{\tens}=(\rB''_{ij})\in \Mat_r(\oK[t])$ and let $\rS:= [\rB''_{r1}, \dots, \rB''_{rr}]^{\tr}$ be its last column. Using \eqref{E:actionmor}, one can compare \eqref{E:acthom2} and \eqref{E:actionhom} to calculate that 
\begin{equation}\label{E:dEnd1}
\rd\rP_{\eta}^{\tens} = 
\left.\begin{bmatrix}
\rB''_{rr} & \pd_t^{1}(\rS)^{\tr} & \dots & \dots & \dots & \pd_{t}^{n-1}(\rS)^{\tr} & \pd_t^n(\rS)^{\tr}\\
& \bQ^{\tr} & \pd_t^1(\bQ)^{\tr} & \dots & \dots & \pd_t^{n-2}(\bQ)^{\tr}& \pd_t^{n-1}(\bQ)^{\tr}\\
&&\ddots&\ddots&\ddots&\vdots&\pd_t^{n-2}(\bQ)^{\tr}\\
&&&\ddots&\ddots&\vdots&\vdots\\
&&&&\ddots&\vdots&\vdots\\
&&&&&\bQ^{\tr}&\pd_t^1(\bQ)^{\tr}\\
&&&&&&\bQ^{\tr}
\end{bmatrix}\right|_{t=\theta} \in \Mat_{rn+1}(\oK).
\end{equation}
By direct calculation and using \cite[Prop.~4.1.1(b), (c)]{CP12}, Remark \ref{R:isomorphismsofvss}(ii) and  \eqref{E:actiondrinfelfnor}, the last row of $\bQ^{\tr}$ is equal to 
\[
\rd \eta =(\sum_{j=2}^r a_j^{(-j+1)} F_{j1}(\theta),\sum_{j=2}^{r-1} a_{j+1}^{(-j+1)} F_{j1}(\theta), \dots, \sum_{j=2}^3 a_{j+(r-3)}^{(-j+1)} F_{j1}(\theta), a_r^{(-1)} F_{21}(\theta),  F_{11}(\theta)), 
\]
and therefore the last row $\Delta\in \Mat_{1\times rn+1}(\oK)$ of 
$\rd\rP^{\tens}_{\eta}$ is equal to 
\begin{equation}\label{E:W}
\begin{split}
\Delta &=[0, \dots, 0, \rd\eta]
\end{split}
\end{equation}

The following proposition is simply a consequence of the methods used in the proof of Proposition \ref{P:rings2}.
\begin{proposition}\label{P:rings}
	The map $\End(G_0) \to \End(G_n)$ given by 
	$\eta \mapsto \rP_{\eta}^{\tens}
	$ for $\eta \in \End(G_0)$ is an isomorphism of rings. The matrix $\rd \rP_\eta^{\tens}$ and its last row are of the form \eqref{E:dEnd1} and \eqref{E:W} respectively. Moreover, $\End(G_n)$ can be identified with $\{\rd\rP^{\tens} \ \ |  \rP^{\tens} \in \End(G_n)\}$ which is an integral  domain.
\end{proposition}

\section{The proof of the main result for \texorpdfstring{$\mathcal{E}_n$}{En}}\label{S:MainResultsEn}
In this section, our goal is to prove Theorem \ref{T:1}. 

\subsection{Construction of \texorpdfstring{$\bg_{\bsalpha}$}{ga} and \texorpdfstring{$\bh_{\bsalpha}$}{ha} for \texorpdfstring{$\mathcal{E}_n$}{En}}\label{SS:gahaextn}
Throughout this subsection, empty sums are considered to be zero. Let $\bsy = [y_1, \dots, y_{rn+r-1}]^{\tr}\in \Lie(\mathcal{E}_n)(\CC_{\infty})$ and $\bsalpha=[\alpha_1, \dots, \alpha_{rn+r-1}]^{\tr}\in \mathcal{E}_n(\oK)$ so that $\Exp_{\mathcal{E}_n}(\bsy)=\bsalpha$. Let $\cG_{\bsy}(t)=[g_1, \dots, g_{rn+r-1}]^{\tr}\in \TT^{rn+r-1}$ be the Anderson generating function of $\mathcal{E}_n$ at $\bsy$ given as in \eqref{E:AGF}.

Recall the matrices $E_1$ and $E_2\in \Mat_{r-1}(\oK)$ given in  \eqref{E:matrices20} and $E'\in \Mat_{rn+r-1}(\oK)$  given in \eqref{E:matrices2A}. By Lemma~\ref{L:AGFNP21}(ii) we have 
\begin{equation}\label{E:AGFQuasi2}
(t\Id_{rn+r-1}-\rd_{\varphi_{n}}(t))\cG_{\bsy}(t) + \bsalpha = C_1\cG_{\bsy}^{(1)}(t)+C_2\cG_{\bsy}^{(2)}(t)
\end{equation}
where $C_1:=E_1$ and $C_2:=E_2$ if $n=0$ and $C_1:=E'$ and $C_2:=0$ if $n\geq 1$.

We now introduce the vectors $\bg_{\bsy}$ and $\bh_{\bsalpha}$  with coefficients in $\TT$ to be able to construct a certain $t$-motive. For more details on their construction, one can consult \cite[Sec. 4.4]{NPapanikolas21}. 

Set 
\begin{equation}\label{E:gformathcal{E}}
\bg_{\bsy}:=-[\bG_{1,1},g_{r-1}^{(1)},\dots,g_1^{(1)}](V^{-1})^{\tr}\in \Mat_{1 \times r}(\TT)
\end{equation}
where $\bG_{1,1}:=(-1)^{r-1}g_1^{(2)}$ if $n=0$ and  $\bG_{1,1}:=g_r^{(1)}$ otherwise. Recall $V\in \GL_{r}(\oK)$ from \eqref{E:UpsilonNV} and note that $V^{-1}$ is of the form
\[
V^{-1}= \begin{bmatrix}
&&& \mathfrak{v}_{1,r}\\
&& \mathfrak{v}_{2,r-1}& \mathfrak{v}_{2,r}\\
&\reflectbox{$\ddots$}&&\vdots\\
\mathfrak{v}_{r,1} & \mathfrak{v}_{r, 2}&\cdots&\mathfrak{v}_{r,r}
\end{bmatrix}.
\]

Note that the last $r$ rows of $N'\in \Mat_{rn+r-1}(\mathbb{F}_q)$ given in \eqref{E:matrices2} have only zeros. Therefore, by using \eqref{E:AGFQuasi2}, we see that
\[
\bg_{\bsy}= \begin{bmatrix}
\bg_1(\bsy)\\
-\mathfrak{v}_{2,r-1}(-1)^{r-1}a_r^{-1}\left((t-\theta)g_{rn+1}+\alpha_{rn+1}\right)\\
-\sum_{i=1}^{2} \mathfrak{v}_{3,r-i}(-1)^{r-1}a_r^{-1}\left((t-\theta)g_{rn+i}+\alpha_{rn+i}\right)\\
\vdots \\
-\sum_{i=1}^{r-2} \mathfrak{v}_{r-1,r-i}(-1)^{r-1}a_r^{-1}\left((t-\theta)g_{rn+i}+\alpha_{rn+i}\right)\\
-\sum_{i=1}^{r-1} \mathfrak{v}_{r,r-i}(-1)^{r-1}a_r^{-1}\left((t-\theta)g_{rn+i}+\alpha_{rn+i}\right)
\end{bmatrix}^{\tr} \in \Mat_{1 \times r}(\TT)
\]
where $\bg_{1}(\bsy):=-a_r^{-1}g_{1}^{(1)}$ if $n=0$ and $\bg_{1}(\bsy):=-(-1)^{r-1}a_r^{-1}\left((t-\theta)g_{rn}- \alpha_{rn}\right)$ otherwise.

Consider the quasi-periodic function $F_{\mathcal{E}_0,\delta}$ associated to the $\phi_0$-biderivation $\delta$ uniquely determined by $\delta_t= (\tau, 0,\dots,0)$. By \cite[Prop. 4.3.5(a)]{NPapanikolas21} we have \[
F_{\mathcal{E}_0,\delta}(\bsy) = g_1^{(1)}|_{t=\theta}.
\]
By using Lemma \ref{L:AGFNP21}(i) (see also \cite[Prop. 4.2.7]{NPapanikolas21}), we obtain
\begin{equation}\label{E:gtheta1}
\bg_{\bsy}(\theta)= \begin{bmatrix}
\bg_{1}(\bsy)|_{t=\theta} \\
\mathfrak{v}_{2,r-1}(-1)^{r-1}a_r^{-1}\left(y_{rn+1}-\alpha_{rn+1}\right)\\
\sum_{i=1}^{2} \mathfrak{v}_{3,r-i}(-1)^{r-1}a_r^{-1}\left(y_{rn+i}-\alpha_{rn+i}\right)\\
\vdots\\
\sum_{i=1}^{r-2} \mathfrak{v}_{r-1,r-i}(-1)^{r-1}a_r^{-1}\left(y_{rn+i}-\alpha_{rn+i}\right)\\
\sum_{i=1}^{r-1} \mathfrak{v}_{r,r-i}(-1)^{r-1}a_r^{-1}\left(y_{rn+i}-\alpha_{rn+i}\right)
\end{bmatrix}^{\tr} \in \Mat_{1 \times r}(\mathbb{C}_{\infty})
\end{equation}
where $\bg_{1}(\bsy)|_{t=\theta}:=-a_r^{-1}F_{\mathcal{E}_0,\delta}(\bsy)$ if $n=0$ and $\bg_{1}(\bsy)|_{t=\theta}:=(-1)^{r-1}a_r^{-1}(y_{rn} - \alpha_{rn})$ otherwise.

We now set 
\[
\begin{split}
\bh_{\bsalpha}:= 
\begin{bmatrix}
\sum_{j=0}^{n-1} (-1)^{r-1}a_r^{-1} (t-\theta)^j\alpha_{r(n-j)} \\
\mathfrak{v}_{2,r-1}(-1)^{r-1}a_r^{-1}\left(\alpha_{rn+1}+\sum_{j=0}^{n-1} (t-\theta)^{j+1}\alpha_{r(n-j-1)+1}\right)\\
\sum_{i=1}^2\mathfrak{v}_{3,r-i}(-1)^{r-1}a_r^{-1}\left(\alpha_{rn+i}+\sum_{j=0}^{n-1} (t-\theta)^{j+1}\alpha_{r(n-j-1)+i}\right)\\
\sum_{i=1}^3 \mathfrak{v}_{4,r-i}(-1)^{r-1}a_r^{-1}\left(\alpha_{rn+i}+\sum_{j=0}^{n-1} (t-\theta)^{j+1}\alpha_{r(n-j-1)+i}\right)\\
\vdots\\
\sum_{i=1}^{r-2}\mathfrak{v}_{r-1,r-i}(-1)^{r-1}a_r^{-1}\left(\alpha_{rn+i}+\sum_{j=0}^{n-1} (t-\theta)^{j+1}\alpha_{r(n-j-1)+i}\right)\\
\sum_{i=1}^{r-1}\mathfrak{v}_{r,r-i}(-1)^{r-1}a_r^{-1}\left(\alpha_{rn+i}+\sum_{j=0}^{n-1} (t-\theta)^{j+1}\alpha_{r(n-j-1)+i}\right)\\
\end{bmatrix}^{\tr}\in \Mat_{1\times r}(\oK[t]).
\end{split}
\]
By using the definition of $\bg_{\bsy}$ and \eqref{E:AGFQuasi2}, we have  
\begin{equation}\label{E:RATg2}
\bg_{\bsy}^{(-1)}\Phi_n-\bg_{\bsy} = \bh_{\bsalpha}.
\end{equation}

Since $\sH_{n}$ is rigid analytically trivial, by \cite[Lem. 2.4.1, Thm. 2.5.32]{HartlJuschka16}, we obtain that $\Lambda_{\mathcal{E}_n}$ is a free $\bA$-module of rank $r$. Let $\{\bsomega_1, \dots, \bsomega_r\}$ be a fixed $\bA$-basis of $\Lambda_{\mathcal{E}_n}$. For each $1\leq k \leq r$, denote $\cG_{\bsomega_k}(t)=[g_{k,1},\dots,g_{k,rn+r-1}]^{\tr}$. Consider
\[
\mathfrak{B}_n:= \begin{bmatrix}
\bg_{1,1}(\bsomega_1)& g_{1,r-1}^{(1)} & \dots &g_{1,1}^{(1)}\\
\vdots &  \vdots & &  \vdots\\
\bg_{1,1}(\bsomega_r)& g_{r,r-1}^{(1)} &\dots & g_{r,1}^{(1)}
\end{bmatrix} (V^{-1})^{\tr}\in \Mat_{r}(\TT),
\] 
where for $1\leq j \leq r$, we set $\bg_{1,1}(\bsomega_j) := (-1)^{r-1}g_{j,1}^{(2)}$ if $n=0$ and $\bg_{1,1}(\bsomega_j) := g_{j,r}^{(1)}$ otherwise. It follows from \cite[Prop. 4.3.10]{NPapanikolas21} that $\mathfrak{B}_n\in \GL_r(\TT)$. Using \cite[Prop. 4.4.14]{NPapanikolas21} and setting $\Pi_n := \mathfrak{B}_n^{-1}$, one further sees that $\Pi_n^{(-1)} = \Phi_{n}\Pi_n$. Thus, $\Pi_n$ is a rigid analytic trivialization of $\sH_n$ and hence, by the discussion in \cite[Sec. 4.1.6]{P08}, there exists a  matrix $D=(D_{ij})\in \GL_r(\mathbb{F}_q(t))$ satisfying 
\[
\Pi_n D^{-1} = \Psi_{n}.
\]
Using a similar calculation as in \eqref{E:gtheta1} and setting  $\bsomega_i = [\omega_{i, 1}, \dots, \omega_{i,rn+r-1}]^{\tr}\in \Lie(\mathcal{E}_n)(\CC_{\infty})$ for $1 \leq i \leq r$, we see that the $i$-th row of $\Psi_n^{-1}(\theta) = D(\theta)\mathfrak{B}_n(\theta)$ is equal to
\begin{equation}\label{E:matrix22}
\begin{bmatrix}
\rR_{11} \\
(-1)^{r-1}a_r^{-1}(D_{i1}(\theta)\mathfrak{v}_{2,r-1}\omega_{1,rn+1}+\dots+D_{ir}(\theta)\mathfrak{v}_{2,r-1}\omega_{r,rn+1})  \\
(-1)^{r-1}a_r^{-1}(D_{i1}(\theta)\sum_{i=1}^{2} \mathfrak{v}_{3,r-i}\omega_{1,rn+i}+\dots+D_{ir}(\theta)\sum_{i=1}^{2} \mathfrak{v}_{3,r-i}\omega_{r,rn+i}) \\
\vdots \\
(-1)^{r-1}a_r^{-1}(D_{i1}(\theta)\sum_{i=1}^{r-1}\mathfrak{v}_{r,r-i}\omega_{1,rn+i}+\dots+D_{ir}(\theta)\sum_{i=1}^{r-1}\mathfrak{v}_{r,r-i}\omega_{r,rn+i})
\end{bmatrix}^{\tr},
\end{equation}
where 
\[
\rR_{11}:=\begin{cases} a_r^{-1}D_{i1}(\theta)F_{\mathcal{E}_0,\delta}(\bsomega_1)+\dots+a_r^{-1}D_{ir}(\theta)F_{\mathcal{E}_0,\delta}(\bsomega_r) \text{ if } n=0 \\  (-1)^{r-1}a_r^{-1}(D_{i1}(\theta)\omega_{1,rn}+\dots+D_{ir}(\theta)\omega_{r,rn}) \text{  otherwise}
\end{cases}.
\]

\par
Consider the $\CC_{\infty}[t,\sigma]$-module $\Mat_{1\times rn+r-1}(\CC_{\infty}[\sigma])$ whose $\CC_{\infty}[t]$-module action is given by 
\[
\mathfrak{c}t \cdot \mathfrak{u} := \mathfrak{c}\mathfrak{u}\varphi_n(t)^{*}, \ \ \mathfrak{c}\in \CC_{\infty}, \ \ \mathfrak{u}\in \Mat_{1\times rn+r-1}(\CC_{\infty}[\sigma]).
\]
Recalling the elements $\ell_{1,n},\dots,\ell_{r,n}$ given in \eqref{E:basisEn}, we define the map
\[
\iota: \Mat_{1\times r}(\CC_{\infty}[t]) \rightarrow  \Mat_{1\times rn+r-1}(\CC_{\infty}[\sigma])
\]
by setting for $[\mathfrak{g}_1, \dots, \mathfrak{g}_r] \in \Mat_{1\times r}(\oK[t])$,
\[
\iota([\mathfrak{g}_1, \dots, \mathfrak{g}_r]) := \mathfrak{g}_1 \cdot \ell_{1,n} + \dots+\mathfrak{g}_r \cdot \ell_{r,n}.
\]
We also define the $\CC_{\infty}$-linear map 
\[\delta_0:\Mat_{1\times rn+r-1}(\CC_{\infty}[\sigma])\to \Mat_{rn+r-1\times 1}(\CC_{\infty})
\]
given by
\[
\delta_0\bigg(\sum_{i\geq 0}C_i\sigma^i\bigg) := C_0^{\tr},\ \ C_i\in \Mat_{1\times rn+r-1}(\CC_{\infty}), \ \ C_i=0 \text{ for } i\gg0 .
\]
By using the ideas of the proof of \cite[Prop.~5.2]{G20} we have 
\begin{multline}\label{E:deltaiota2}
\delta_0\circ\iota([\mathfrak{g}_1, \dots, \mathfrak{g}_r]) = [\pd_t^n(\mathfrak{g}_2), \dots,\pd_t^n(\mathfrak{g}_r),\pd_t^{n-1}(\mathfrak{g}_1), \dots, \pd_t^{n-1}(\mathfrak{g}_r),\dots,\\ \pd_t^1(\mathfrak{g}_1), \dots,  \pd_t^1(\mathfrak{g}_r), \mathfrak{g}_1, \dots, \mathfrak{g}_r ]^{\tr} |_{t=\theta}.
\end{multline}
Moreover, Anderson proved that (see \cite[Prop.~2.5.8]{HartlJuschka16}), we have a unique extension of $\delta_0 \circ \iota$ given by
\[
\widehat{\delta_0 \circ \iota}: \Mat_{1\times r}(\TT_{\theta}) \rightarrow \Mat_{rn+r-1\times 1}(\CC_{\infty})
\]
where $\TT_{\theta}$ is the set of elements $f=\sum c_it^i\in \TT$ so that as a function of $t$, $f$ converges in the disk $\{z\in \CC_{\infty}|\ \  \inorm{z}\leq q\}$.

Set $\widetilde{\bg}_{\bsy}:= \bg_{\bsy}\rB$ and $\widetilde{\bh}_{\bsalpha}:= \bh_{\bsalpha}\rB$, where $\rB$ is as in \eqref{E:CchangeofC}. By using Lemma \ref{L:AGFNP21}(i), \eqref{E:RATg2} and \eqref{E:deltaiota2} one can obtain
\begin{equation}\label{E:epsiotay2}
\widehat{\delta_0 \circ \iota}(\widetilde{\bg}_{\bsy}+\widetilde{\bh}_{\bsalpha}) = \bsy.
\end{equation}
Similarly, for $\widetilde{\Gamma}:= \mathfrak{B}_n \rB$ we have 
\begin{equation}\label{E:epsiotaomega2}
\widehat{\delta_0 \circ \iota}(\widetilde{\Gamma}_i) = -\bsomega_i,
\end{equation}
where $\widetilde{\Gamma}_i$ is the $i$-th row of $\widetilde{\Gamma}$.\par

\subsection{Logarithms of \texorpdfstring{$\mathcal{E}_n$}{En} and \texorpdfstring{$\Ext_\sT^1(\mathbf{1}_{\sP}, H_n)$}{ExtHn}}\label{SS:tmotiveconsextn}
We now define the pre-$t$-motive $Y_{\bsalpha}$  so that it is of dimension $r+1$ over $\oK(t)$ and for a chosen $\oK(t)$-basis $\bsy_{\alpha}$ of $Y_{\bsalpha}$, we have $\sigma \bsy_{\alpha}=\Phi_{Y_{\bsalpha}}\bsy_{\alpha}$ where 
\[
\Phi_{Y_{\bsalpha}} := \begin{bmatrix}
\Phi_{n} & {\bf{0}}\\
\bh_{\bsalpha} & 1
\end{bmatrix}\ \in \Mat_{r+1}(\oK[t]).
\]
It is easy to see  by using \eqref{E:RATg2} that $Y_{\bsalpha}$ has a rigid analytic trivialization $\Psi_{Y_{\bsalpha}} \in \GL_{r+1}(\TT)$ given by 
\[
\Psi_{Y_{\bsalpha}} = \begin{bmatrix}
\Psi_{n} & {\bf{0}}\\
\bg_{\bsy}\Psi_{n} & 1
\end{bmatrix}.
\]

\begin{lemma}[{cf.\ \cite[Prop.~6.1.3]{P08}}]\label{L:Yalpha2} 
The pre-$t$-motive $Y_{\bsalpha}$ is  a $t$-motive.
\end{lemma}

\begin{proof}
Let $\cN:= \oplus_{i=1}^{r+1} \oK[t]\bse_i$ be the free $\oK[t]$-module with $\oK[t]$-basis given by the entries of $\bse :=(\bse_1, \dots, \bse_{r+1})^{\tr}$.  We equip $\cN$ with a left $\oK[t, \sigma]$-module structure by setting \[
\sigma \bse := (t-\theta)\Phi_{Y_{\bsalpha}}\bse.
\]
Recall from the proof of Proposition \ref{P:equalend2} that $\sC:=\sC_1$ is the $\bA$-finite dual $t$-motive associated  to the Carlitz module $\mathbf{C}^{\otimes 1}=(\GG_{a.K}, C^{\otimes 1})$. Then $\sigma \cdot f = (t-\theta)f^{(-1)}$ for all $f \in \sC$. We obtain the following short exact sequence of left $\oK[t, \sigma]$-modules:
\begin{equation}\label{SEStensor}
0\rightarrow \sH_n \otimes_{\oK[t]} \sC\rightarrow \cN \rightarrow \sC\rightarrow 0.
\end{equation}
Since $\sC$ and $\sH_n \otimes_{\oK[t]} \sC$ are finitely generated left $\oK[\sigma]$-modules, it follows from \cite[Prop.~4.3.2]{ABP04} that $\cN$ is free and finitely generated as a left $\oK[\sigma]$-module. Since $\sH_n \otimes_{\oK[t]} \sC$ is an $\bA$-finite dual $t$-motive, we have $(t-\theta)^u(\sH_n \otimes_{\oK[t]} \sC) \subseteq \sigma(\sH_n \otimes_{\oK[t]} \sC)$ for some $u \in \ZZ_{\geq 1}$. Moreover, $(t-\theta)\sC=\sigma \sC$. By \eqref{SEStensor}, we obtain $(t-\theta)^v\cN \subseteq \sigma\cN$ for some $v \in \ZZ_{\geq 1}$. Thus, we see that $\cN$ is an $\bA$-finite dual $t$-motive. Since $\cN \otimes_{\oK[t]}\oK(t) \cong Y_{\bsalpha} \otimes_{\oK(t)} \sC$ as left $\oK(t)[\sigma, \sigma^{-1}]$-modules, $Y_{\bsalpha} \otimes_{\oK(t)} \sC$ is a $t$-motive and hence, it follows from the discussion in \cite[Sec. 3.4.10]{P08} that $Y_{\bsalpha}$ is a $t$-motive.
\end{proof}

Recall that $\mathbf{1}_{\sP}$ is the trivial object in $\sP$ and  $\rL_n=\End_{\sT}(H_n)$. One can define a $\rL_n$-vector space structure on the group $\Ext_\sT^1(\mathbf{1}_{\sP}, H_n)$ as follows: Let $Y_1$ and $Y_2$ represent classes in $\Ext_\sT^1(\mathbf{1}_{\sP}, H_n)$. If the multiplication by $\sigma$ on suitable $\oK(t)$-bases of $Y_1$ and $Y_2$ are represented by 
$\begin{bsmallmatrix}\Phi_n & \mathbf{0}\\ \bv_1 & 1\end{bsmallmatrix}\in \Mat_{r+1}(\oK(t)) \, {\mathrm{and}} \,  \begin{bsmallmatrix}\Phi_n & \mathbf{0}\\ \bv_2 & 1\end{bsmallmatrix}\in \Mat_{r+1}(\oK(t))$
respectively, then their Baer sum $Y_1+Y_2$ represents a class in $\Ext_\sT^1(\mathbf{1}_{\sP},H_n)$ so that the multiplication by $\sigma$ on a chosen $\oK(t)$-basis of $Y_1+Y_2$ is represented by the matrix 
$\begin{bsmallmatrix}\Phi_n & \mathbf{0}\\ \bv_1+\bv_2 & 1\end{bsmallmatrix} \in \Mat_{r+1}(\oK(t))$.

Recall the $\oK[t]$-basis $\bsh_{\mathcal{E}_n}$ of $H_n$ from \S\ref{S:tmoduleEn}. We now suppose that the element $e\in \rL_n$ is represented by $\mathfrak{E} \in \Mat_{r}(\oK(t))$, that is $[e(h_1),\dots,e(h_r)]^{\tr}=\mathfrak{E}\bsh_{\mathcal{E}_n}$. Then the element $e\cdot Y_1=e_*Y_1$, the action of $e$ on $Y_1$, defines a class in $\Ext_\sT^1(\mathbf{1}_{\sP}, H_n)$ such that the multiplication by $\sigma$ on a chosen $\oK(t)$-basis of $e_*Y_1$ is represented by the matrix 
$\begin{bsmallmatrix}\Phi_n & \mathbf{0}\\ \bv_1 \mathfrak{E} & 1\end{bsmallmatrix}\in \Mat_{r+1}(\oK(t))$.
Note that by the construction and Lemma \ref{L:Yalpha2}, $Y_{\bsalpha}$ represents a class in $\Ext_{\sT}(\mathbf{1}_{\sP}, H_n)$.

For the rest of the paper,  we set $K_n$ to be the fraction field of $\End(\mathcal{E}_n)$. Before introducing the main result of this subsection, we say that  the elements $\bsz_1,\dots,\bsz_k\in \Lie(\mathcal{E}_n)(\CC_{\infty})$ are \textit{linearly independent over $K_n$} if whenever $\rd \rP_1\bsz_1+\dots+\rd \rP_k\bsz_k=0$ for some $\rP_1,\dots,\rP_k\in K_n$, we have $\rP_1=\dots=\rP_k=0$. 
 
 Recall the integer $\bs=[\rL_{\phi}:\mathbb{F}_q(t)]=[K_{\phi}:K]$. By Proposition~\ref{P:equalend2} and Proposition \ref{P:rings2}, $\bs = [\rL_{n}:\mathbb{F}_q(t)]=[K_{n}:K]$. Suppose that $\{\bsomega_1, \dots, \bsomega_{r/\bs}\}$ is the maximal $K_n$-linearly independent subset of $\{\bsomega_1, \dots, \bsomega_{r}\}$.

\begin{theorem}\label{T:Trivial2}
	For each $1 \leq \ell \leq m$, let $\bsy_{\ell}\in \Lie(\mathcal{E}_n)(\CC_{\infty})$ be such that $\Exp_{G_{n}}(\bsy_{\ell})=\bsalpha_{\ell}\in \mathcal{E}_n(\oK)$. Suppose that the set $\{\bsomega_1, \dots, \bsomega_{r/\bs}, \bsy_1, \dots, \bsy_m\}$ is linearly independent over $K_n$. We further let $Y_{\ell}:=Y_{\bsalpha_{\ell}}$. Then, for any choice of endomorphisms $e_1, \dots, e_m \in \rL_n$ not all zero, the equivalence class $\mathcal{S}:=e_{1*}Y_{1} + \dots + e_{m*}Y_{m}\in \Ext_\sT^1(\mathbf{1}_{\sT},H_n)$ is non-trivial.
\end{theorem}

\begin{proof} Our proof adapts the ideas of the proof of \cite[Thm.~4.2.2]{CP12}.  Suppose on the contrary that $\mathcal{S}$ is trivial in $\Ext_\sT^1(\mathbf{1}_{\sP}, H_n)$ and consider  $E_{\ell} = ((E_{\ell})_{u,v})\in \Mat_{r}(\oK(t))$ satisfying $[e_{\ell}(h_1),\dots,e_{\ell}(h_r)]^{\tr}=E_{\ell}\bsh_{\mathcal{E}_n}$. Set $\bg_{\ell}:=\bg_{\bsy_\ell}$ and $\bh_{\ell}=\bh_{\bsalpha_{\ell}}$. Note that by choosing an appropriate $\oK(t)$-basis~$\widetilde{\bsh}$ for $\mathcal{S}$, the multiplication by $\sigma$ on $\widetilde{\bsh}$ and a corresponding rigid analytic trivialization may be represented by 
\[\Phi_\mathcal{S}~:=~\begin{bmatrix}
\Phi_n & \mathbf{0} \\
\sum_{\ell=1}^m\bh_{\ell} E_{\ell} & 1
\end{bmatrix} \in \GL_{r+1}(\oK(t)), \, \, \,
\Psi_\mathcal{S}~:=~\begin{bmatrix}
\Psi_n & \mathbf{0} \\
\sum_{\ell=1}^m\bg_{\ell} E_\ell\Psi_n& 1
\end{bmatrix} \in \GL_{r+1}(\TT).\]

  Since $\mathcal{S}$ is trivial, there exists $\bsgamma = \begin{bsmallmatrix} \Id_{r} & \mathbf{0}\\ \gamma_{1} \dots \gamma_{r} & 1\end{bsmallmatrix} \in \GL_{r+1}(\oK(t))$ such that letting $\widetilde{\bsh}' := \bsgamma\widetilde{\bsh}$ implies that $\sigma \widetilde{\bsh}'= (\Phi_n \oplus (1))\widetilde{\bsh}'$, where $\Phi_n\oplus (1) \in \GL_{r+1}(\oK(t))$ is the block diagonal matrix with $\Phi_n$ and $1$ down the diagonal blocks and all other entries are zero. Thus, we obtain 
\begin{equation}\label{E:Spe1}
\bsgamma^{(-1)}\Phi_\mathcal{S} = (\Phi_n \oplus (1)) \bsgamma.
\end{equation}
Note that by the same argument in \cite[p. 146]{P08} applied to an equation of the form \eqref{E:Spe1}, for $1\leq i \leq r$, the denominator of each $\gamma_{i}$ is in $\bA$ and hence is regular at $t = \theta, \theta^q, \theta^{q^2}, \dots$. For each $\ell$, we let $\bsalpha_{\ell}=[\alpha_{\ell,1}, \dots, \alpha_{\ell,rn+r-1}]^{\tr}$. Comparing the $(r+1)$-st row of both sides of \eqref{E:Spe1} yields 
\begin{equation}\label{E:entry}
[\gamma_1, \dots, \gamma_r]^{(-1)}\Phi_n + \sum_{\ell=1}^m \bh_{\ell}E_{\ell} = [\gamma_1, \dots, \gamma_r].
\end{equation}
Multiplying both sides of \eqref{E:Spe1} from the right by $\Psi_\mathcal{S}$ implies that $(\bsgamma \Psi_\mathcal{S})^{(-1)} = (\Phi_n \oplus (1)) (\bsgamma \Psi_\mathcal{S})$ and therefore, by \cite[$\S$4.1.6]{P08}, for some 
$\bsdelta = \begin{bsmallmatrix} \Id_{r} & \mathbf{0}\\ \delta_{1} \dots \delta_{r}&1\end{bsmallmatrix}~\in~\GL_{r+1}(\FF_q(t))$,
we have 
$
\bsgamma \Psi_\mathcal{S} = (\Psi_n \oplus (1))\bsdelta.
$
This implies that
\begin{equation}\label{ExtProof2}
[\gamma_{1},\dots,\gamma_{r}] + \sum_{\ell=1}^m\bg_{\ell} E_{\ell} = [\delta_{1}, \dots, \delta_{r}]\Psi_n^{-1}.
\end{equation}
Adding \eqref{E:entry} to \eqref{ExtProof2} and right multiplication of both sides by $\rB$ imply
\begin{equation}\label{E:eps11}
[\gamma_1, \dots, \gamma_r]^{(-1)}\Phi_n\rB  + \sum_{\ell=1}^m(\widetilde{\bg}_{\ell}+\widetilde{\bh}_{\ell}) \rB^{-1}E_{\ell}\rB = [\delta_{1}, \dots, \delta_{r}]D\widetilde{\Gamma}.
\end{equation}
Using \eqref{E:deltaiota2}, we obtain
\begin{equation}\label{E:eps12}
\delta_0 \circ \iota\left([\gamma_1, \dots, \gamma_r]^{(-1)}\Phi_n \rB\right) = 0. 
\end{equation}
Set $\bQ_{\ell}:=\rB^{-1}E_{\ell} \rB=((\rB'_{\ell})_{i,j})\in \Mat_r(\oK[t])$ and let $\rS_{\ell}:= [(\rB'_{\ell})_{r,1}, \dots, (\rB'_{\ell})_{r,r}]^{\tr}$ be its last column. Then, by using \eqref{E:deltaiota2} for $1\leq \ell \leq m$ we have
\[
\begin{split}
\widehat{\delta_0 \circ \iota}&\left((\widetilde{\bg}_{\ell}+\widetilde{\bh}_{\ell}) \bQ_{\ell} \right)\\&=
\left.\begin{bmatrix}
\rS_{\ell}^{\tr} & \pd_t^{1}(\rS_{\ell})^{\tr} & \dots & \dots  & \pd_{t}^{n-1}(\rS_{\ell})^{\tr} & \pd_t^n(\rS_{\ell})^{\tr}\\
         & \bQ_{\ell}^{\tr} & \pd_t^1(\bQ_{\ell})^{\tr} & \dots  & \pd_t^{n-2}(\bQ_{\ell})^{\tr}& \pd_t^{n-1}(\bQ_{\ell})^{\tr}\\
         &&\ddots&\ddots&\vdots&\pd_t^{n-2}(\bQ_{\ell})^{\tr}\\
         &&&\ddots&\vdots&\vdots\\
         &&&&\bQ_{\ell}^{\tr}&\pd_t^1(\bQ_{\ell})^{\tr}\\
         &&&&&\bQ_{\ell}^{\tr}
\end{bmatrix}\begin{bmatrix}
\pd_t^n(\widetilde{\bg}_{\ell}+\widetilde{\bh}_{\ell})^{\tr})\\
\pd_t^{n-1}(\widetilde{\bg}_{\ell}+\widetilde{\bh}_{\ell})^{\tr}\\
\vdots\\
\pd_t^1(\widetilde{\bg}_{\ell}+\widetilde{\bh}_{\ell})^{\tr}\\
(\widetilde{\bg}_{\ell}+\widetilde{\bh}_{\ell})^{\tr}
\end{bmatrix}\right|_{t=\theta}.
\end{split}
\]
Let $(\widetilde{\bg}_{\ell}+\widetilde{\bh}_{\ell})_r$ be the last entry of $\widetilde{\bg}_{\ell}+\widetilde{\bh}_{\ell}$ and let $\rP_{\ell}\in \End(G_n)$ be the morphism defined by $e_{\ell}$ via Proposition~\ref{P:fullyfaithful}. Using \cite[Prop.~4.1.1(b)]{CP12}, one finds that $\rS_{\ell}|_{t=\theta}=[0, \dots, 0, (\rB'_{\ell})_{r,r}(\theta)]^{\tr}$. Hence, we obtain by using \eqref{E:dEnd2} and \eqref{E:epsiotay2} that
\begin{equation}\label{E:linearrely}
\widehat{\delta_0 \circ \iota}\left((\widetilde{\bg}_{\ell}+\widetilde{\bh}_{\ell}) \bQ_{\ell} \right) = \rd\rP_{\ell} \, \, \widehat{\delta_0 \circ \iota} \left(\widetilde{\bg}_{\ell}+\widetilde{\bh}_{\ell}\right)
= \rd\rP_{\ell}\bsy_{\ell}.
\end{equation}
For $1\leq i \leq r$, let $\Delta_i \in \Mat_r(\FF_q(t))$ denote the diagonal matrix with $\delta_1 D_{i,1}+\cdots+ \delta_r D_{i,r}$ in the diagonal entries and zero everywhere else and let $\widetilde{\Gamma}_i$ be the $i$-th row of $\widetilde{\Gamma}$. Then,
\[
[\delta_{1}, \dots, \delta_{r}]D\widetilde{\Gamma}= \widetilde{\Gamma}_1\Delta_1+\dots+\widetilde{\Gamma}_r \Delta_r.
\]
Note that $[f_i(h_1),\dots,f_i(h_r)]^{\tr}=\Delta_i\bsh_{\mathcal{E}_n}$ for some $f_i \in \rL_n$ and $\rB^{-1}\Delta_i \rB= \Delta_i$. Let $\rP_{f_i} \in \End(G_n)$ be the morphism defined by $f_i$ via Proposition~\ref{P:fullyfaithful} and \eqref{E:actionhom2}. Then, by a similar calculation to \eqref{E:linearrely} and using \eqref{E:epsiotaomega2} we obtain
\begin{equation}\label{E:eps13}
\widehat{\delta_0 \circ \iota}\left([\delta_{1}, \dots, \delta_{r}]D\widetilde{\Gamma}\right) = -\rd\rP_{f_1}\bsomega_1-\dots-\rd\rP_{f_r}\bsomega_r.
\end{equation}
Finally, \eqref{E:eps11}, \eqref{E:eps12}, \eqref{E:linearrely}, and \eqref{E:eps13} yield
\[
\sum_{\ell=1}^m \rd\rP_{\ell}\bsy_{\ell}= -\rd\rP_{f_1}\bsomega_1-\dots-\rd\rP_{f_r}\bsomega_r.
\]
Since $e_1, \dots, e_m$ are not all zero, $\rP_\ell$ is nonzero for some $1 \leq \ell \leq m$. Moreover, since $D\in \GL_r(\FF_q(t))$ we have $\rP_{f_i}$ is nonzero for some $1 \leq i \leq r$. Thus, we get a contradiction to our assumption in the statement of the theorem. 
\end{proof}

\subsection{The proof of Theorem \ref{T:1}}\label{S:Proof1}
For $1 \leq \ell \leq m$, we continue with $\bg_\ell$,  $\bh_\ell$ and $Y_{\ell}$ as in the proof of Theorem~\ref{T:Trivial2}. Then, the block diagonal matrix $\Phi :=\oplus_{\ell=1}^m \Phi_{Y_\ell}\in \Mat_{(r+1)m}(\oK[t])$ represents the multiplication by $\sigma$ on a chosen $\oK(t)$-basis of the $t$-motive $Y :=\oplus_{\ell=1}^m Y_{\ell}$ constructed as the direct sum of $t$-motives $Y_{1},\dots,Y_m$. Note that the block diagonal matrix $\Psi:= \oplus_{\ell=1}^m\Psi_{Y_{\ell}}\in \GL_{(r+1)m}(\TT)$ is a rigid analytic trivialization of $Y$. \par
 
Define the $t$-motive $\textbf{X}$ such that the multiplication by $\sigma$ on a $\oK(t)$-basis is given by $\Phi_{\textbf{X}} \in \GL_{rm+1}(\oK(t))$ along with a rigid analytic trivialization $\Psi_{\textbf{X}} \in \GL_{rm+1}(\TT)$ where
 \begin{equation}\label{E:PhiPsiN}
 \Phi_{\textbf{X}} := \begin{bmatrix} \Phi_n && & \\ 
 &\ddots &&\\
& &\Phi_n &\\ 
 \bh_{1} & \dots& \bh_{m} & 1 \end{bmatrix} \quad \text{and} \quad \Psi_{\textbf{X}} := \begin{bmatrix} \Psi_n & && \\
 & \ddots &&\\
 & &\Psi_n&\\ 
 \bg_{1}\Psi_n & \dots& \bg_{m}\Psi_n & 1 \end{bmatrix}.
 \end{equation}
Observe that $\textbf{X}$ is an extension of $\mathbf{1}_{\sP}$ by $H_n^m$. Moreover, $\textbf{X}$ is a pullback of the surjective map $Y \twoheadrightarrow \mathbf{1}_{\sP}^m$ and the diagonal map $\mathbf{1}_{\sP} \rightarrow \mathbf{1}_{\sP}^m$. Thus, the two $t$-motives $Y$ and $\textbf{X}$ generate the same Tannakian subcategory of $\sT$ and hence, the Galois groups $\Gamma_{Y}$ and $\Gamma_{\textbf{X}}$ are isomorphic. For any $\FF_q(t)$-algebra $\rR$, an element of $\Gamma_{\textbf{X}}(\rR)$ is of the form 
 \[
\nu = \begin{bmatrix} \mu &&& \\ 
& \ddots &&\\
& & \mu &\\ 
\bv_{1} &\dots &\bv_m & 1\end{bmatrix},
 \]
 where $\mu \in \Gamma_{H_n}(\rR)$ and each $\bv_1, \dots, \bv_m \in \GG_a^{r}$. Since $H_n^m$ is a sub-$t$-motive of $\textbf{X}$, we have the following short exact sequence of affine group schemes over $\FF_q(t)$,
\begin{equation}\label{E:SESlog2}
1 \rightarrow W \rightarrow \Gamma_{\textbf{X}} \xrightarrow{\pi} \Gamma_{H_n} \rightarrow 1,
\end{equation}
where $\pi^{(\rR)}:\Gamma_{\textbf{X}}(\rR)\rightarrow \Gamma_{H_n}(\rR)$ is the map $\nu \mapsto \mu$ (cf.~\cite[p. 138]{CP12}). Note that via conjugation  \eqref{E:SESlog2} gives an action of any $\mu \in \Gamma_{H_n}(\rR)$ on 
\[
\bv = \begin{bmatrix} \Id_{r} && &\\ 
& \ddots &&\\
&&\Id_{r} & \\ 
\bu_1 & \dots & \bu_m & 1 \end{bmatrix} \in W(\rR)
\]
defined by 
\begin{equation}\label{E:Actionlog2}
\nu \bv \nu^{-1}= \begin{bmatrix}
\Id_{r} && &\\ 
&\ddots&&\\
&&\Id_{r} & \\ 
\bu_1\mu^{-1} & \dots & \bu_m \mu^{-1}  & 1 \end{bmatrix}.
\end{equation}
For $1\leq i \leq m$, we choose a $\oK(t)$-basis  for $H_n$

Let $\bsx\in\Mat_{rm+1\times 1}(\textbf{X})$ be the $\oK(t)$-basis of $\textbf{X}$ satisfying $\sigma \bsx = \Phi_{\textbf{X}} \bsx$. Letting  $\bsx =  [\bsx_1, \dots, \bsx_m, y]^\tr \in \Mat_{rm+1\times 1}(\textbf{X})$ so that each $\bsx_i\in H_n$, we see that $[\bsx_1,\dots,\bsx_m]^{\tr}$ forms a $\oK(t)$-basis of  $H_n^m$. On the other hand, by the discussion in \S\ref{SS:t-motives}, we see that the entries of $\Psi_{\textbf{X}}^{-1}\bsx$ form an $\FF_q(t)$-basis of $\textbf{X}^{\text{Betti}}$ and hence the entries of $\bu := [\Psi_n^{-1} \bsx_1, \dots, \Psi_n^{-1} \bsx_m]^\tr$ form an $\FF_q(t)$-basis of $(H_n^{\text{Betti}})^m$. We describe the action of $\Gamma_{H_n}(\rR)$ on $\rR\otimes_{\FF_q(t)}(H_n^{\text{Betti}})^m$  as follows (see also \cite[\S4.5]{P08}): For any $\mu \in \Gamma_{H_n}(\rR)$ and any $\bv_\ell \in \Mat_{1\times r}(\rR)$ where $1 \leq \ell \leq m$, the action of $\mu$ on $(\bv_1, \dots, \bv_m) \cdot \bu \in \rR \otimes_{\FF_q(t)} (H_n^m)^{\text{Betti}}$ can be represented by
\[
(\bv_1, \dots, \bv_m) \cdot \bu \mapsto (\bv_1 \mu^{-1}, \dots,  \bv_m \mu^{-1}) \cdot \bu.
\]

\noindent Thus, by \eqref{E:Actionlog2} the action of $\Gamma_{H_n}$ on $(H_n^{\text{Betti}})^m$ is compatible with the action of $\Gamma_{H_n}$ on $W$.  \par

Observe the following facts:
\begin{itemize}
	\item[(i)] $\rL_n$ can be embedded into $\End((H_n)^{\text{Betti}})^m$ and thus, one can regard $(H_n^{\text{Betti}})^m$ to be the $\rL_n$-vector space $\mathcal{V}'$ over $\rL_n$.
	\item[(ii)] Setting $G'$ to be the $\rL_n$-group $\GL(\mathcal{V}')$ and using \cite[Lem.~3.5.6]{CP12} and Theorem~\ref{T:equalitygamma2}, we have that $\Gamma_{H_n}=\Res_{\rL_n/\FF_q(t)}(G')$, where $\Res_{\rL_n/\FF_q(t)}$ is the Weil restriction of scalars over $\FF_q(t)$. 
	\item[(iii)]By Theorem \ref{T:Pap}, $\Gamma_{\textbf{X}}$ is $\FF_q(t)$-smooth.
	\item[(iv)] Set $\bar{V}$ to be the vector group over $\FF_q(t)$ associated to the underlying $\FF_q(t)$-vector space $(\mathcal{V}')^{\oplus m}$. Then $W$ is equal to the scheme-theoretic intersection $\Gamma_{\textbf{X}}\cap \bar{V}$.
\end{itemize}

Since the map $\pi:\Gamma_{Y}\rightarrow \Gamma_{H_n}$ is surjective, by using the above facts, \cite[Ex. A.1.2]{CP12}, and \cite[Prop. A.1.3]{CP12}, we have the following.

\begin{lemma}\label{L:smooth}
	The affine group scheme	$W$ over $\FF_q(t)$ is $\FF_q(t)$-smooth.
\end{lemma}

Since $W$ is $\FF_q(t)$-smooth, we can regard $(H_n^\text{Betti})^m$ as a vector group over $\FF_q(t)$. Moreover we see that $W$ is a $\Gamma_{H_n}$-submodule of $(H_n^\text{Betti})^m$. Thus, by the equivalence of categories $\sT_{H_n} \approx \Rep(\Gamma_{H_n}, \FF_q(t))$, there exists a sub-$t$-motive $U$ of $H_n^m$ such that 
\begin{equation}\label{E:Tmotiveandbetti2}
W \cong U^{\text{Betti}}.
\end{equation}

Recall the fixed $\bA$-basis $\{\bsomega_1,\dots,\bsomega_r\}$ of $\Lambda_{\mathcal{E}_n}$ where $\bsomega_i = [\omega_{i, 1}, \dots, \omega_{i,rn+r-1}]^{\tr}\in \Lie(\mathcal{E}_n)(\CC_{\infty})$ for each $1 \leq i \leq r$ and its maximal $K_n$-linearly independent subset $\{\bsomega_1, \dots, \bsomega_{r/\bs}\}$.

\begin{theorem}\label{T:MaintTtr2}
	For each $1 \leq \ell \leq m$, let $\bsy_{\ell} =[y_{\ell,1}, \dots, y_{\ell,rn+r-1}]^{\tr} \in \Lie(\mathcal{E}_n)(\CC_{\infty})$ be such that $\Exp_{G_{n}}(\bsy_{\ell})\in \mathcal{E}_n(\oK)$. Suppose that the set $\{\bsomega_1, \dots, \bsomega_{r/\bs}, \bsy_1, \dots, \bsy_m\}$ is linearly independent over $K_n$. Then, $\dim \Gamma_{\textbf{\textup{X}}} = r^2/\bs+rm$. In particular, the following statements hold:
	\begin{itemize}
		\item[(i)] Let  $F_{\mathcal{E}_0,\delta}$ be the quasi-periodic function associated to the $\varphi_0$-biderivation $\delta$ which maps $t\mapsto (\tau, 0,\dots,0)$. If $n=0$, then
		\[
		\trdeg_{\oK}\oK \left(\bigcup\limits_{\ell=1}^{m}\bigcup\limits_{i=1}^{r}\bigcup_{j=1}^{r-1}\{\omega_{i,j},F_{\mathcal{E}_0,\delta}(\bsomega_i), y_{\ell, j}, F_{\mathcal{E}_0,\delta}(\bsy_i)\}\right) = r^2/\bs + rm.
		\]
		\item[(ii)] If $n\in \ZZ_{\geq 1}$, then 
		
		\[
		\trdeg_{\oK}\oK \left(\bigcup\limits_{\ell=1}^{m}\bigcup\limits_{i=1}^{r}\bigcup_{j=0}^{r-1}\{\omega_{i,rn+j}, y_{\ell, rn+j}\}\right) = r^2/\bs + rm.
		\]
	\end{itemize}
	
\end{theorem}

\begin{proof} We adapt the methods used in the proof of  \cite[Thm. 5.1.5]{CP12} and \cite[Lem.~1.2]{Har11}. By Theorem \ref{T:TannakianMain}, Theorem~\ref{T:equalitygamma2} and \eqref{E:SESlog2}, we have
 \begin{equation*}\label{E:TrdegLogGamma}
 \dim \Gamma_{\textbf{X}} = \trdeg_{\oK} \oK(\Psi_{\textbf{X}}(\theta)) \leq \frac{r^2}{\bs} + rm.
 \end{equation*}
Thus, we need to prove that $\dim W = rm$. 
By \eqref{E:Tmotiveandbetti2} it suffices to show that $U^{\text{Betti}} \cong (H_n^m)^{\text{Betti}}$.  

We first show that the extension $\textbf{X}/U$ is trivial in $\Ext_\sT^1(\mathbf{1}_{\sP}, H_n/U)$. Since $W \cong U^{\text{Betti}}$, we see that $\Gamma_{\textbf{X}}$ acts on $\textbf{X}^{\text{Betti}}/U^{\text{Betti}}$ through $\Gamma_{\textbf{X}}/W \cong  \Gamma_{H_{\phi}}$ via \eqref{E:SESlog2}. Note that  $\Gamma_{H_n} =\Gamma_{H_{\phi}}$ by Theorem \ref{T:equalitygamma2}. Hence, it follows that $\textbf{X}^{\text{Betti}}/U^{\text{Betti}}$ is an extension of $\FF_q(t)$ by $(H_n^m)^{\text{Betti}}/ U^{\text{Betti}}$ in $\Rep(\Gamma_{H_{\phi}}, \FF_q(t))$. Since every object in $\Rep(\Gamma_{H_{\phi}}, \FF_q(t))$ is completely reducible by \cite[Cor.~3.5.7]{CP12}, by the equivalence  $\sT_{H_{\phi}} \approx \Rep(\Gamma_{H_{\phi}}, \FF_q(t))$ of the categories, we get the required conclusion.

We prove $U^{\text{Betti}} \cong (H_n^m)^{\text{Betti}}$. Suppose on the contrary that $U^{\text{Betti}} \subsetneq (H_n^m)^{\text{Betti}}$. Since $H_n$ is simple by Proposition~\ref{P:simpletensor2} and $U$ is a sub-$t$-motive of $H_n^m$, there exists a non-trivial morphism $\bp~\in~\Hom_\sT(H_n^m, H_n)$ so that $U \subseteq \Ker \bp$. Moreover, the morphism $\bp$ factors through the map $H_n^m/U \rightarrow H_n^m/\Ker \bp $. Since $\bp \in \Hom_\sT(H_n^m, H_n)$, there exist $e_{1}, \dots, e_m\in\rL_n$ not all zero such that $\bp\left(\mathfrak{n}_1, \dots, \mathfrak{n}_m\right)= e_{1}(\mathfrak{n}_1)+\dots + e_m(\mathfrak{n}_m)$ where $\mathfrak{n}_1,\dots,\mathfrak{n}_m\in H_n$. Then, we deduce that the pushout $\bp_{*} \textbf{X} := e_{1*} Y_{1} +\dots+ e_{m*} Y_{m}$ is a quotient of $\textbf{X}/U$. By using the claim above, it follows that $\bp_{*}\textbf{X}$ is trivial in $\Ext_\sT^1(\mathbf{1}_{\sP}, H_n)$, which is a contradiction by Theorem~\ref{T:Trivial2}. Thus, $\dim W = rm$ which proves that $\dim \Gamma_{\textbf{X}} = r^2/\bs+rm$. Thus, by Theorem \ref{T:TannakianMain} and  \eqref{E:matrix22}, we obtain the first part of the theorem. On the other hand,  note that  
	\begin{equation*}
\oK(\Psi_m^{\tens}(\theta))= \oK(\tilde{\pi}^{m}\Upsilon(\theta)) = \oK\left(\bigcup_{i =1}^r\bigcup_{j=0}^{r-1}\widetilde{\pi}^{m} F_{G_0,\delta_j}(\lambda_i)\right)
	\end{equation*}
Then, by the construction of $\Psi_{\textbf{X}}$ and \eqref{E:gtheta1} we have
 \begin{equation}\label{E:thefield}
\oK(\Psi_{\textbf{X}}(\theta)) =  \oK\bigg(\bigcup\limits_{\ell=1}^{m}\bigcup\limits_{i=1}^{r}\bigcup_{j=1}^{r-1}\{\widetilde{\pi}\lambda_i, \widetilde{\pi}F_{G_0,\delta_j}(\lambda_i), y_{\ell, j}, F_{\mathcal{E}_0,\delta}(\bsy_i)\}\bigg)
 \end{equation}
for $n=0$, and 
 \begin{equation}\label{E:thefield2}
 \oK(\Psi_{\textbf{X}}(\theta)) =  \oK\bigg(\bigcup\limits_{\ell=1}^{m} \bigcup\limits_{i=1}^{r}\bigcup_{j=1}^{r-1}\{\widetilde{\pi}^{n+1}\lambda_i, \widetilde{\pi}^{n+1}F_{G_0,\delta_j}(\lambda_i),y_{\ell, rn}, y_{\ell,  rn+1}, \dots, y_{\ell, rn+r-1}\}\bigg)
 \end{equation}
otherwise. Since $\dim \Gamma_{\textbf{X}} = r^2/\bs+rm$, the second assertion now follows from Theorem \ref{T:TannakianMain}, \eqref{E:matrix22} and \eqref{E:thefield}.
\end{proof}

We are now ready to prove Theorem \ref{T:1} and Corollary \ref{C:Int0}.

\begin{proof}[Proof of Theorem~\ref{T:1}]
	Let $\{\bseta_1, \dots, \bseta_{j}\}\subseteq\{\bsomega_{1},\dots, \bsomega_{r}, \bsy_{1}, \dots, \bsy_{m}\}$ be a maximal $K_n$-linearly independent set containing $\{\bsy_1, \dots, \bsy_m\}$, where $r/\bs \leq j \leq r/\bs+m$. Suppose that $\bseta_{u} =[\eta_{u, 1}, \dots, \eta_{u, rn+r-1}]^\tr$ for $1 \leq u \leq j$. For $n=0$, we obtain
	\[
	\oK \left(\bigcup\limits_{\ell=1}^{m}\bigcup\limits_{i=1}^{r}\bigcup_{k=1}^{r-1}\{\omega_{i,j},F_{\mathcal{E}_0,\delta}(\bsomega_i), y_{\ell, k}, F_{\mathcal{E}_0,\delta}(\bsy_\ell)\}\right) =  \oK \bigg(\bigcup_{u=1}^{j}\{F_{\mathcal{E}_0,\delta}(\bseta_u), \eta_{u, 1}, \dots, \eta_{u, r-1}\}\bigg),
	\]
 and for $n\geq 1$, we have
	\[
	\oK \bigg(\bigcup_{i=1}^r\bigcup_{k=0}^{r-1}\bigcup_{\ell=1}^{m} \{\omega_{i,rn+k}, y_{\ell, rn+k}\}\bigg) =  \oK \bigg(\bigcup_{u=1}^{j}\{\eta_{u, rn}, \eta_{u, rn+1}, \dots, \eta_{u, rn+r-1}\}\bigg)
	\]
	where the equality follows from the structure of $\rd\rP$ for any $\rP\in \End(\mathcal{E}_n)$ given in \eqref{E:dEnd2} and \eqref{E:W2}. By Theorem~\ref{T:MaintTtr2}, the set  $\bigcup_{u=1}^{j}\{F_{\mathcal{E}_0,\delta}(\bseta_u), \eta_{u, 1}, \dots, \eta_{u, r-1}\}$ when $n=0$ and the set $\bigcup_{u=1}^{j}\{\eta_{u, rn}, \eta_{u, rn+1}, \dots, \eta_{u, rn+r-1}\}$  for $n\geq 1$, is algebraically independent over $\oK$ and hence, the algebraic independence of \eqref{E:T1i} and \eqref{E:T1ii} follow. Finally, the transcendence of $\mathcal{Y}_n$ for $n \in \ZZ_{\geq 0}$ over $\oK$ follows immediately.
\end{proof}

\begin{proof} [Proof of Corollary ~\ref{C:Int0}]\label{C:ProofofCorMain1}
   Let $\bg_{\bsy_1}, \dots, \bg_{\bsy_m}$ be as defined in \eqref{E:gformathcal{E}}. Note that by \eqref{E:UpsilonNV}, we have $(V^{-1})^\tr \in \Mat_r(\oK)$. Then, by \cite[Theorem~C part (c)]{NPapanikolas21}, for $n \geq 0$, we have 
   \[
   \oK\left( \bigcup\limits_{\ell=1}^{m} \{F_{\mathcal{E}_n, \delta_1}(\bsy_\ell), \dots, F_{\mathcal{E}_n, \delta_r}(\bsy_\ell)\}\right) =\oK\left(\bg_{\bsy_1}|_{t=\theta}, \dots, \bg_{\bsy_\ell}|_{t=\theta}\right).
   \]
    Moreover, by \eqref{E:gtheta1}, if $n=0$, 
   \[
   \oK\left(\bg_{\bsy_1}|_{t=\theta}, \dots, \bg_{\bsy_\ell}|_{t=\theta}\right) = \oK\left( \bigcup\limits_{\ell=1}^{m} \{F_{\mathcal{E}_0,\delta}(\bsy_{\ell}), y_{\ell, 2}, \dots, y_{\ell, r-1}\}\right),
   \]
   where $F_{\mathcal{E}_0,\delta}$ is the quasi-periodic function defined as in 
   Theorem~\ref{T:MaintTtr2}, and if $n \in \ZZ_{\geq 1}$, we have
   \[
    \oK\left(\bg_{\bsy_1}|_{t=\theta}, \dots, \bg_{\bsy_\ell}|_{t=\theta}\right) = \oK\left( \bigcup\limits_{\ell=1}^{m} \{y_{\ell, rn}, y_{\ell, rn+1}, \dots, y_{\ell, rn+r-1}\}  \right).
   \]
  Then, the desired results follow from Theorem~\ref{T:1}. 
\end{proof}

\section{The proof of the main result for \texorpdfstring{$G_n$}{Gn}}\label{S:MainResultsGn}
Our aim in \S\ref{SS:gahafortens} is to provide necessary ideas by constructing certain objects which can be easily derived from similar calculations as was done in \S\ref{S:MainResultsEn}. Whereas proofs of many results follow from similar arguments used in \S\ref{S:MainResultsEn}, we explicitly state whenever we use a different argument for our setting in the present section. 

\subsection{Construction of \texorpdfstring{$\bg_{\bsalpha}$}{ga} and \texorpdfstring{$\bh_{\bsalpha}$}{ha} for \texorpdfstring{$G_n$}{Gn}}\label{SS:gahafortens}
Let $n$ be a non-negative integer. We define elements $\bsy =[y_1, \dots, y_{rn+1}]^\tr \in \Lie(G_n)(\CC_{\infty})$ and $\bsalpha=[\alpha_1, \dots, \alpha_{rn+1}]^{\tr}\in G_n(\oK)$ so that $\Exp_{G_{n}}(\bsy)=\bsalpha$. Let $\cG_{\bsy}(t)=[g_1, \dots, g_{rn+1}]^{\tr}\in \TT^{rn+1}$ be the  Anderson generating function of $G_n$ at $\bsy$. 
Recall the matrices given in \eqref{E:matrices} and \eqref{E:matricesA}. Using Lemma \ref{L:AGFNP21}(ii), we obtain
\begin{equation}\label{E:AGFQuasi}
(t\Id_{rn+1}-\rd_{\phi_{n}}(t))\cG_{\bsy}(t) + \bsalpha = E\cG_{\bsy}^{(1)}(t)
\end{equation}
where $E\in \Mat_{rn+1}(K)$ is defined as in \eqref{E:matricesA}. Our next aim is to define the vectors $\bg^{\tens}_{\bsy}$ and $\bh^{\tens}_{\bsalpha}$ to be able to construct a certain $t$-motive. 

Set 
\begin{equation}\label{E:gforG}
\bg_{\bsy}^{\tens}:=-[g_1^{(1)},\dots,g_r^{(1)}]V\in \Mat_{1 \times r}(\TT),
\end{equation}
where  $V\in \GL_{r}(\oK)$ is defined in \eqref{E:UpsilonNV}.
Since the last $r$ rows of $N\in \Mat_{rn+1}(\mathbb{F}_q)$ given as in \eqref{E:matrices} contain only zeros, by using \eqref{E:AGFQuasi} we see that
\[
\bg^{\tens}_{\bsy}= \begin{bmatrix}
-(t-\theta)g_{rn+1} - \alpha_{rn+1} \\
-\sum_{i=2}^ra_i^{(-1)}\left((t-\theta)g_{r(n-1)+i}+\alpha_{r(n-1)+i}\right)\\
-\sum_{i=2}^{r-1} a_{i+1}^{(-2)}\left((t-\theta)g_{r(n-1)+i}+\alpha_{r(n-1)+i}\right)\\
\vdots \\
-\sum_{i=2}^{3} a_{i+(r-3)}^{(-r+2)}\left((t-\theta)g_{r(n-1)+i}+\alpha_{r(n-1)+i}\right)\\
-a_r^{(-r+1)}\left((t-\theta)g_{r(n-1)+2}+\alpha_{r(n-1)+2}\right)
\end{bmatrix}^{\tr} \in \Mat_{1 \times r}(\TT).
\]
If we set
\[
\bh_{\bsalpha}^{\tens}:= \begin{bmatrix}
\sum_{j=0}^{n} (t-\theta)^{n-j}\alpha_{rj+1} \\
\sum_{i=2}^r\sum_{j=1}^{n}  a_i^{(-1)}(t-\theta)^{n-j}\alpha_{r(j-1)+i}\\
\sum_{i=2}^{r-1}\sum_{j=1}^{n} a_{i+1}^{(-2)}(t-\theta)^{n-j}\alpha_{r(j-1)+i} \\
\vdots \\
\sum_{i=2}^3\sum_{j=1}^{n} a_{i+(r-3)}^{(-r+2)}(t-\theta)^{n-j}\alpha_{r(j-1)+i}\\
\sum_{j=1}^{n} a_r^{(-r+1)}(t-\theta)^{n-j}\alpha_{r(j-1)+2}
\end{bmatrix}^{\tr}\in \Mat_{1 \times r}(\oK[t])
\]
then by using the definition of $\bg_{\bsy}^{\tens}$ and \eqref{E:AGFQuasi}, we have  
\begin{equation}\label{E:RATg}
(\bg_{\bsy}^{\tens})^{(-1)}\Phi_n^{\tens}-\bg^{\tens}_{\bsy} = \bh^{\tens}_{\bsalpha}.
\end{equation}
It follows easily from Lemma \ref{L:AGFNP21}(i) (see also \cite[Prop. 4.2.7]{NPapanikolas21}) that
\begin{equation}\label{E:gtheta}
\bg_{\bsy}^{\tens}(\theta)= \begin{bmatrix}
y_{rn+1} - \alpha_{rn+1} \\
\sum_{i=2}^ra_i^{(-1)}\left(y_{r(n-1)+i}-\alpha_{r(n-1)+i}\right)\\
\sum_{i=2}^{r-1}a_{i+1}^{(-2)}\left(y_{r(n-1)+i}-\alpha_{r(n-1)+i}\right) \\
\sum_{i=2}^{r-2}a_{i+2}^{(-3)}\left(y_{r(n-1)+i}-\alpha_{r(n-1)+i}\right)\\
\vdots \\
\sum_{i=2}^3a_{i+(r-3)}^{(-r+2)}\left(y_{r(n-1)+i}-\alpha_{r(n-1)+i}\right)\\
a_r^{(-r+1)}(y_{r(n-1)+2}-\alpha_{r(n-1)+2})
\end{bmatrix}^{\tr}\ \in \Mat_{1 \times r}(\CC_{\infty}).
\end{equation}
Consider the $\bA$-module $\Lambda_{G_n}$ comprising the periods of $G_n$. Since
$\sH_{G_n}$ is rigid analytically trivial, by \cite[Lem. 2.4.1, Thm. 2.5.32]{HartlJuschka16},  $\Lambda_{G_n}$ is a free $\bA$-module of rank $r$. For each $1 \leq i \leq r$, set $\boldsymbol{\lambda}_i := [\lambda_{i, 1}, \dots, \lambda_{i,rn+1}]^{\tr}\in \Lie(G_n)(\CC_{\infty})$ and let $\{\boldsymbol{\lambda}_1, \dots, \boldsymbol{\lambda}_r\}$ be a fixed $\bA$-basis of $\Lambda_{G_{n}}$. For each $1\leq k \leq r$, recalling the Anderson generating function $\cG_{\boldsymbol{\lambda}_k}(t)=[g^{\tens}_{k,1},\dots,g^{\tens}_{k,rn+1}]^{\tr}$ of $G_n$ at $\bslambda_k$, we define the matrix
\[
\mathfrak{B}_n^{\tens} := \begin{bmatrix}
(g_{1,1}^{\tens})^{(1)} & \dots &(g_{1,r}^{\tens})^{(1)}\\
\vdots &  & \vdots \\
(g_{r,1}^{\tens})^{(1)} &\dots & (g_{r,r}^{\tens})^{(1)}
\end{bmatrix} V\in \Mat_{r}(\TT)
\]
where $V\in \GL_r(\oK)$ is as defined in \eqref{E:UpsilonNV}. Moreover, by \cite[Prop. 4.3.10]{NPapanikolas21}, it follows that $\mathfrak{B}_n^{\tens}\in \GL_r(\TT)$. Setting $\Pi_n^{\tens} := (\mathfrak{B}_n^{\tens})^{-1}$, we further observe, by \cite[Prop. 4.4.14]{NPapanikolas21}, that $(\Pi_n^{\tens})^{(-1)} = \Phi^{\tens}_{n}\Pi_n^{\tens}$. Thus, $\Pi_n^{\tens}$ is a rigid analytic trivialization of $\sH_n^{\tens}$ and hence, by the discussion in \cite[Sec. 4.1.6]{P08}, there exists a  matrix $D^{\tens}=(D^{\tens}_{ij})\in \GL_r(\mathbb{F}_q(t))$  regular at $t=\theta$ so that 
\[
\Pi_n^{\tens} (D^{\tens})^{-1} = \Psi_{n}^{\tens}.
\]
By a similar calculation as in \eqref{E:gtheta}, the $i$-th row of $(\Psi_n^{\tens})^{-1}(\theta)=D^{\tens}(\theta)\mathfrak{B}_n^{\tens}(\theta)$ is 
\begin{equation}\label{E:matrix2}
\begin{bmatrix}
D_{i1}^{\tens}(\theta)\lambda_{1,rn+1}+\dots+D^{\tens}_{ir}(\theta)\lambda_{r,rn+1} \\
D^{\tens}_{i1}(\theta)\sum_{j=2}^{r}a_{j}^{(-1)}\lambda_{1,r(n-1)+j}+\dots+ D^{\tens}_{ir}(\theta)\sum_{j=2}^{r}a_{j}^{(-1)}\lambda_{r,r(n-1)+j}   \\
D^{\tens}_{i1}(\theta)\sum_{j=2}^{r-1}a_{j+1}^{(-2)}\lambda_{1,r(n-1)+j}+\dots+ D^{\tens}_{ir}(\theta)\sum_{j=2}^{r-1}a_{j+1}^{(-2)}\lambda_{r,r(n-1)+j}   \\
\vdots\\
D^{\tens}_{i1}(\theta)\sum_{j=2}^{3}a_{j+(r-3)}^{(-r+2)}\lambda_{1,r(n-1)+j}+\dots+ D^{\tens}_{ir}(\theta)\sum_{j=2}^{3}a_{j+(r-3)}^{(-r+2)}\lambda_{r,r(n-1)+j}  \\
D^{\tens}_{i1}(\theta)a_{r}^{(-r+1)}\lambda_{1,r(n-1)+2}+\dots+ D^{\tens}_{ir}(\theta)a_{r}^{(-r+1)}\lambda_{r,r(n-1)+2}  
\end{bmatrix}^{\tr}.
\end{equation}
\par 

We now consider the $\CC_{\infty}[\sigma]$-module $\Mat_{1\times rn+1}(\CC_{\infty}[\sigma])$ and equip it with a $\CC_{\infty}[t]$-module structure given by 
\[
\mathfrak{c}t \cdot \mathfrak{u} :=\mathfrak{c} \mathfrak{u}\phi_n(t)^{*}, \ \ \mathfrak{c}\in \CC_{\infty}, \ \ \mathfrak{u}\in \Mat_{1\times rn+1}(\CC_{\infty}[\sigma]).
\]
Recall the element $\ell_i^{\tens}$ for each $1\leq i \leq r$ introduced in \S\ref{S:tmoduleGn} and define the map
\[
\tilde{\iota}: \Mat_{1\times r}(\CC_{\infty}[t]) \rightarrow  \Mat_{1\times rn+1}(\CC_{\infty}[\sigma])
\]
by setting for $[\mathfrak{g}_1, \dots, \mathfrak{g}_r] \in \Mat_{1\times r}(\oK[t])$,
\[
\tilde{\iota}([\mathfrak{g}_1, \dots, \mathfrak{g}_r]) := \mathfrak{g}_1 \cdot \ell_1^{\tens} + \dots+\mathfrak{g}_r \cdot \ell_r^{\tens}.
\]
We consider the $\CC_{\infty}$-linear map 
\[\tilde{\delta_0}:\Mat_{1\times rn+1}(\CC_{\infty}[\sigma])\to \Mat_{rn+1\times 1}(\CC_{\infty})
\]
given by
\[
\tilde{\delta_0}\left(\sum_{i\geq 0}B_i\sigma^i\right) := B_0^{\tr},\ \ B_i\in \Mat_{1\times rn+1}(\CC_{\infty}), \ \ B_i=0 \text{ for } i\gg0 .
\]
By using the ideas of the proof of \cite[Prop.~5.2]{G20} we have
\begin{equation}\label{E:deltaiota}
\tilde{\delta_0}\circ\tilde{\iota}([\mathfrak{g}_1, \dots, \mathfrak{g}_r]) = [\pd_t^n(\mathfrak{g}_r),\pd_t^{n-1}(\mathfrak{g}_1), \dots, \pd_t^{n-1}(\mathfrak{g}_r),\dots, \dots, \pd_t^1(\mathfrak{g}_1), \dots,  \pd_t^1(\mathfrak{g}_r), \mathfrak{g}_1, \dots, \mathfrak{g}_r ]^{\tr} |_{t=\theta}.
\end{equation}
There exists a unique extension of $\tilde{\delta_0}\circ\tilde{\iota}$ given by \[
\widehat{\tilde{\delta_0}\circ\tilde{\iota}}: \Mat_{1\times r}(\TT_{\theta}) \rightarrow \Mat_{rn+1\times 1}(\CC_{\infty}),
\]
where $\TT_{\theta}$. Recall the matrix $\rB^{\tens}$ satisfying $\Phi_n^{\tens} = (\rB^{\tens})^{(-1)}\rA^{\tens} (\rB^{\tens})^{-1}$ given in \S4 (see \eqref{E:BchangeofB}). Set $\widetilde{\bg^{\tens}_{\bsy}} := \bg^{\tens}_{\bsy}\rB^{\tens}$ and $\widetilde{\bh^{\tens}_{\bsalpha}} := \bh^{\tens}_{\bsalpha}\rB^{\tens}$. By using Lemma~\ref{L:AGFNP21}(i), \eqref{E:RATg} and \eqref{E:deltaiota}, one checks directly that
\begin{equation}\label{E:epsiotay}
\widehat{\tilde{\delta_0}\circ\tilde{\iota}}(\widetilde{\bg^{\tens}_{\bsy}}+\widetilde{\bh^{\tens}_{\bsalpha}}) = \bsy.
\end{equation}
Similarly, setting $\widetilde{\Gamma^{\tens}}:=\mathfrak{B}_n^{\tens} \rB^{\tens}\in \GL_{r}(\TT)$ and $\widetilde{\Gamma_i^{\tens}}$ to be the $i$-th row of $\widetilde{\Gamma^{\tens}}$, we have 
\begin{equation}\label{E:epsiotaomega}
\widehat{\tilde{\delta_0}\circ\tilde{\iota}}(\widetilde{\Gamma_i^{\tens}}) = -\bslambda_i.
\end{equation}
\par

\subsection{Logarithms of \texorpdfstring{$G_n$}{Gn} and the proof of Theorem \ref{T:2}}\label{SS:ProofThmiii}
	
	We now define the pre-$t$-motive $Y^{\tens}_{\bsalpha}$  so that it is of dimension $r+1$ over $\oK(t)$ and for a chosen $\oK(t)$-basis $\bsy^{\tens}_{\alpha}$ of $Y^{\tens}_{\bsalpha}$, we have $\sigma \bsy^{\tens}_{\alpha}=\Phi_{Y^{\tens}_{\bsalpha}}\bsy^{\tens}_{\alpha}$ where 
	\[
	\Phi_{Y^{\tens}_{\bsalpha}} := \begin{bmatrix}
	\Phi^{\tens}_{n} & {\bf{0}}\\
	\bh^{\tens}_{\bsalpha} & 1
	\end{bmatrix}\ \in \Mat_{r+1}(\oK[t]).
	\]
	It is easy to see by using \eqref{E:RATg} that $Y^{\tens}_{\bsalpha}$ has a rigid analytic trivialization $\Psi_{Y^{\tens}_{\bsalpha}}$ given by 
	\[
	\Psi^{\tens}_{Y_{\bsalpha}} = \begin{bmatrix}
	\Psi^{\tens}_{n} & {\bf{0}}\\
	\bg^{\tens}_{\bsalpha}\Psi_{n}^{\tens} & 1
	\end{bmatrix}  \in \GL_{r+1}(\TT).
	\]
	Using the same argument in the proof of Lemma \ref{L:Yalpha2}, one can prove the following lemma. 
	\begin{lemma}\label{L:Yalpha}
		The pre-$t$-motive $Y_{\bsalpha}^{\tens}$ is  a $t$-motive. Moreover, $Y_{\bsalpha}^{\tens}$ represents a class in $\Ext_{\sT}(\mathbf{1}_{\sP}, H_n^{\tens})$.
	\end{lemma}
	
	Let $K_n^{\tens}$ to be the fraction field of $\End(G_n)$. We say that  the elements $\bsz_1,\dots,\bsz_k\in \Lie(G_n)(\CC_{\infty})$ are \textit{linearly independent over $K_n^{\tens}$} if whenever $\rd \rP_1\bsz_1+\dots+\rd \rP_k\bsz_k=0$ for some $\rP_1,\dots,\rP_k\in K_n^{\tens}$, we have $\rP_1=\dots=\rP_k=0$. Recall from \S3.1 that $\bs=[\rL_{\phi}:\mathbb{F}_q(t)]$ where $\phi$ is the Drindel $\bA$-module given as in \eqref{E:DrinfeldDef}.  We further suppose that $\{\bslambda_1, \dots, \bslambda_{r/\bs}\}$ is the maximal $K_n^{\tens}$-linearly independent subset of $\{\bslambda_1, \dots, \bslambda_{r}\}$.
	
	Before introducing our next result, following the ideas in \S\ref{SS:tmotiveconsextn}, we note that one can define a $\rL_n^{\tens}$-vector space structure on $ \Ext_\sT^1(\mathbf{1}_{\sT},H_n^{\tens})$ and for each $e\in \rL_n^{\tens}$ and a class $Y\in  \Ext_\sT^1(\mathbf{1}_{\sT},H_n^{\tens})$, we set $e_{*}Y:=e\cdot Y$. 
	
	Repeating the same argument in the proof of Theorem \ref{T:Trivial2} as well as using \eqref{E:epsiotay} and \eqref{E:epsiotaomega} one can obtain the following result.
	\begin{theorem}\label{T:Trivial}
		For each $1 \leq \ell \leq m$, let $\bsy_{\ell} \in \Lie(G_n)(\CC_{\infty})$ be such that $\Exp_{G_{n}}(\bsy_{\ell})=\bsalpha_{\ell}\in G_n(\oK)$. Suppose that  $\bslambda_1, \dots, \bslambda_{r/\bs}, \bsy_1, \dots, \bsy_m$ are linearly independent over $K_n^{\tens}$. We further set $Y^{\tens}_{\ell}:=Y^{\tens}_{\bsalpha_{\ell}}$. Then, for any choice of endomorphisms $e_1, \dots, e_m \in \rL_n^{\tens}$ not all zero, the equivalence class $ \mathcal{S}:=e_{1*}Y^{\tens}_{1} + \dots + e_{m*}Y^{\tens}_{m}\in \Ext_\sT^1(\mathbf{1}_{\sP},H_n^{\tens})$ is non-trivial.
	\end{theorem}

	We are now ready to construct a $t$-motive $\textbf{X}^{\tens}$ by modifying the construction of $\textbf{X}$ in \S\ref{S:Proof1}. Let $\bsy_{1},\dots,\bsy_m,\bsalpha_1,\dots,\bsalpha_m$ be as in Theorem \ref{T:Trivial} and for $1 \leq \ell \leq m$, set $\bg_{\ell}^{\tens}:=\bg^{\tens}_{\bsy_\ell}$ and $\bh_{\ell}^{\tens}:=\bh^{\tens}_{\bsalpha_{\ell}}$. Thus, the block diagonal matrix $\Phi^{\tens} :=\oplus_{\ell=1}^m \Phi^{\tens}_{Y_\ell}$ represents the multiplication by $\sigma$ on a chosen $\oK(t)$-basis of the $t$-motive $Y^{\tens} :=\oplus_{\ell=1}^m Y^{\tens}_{\ell}$, the direct sum of $t$-motives $Y^{\tens}_{1},\dots,Y^{\tens}_m$. It has a rigid analytic trivialization given by $\Psi^{\tens}:= \oplus_{\ell=1}^m\Psi^{\tens}_{Y_{\ell}}$. \par
	
	We define the $t$-motive $\textbf{X}^{\tens}$ on which the multiplication by $\sigma$ on a $\oK(t)$-basis is given by $\Phi_{\textbf{X}^{\tens}} \in \GL_{rm+1}(\oK(t))$ along with a rigid analytic trivialization $\Psi_{\textbf{X}^{\tens}} \in \GL_{rm+1}(\TT)$ where
	\begin{equation*}\label{E:PhiPsiN2}
	\Phi_{\textbf{X}^{\tens}} := \begin{bmatrix} \Phi^{\tens}_n && & \\ 
	&\ddots &&\\
	& &\Phi^{\tens}_n &\\ 
	\bh^{\tens}_{1} & \dots& \bh^{\tens}_{m} & 1 \end{bmatrix}, \quad \text{and} \quad \Psi_{\textbf{X}^{\tens}} := \begin{bmatrix} \Psi_n^{\tens} & && \\
	& \ddots &&\\
	& &\Psi_n^{\tens} &\\ 
	\bg^{\tens}_{1}\Psi_n^{\tens} & \dots& \bg^{\tens}_{m}\Psi^{\tens}_n & 1 \end{bmatrix}.
	\end{equation*}
	Observe that $\textbf{X}^{\tens}$ is an extension of $\mathbf{1}_{\sP}$ by $(H_n^{\tens})^m$. Moreover, $\textbf{X}^{\tens}$ is a pullback of the surjective map $Y^{\tens} \twoheadrightarrow \mathbf{1}_{\sP}^m$ and the diagonal map $\mathbf{1}_{\sP} \rightarrow \mathbf{1}_{\sP}^m$. Thus, the two $t$-motives $Y^{\tens}$ and $\textbf{X}^{\tens}$ generate the same Tannakian subcategory of $\sT$ and hence, the Galois groups $\Gamma_{Y^{\tens}}$ and $\Gamma_{\textbf{X}^{\tens}}$ are isomorphic. For any $\FF_q(t)$-algebra $\rR$, an element of $\Gamma_{\textbf{X}^{\tens}}(\rR)$ is of the form 
	\[
	\nu = \begin{bmatrix} \mu &&& \\ 
	& \ddots &&\\
	& & \mu &\\ 
	\bv_{1} &\dots &\bv_m & 1\end{bmatrix},
	\]
	where $\mu \in \Gamma_{H_n^{\tens}}(\rR)$ and each $\bv_1, \dots, \bv_m \in \GG_a^{r}$. Since $(H_n^{\tens})^m$ is a sub-$t$-motive of $\textbf{X}^{\tens}$, we have the following short exact sequence of affine group schemes over $\FF_q(t)$,
	\begin{equation}\label{E:SESlog}
	1 \rightarrow W^{\tens} \rightarrow \Gamma_{\textbf{X}^{\tens}} \xrightarrow{\pi} \Gamma_{H_n^{\tens}} \rightarrow 1,
	\end{equation}
	where $\pi^{(\rR)}:\Gamma_{\textbf{X}^{\tens}}(\rR)\rightarrow \Gamma_{H_n^{\tens}}(\rR)$ is the map $\nu \mapsto \mu$ (cf.~\cite[p. 138]{CP12}). It can be checked directly via conjugation that \eqref{E:SESlog} gives an action of any $\mu \in \Gamma_{H_n^{\tens}}(\rR)$ on 
	\[
	\bv = \begin{bmatrix}
	\Id_{r} && &\\ 
	& \ddots &&\\
	&&\Id_{r} & \\ 
	\bu_1 & \dots & \bu_m & 1 
	\end{bmatrix} \in W^{\tens}(\rR)
	\]
	defined by 
	\begin{equation}\label{E:Actionlog}
	\nu \bv \nu^{-1}=\begin{bmatrix}
	\Id_{r} && &\\ 
	&\ddots&&\\
	&&\Id_{r} & \\ 
	\bu_1\mu^{-1} & \dots & \bu_m \mu^{-1}  & 1\end{bmatrix}.
	\end{equation}
Let $\bsx^{\tens} \in \Mat_{rm+1\times 1}(\textbf{X}^{\tens})$  be the $\oK(t)$-basis of $\textbf{X}^{\tens}$ satisfying $\sigma \bsx^{\tens} = \Phi_{\textbf{X}^{\tens}} \bsx^{\tens}$. Letting  $\bsx^{\tens} =  [\bsx_1^{\tens}, \dots, \bsx_m^{\tens}, y^{\tens}]^\tr \in \Mat_{rm+1\times 1}(\textbf{X})$ so that each $\bsx_i^{\tens}\in H_n^{\tens}$, we see that $[\bsx_1^{\tens},\dots,\bsx_m^{\tens}]^{\tr}$ forms a $\oK(t)$-basis of  $(H_n^{\tens})^m$. On the other hand, by the discussion in \S\ref{SS:t-motives}, we see that the entries of $\Psi_{\textbf{X}^{\tens}}^{-1}\bsx^{\tens}$ form an $\FF_q(t)$-basis of $(\textbf{X}^{\tens})^{\text{Betti}}$ and hence the entries of $\bu^{\tens} := [(\Psi_n^{\tens})^{-1} \bsx_1^{\tens}, \dots, (\Psi_n^{\tens})^{-1} \bsx_m^{\tens}]^\tr$ form an $\FF_q(t)$-basis of $((H_n^{\tens})^{\text{Betti}})^m$.

    The action of $\Gamma_{H_n^{\tens}}(\rR)$ on $\rR\otimes_{\FF_q(t)}((H_n^{\tens})^{\text{Betti}})^m$ can be described as follows (see also \cite[\S4.5]{P08}): For any $\mu \in \Gamma_{H_n^{\tens}}(\rR)$ and any $\bv_\ell \in \Mat_{1\times r}(\rR)$ where $1 \leq \ell \leq m$, the action of $\mu^{\tens}$ on $(\bv_1, \dots, \bv_m) \cdot \bu^{\tens} \in \rR \otimes_{\FF_q(t)} ((H_n^{\tens})^{\text{Betti}})^m$ can be represented by
	\begin{equation*}
	(\bv_1, \dots, \bv_m) \cdot \bu^{\tens} \mapsto (\bv_1 \mu^{-1}, \dots,  \bv_m \mu^{-1}) \cdot \bu^{\tens}.
	\end{equation*}
	
	\noindent Thus, by \eqref{E:Actionlog} the action of $\Gamma_{H_n^{\tens}}$ on $((H_n^{\tens})^{\text{Betti}})^m$ is compatible with the action of $\Gamma_{H_n^{\tens}}$ on $W^{\tens}$.  \par

	The following statements can be checked easily:
	\begin{itemize}
		\item[(i)] For each $n\in \ZZ_{\geq 0}$, recall $\rL_{n}^{\tens}=\End_{\sT}(H_n^{\tens})$ and note that $\rL_{n}^{\tens}$ can be embedded into $\End((H_n^{\tens})^{\text{Betti}})$. Hence, we can regard $(H_n^{\tens})^{\text{Betti}}$ to be a $\rL_{n}^{\tens}$-vector space and denote it by $\mathcal{V}''$.
		\item[(ii)] Let $G''$ be the $\rL_{n}^{\tens}$-group $\GL(\mathcal{V}'')$. By \cite[Cor.~3.5.6]{CP12} and Theorem~\ref{T:equalitygamma}, we have $\Gamma_{H_n^{\tens}}=\Res_{\rL_{n}^{\tens}/\FF_q(t)}(G'')$, where $\Res_{\rL_{n}^{\tens}/\FF_q(t)}$ is the Weil restriction of scalars over $\FF_q(t)$.
		\item[(iii)] By Theorem~\ref{T:Pap}, $\Gamma_{\textbf{X}^{\tens}}$ is $\FF_q(t)$-smooth.
		\item[(iv)] Set $\bar{V}'$ to be the vector group over $\FF_q(t)$ associated to the underlying $\FF_q(t)$-vector space $(\mathcal{V}'')^{\oplus m}$. Then $W^{\tens}$ is equal to the scheme-theoretic intersection $\Gamma_{\textbf{X}^{\tens}}\cap \bar{V}'$.
	\end{itemize}
	
	Since the map $\pi:\Gamma_{\textbf{X}^{\tens}}\rightarrow \Gamma_{H_n^{\tens}}$ is surjective, using the aforementioned facts, \cite[Ex. A.1.2]{CP12}, and \cite[Prop. A.1.3]{CP12}, we conclude the following.

	\begin{lemma}\label{L:smoothh} The affine group scheme $W^{\tens}$ over $\FF_q(t)$ is $\FF_q(t)$-smooth.
	\end{lemma}

	Since $W^{\tens}$ is $\FF_q(t)$-smooth, we can regard $((H_n^{\tens})^\text{Betti})^m$ as a vector group over $\FF_q(t)$. Moreover, $W^{\tens}$ is a $\Gamma_{H_n^{\tens}}$-submodule of $((H_n^{\tens})^\text{Betti})^m$. Thus, by the equivalence of categories $\sT_{H_n^{\tens}} \approx \Rep(\Gamma_{H_n^{\tens}}, \FF_q(t))$, there exists a sub-$t$-motive $U^{\tens}$ of $(H_n^{\tens})^m$ such that 
	\[
	W^{\tens} \cong (U^{\tens})^{\text{Betti}}.
	\]

 Recall our fixed $\bA$-basis $\{\bslambda_1,\dots,\bslambda_r\}$ of $\Lambda_{G_n}$ where $\bslambda_i = [\lambda_{i, 1}, \dots, \lambda_{i,rn+1}]^{\tr}\in \Lie(G_n)(\CC_{\infty})$ for each $1 \leq i \leq r$ and its maximal $K_n$-linearly independent subset $\{\bslambda_1, \dots, \bslambda_{r/\bs}\}$.
 
	Now one can see that we have achieved all the tools to obtain the following result. In particular, using the same idea as in the proof of Theorem \ref{T:MaintTtr2} as well as applying Proposition \ref{P:simpletensor}, Theorem \ref{T:equalitygamma}, \eqref{E:matrix2}, and Theorem \ref{T:Trivial} to the above setting, we obtain the following.

\begin{theorem}\label{T:MaintTtr}
Let $n \in \ZZ_{\geq 1}$.	For each $1 \leq \ell \leq m$, let $\bsy_{\ell} =[y_{\ell,1}, \dots, y_{\ell,rn+1}]^{\tr} \in \Lie(G_n)(\CC_{\infty})$ be such that $\Exp_{G_{n}}(\bsy_{\ell})\in G_n(\oK)$. Suppose that the set $\{\bslambda_1, \dots, \bslambda_{r/\bs}, \bsy_1, \dots, \bsy_m\}$ is linearly independent over $K_n^{\tens}$. Let $\textbf{\textup{X}}^{\tens}$ be the $t$-motive defined as in \eqref{E:PhiPsiN}. Then, $\dim \Gamma_{\textbf{\textup{X}}^{\tens}} = r^2/\bs+rm$. In particular, we have
	\[
\trdeg_{\oK}\oK\bigg(\bigcup\limits_{\ell=1}^{m}\bigcup\limits_{i=1}^{r/\bs}\bigcup_{j=2}^{r+1}\{\lambda_{i,r(n-1)+j}, y_{\ell, r(n-1)+j}\}\bigg) = r^2/\bs+rm.
	\]
\end{theorem}

\begin{proof}[Proof of Theorem \ref{T:2}]
The proof follows via a simple application of Theorem~\ref{T:MaintTtr} by repeating the same arguments as in the proof of Theorem~\ref{T:1} for the $t$-module $G_n$. For completeness, we provide a proof. 
Let $\{\boldsymbol{\vartheta}_1, \dots, \boldsymbol{\vartheta}_{j}\}\subseteq\{\bslambda_{1},\dots, \bslambda_{r}, \bsy_{1}, \dots, \bsy_{m}\}$ be a maximal $K_n^{\tens}$-linearly independent set containing $\{\bsy_1, \dots, \bsy_m\}$ for some $r/\bs\leq j \leq r/\bs+m$. Set $\boldsymbol{\vartheta}_{u} =[\vartheta_{u, 1}, \dots, \vartheta_{u, rn+1}]^\tr$ for $1 \leq u \leq j$. Then, we obtain
	\[
	\oK \bigg(\bigcup_{i=1}^r\bigcup_{k=2}^{r+1}\bigcup_{\ell=1}^{m} \{\omega_{i,r(n-1)+k}, y_{\ell, r(n-1)+k}\}\bigg) =  \oK \bigg(\bigcup_{u=1}^{j}\{{\vartheta}_{u, r(n-1)+2}, {\vartheta}_{u, r(n-1)+3}, \dots, {\vartheta}_{u, rn}, {\vartheta}_{u, rn+1}\}\bigg)
	\]
	where the equality follows from the structure of $\rd\rP^{\tens}$ for any $\rP^{\tens}\in \End(G_n)$ given in \eqref{E:dEnd1}. By Theorem~\ref{T:MaintTtr}, the set  $\bigcup_{u=1}^{j}\{\vartheta_{u, r(n-1)+2}, \vartheta_{u, r(n-1)+3}, \dots, \vartheta_{u, rn+1}\}$ is algebraically independent over $\oK$ and hence, we obtain the desired algebraic independence of the set $\bigcup_{\ell=1}^{m} \{y_{\ell, r(n-1)+2}, y_{\ell, r(n-1)+3}, \dots, y_{\ell, rn+1}\}$ over $\oK$. Finally, the transcendence of $\widetilde{\mathcal{Y}}_n$ for $n \in \ZZ_{\geq 1}$ over $\oK$ follows immediately.
\end{proof}

\begin{proof}[Proof of Corollary \ref{C:Int1}]
 Let $\bg^{\tens}_{\bsy_1}, \dots, \bg^{\tens}_{\bsy_\ell}$ be as defined in \eqref{E:gforG}. Note that by \eqref{E:UpsilonNV}, we have $V \in \Mat_r(\oK)$. Then, by \cite[Theorem~C part (c)]{NPapanikolas21}, we have 
   \[
   \oK\left(\bigcup\limits_{\ell=1}^{m} \left\{ F_{G, \delta^{\tens}_1}(\bsy_\ell), \dots, F_{G, \delta^{\tens}_r}(\bsy_\ell)\right\}\right) =\oK\left(\bg^{\tens}_{\bsy_1}|_{t=\theta}, \dots, \bg^{\tens}_{\bsy_\ell}|_{t=\theta}\right).
   \]
   Then, as in the proof of Corollary~\ref{C:Int0}, the desired results follow by using \eqref{E:gtheta} and  Theorem~\ref{T:2}. 
    \end{proof}

\end{document}